\documentclass[twoside, 14pt]{article}
\usepackage{amsmath,amssymb,amsthm,mathrsfs,dsfont}
\usepackage[margin=2.0cm]{geometry} 
\usepackage{lipsum}
\usepackage{titlesec,hyperref}
\usepackage{fancyhdr}
\usepackage[numbers,sort&compress]{natbib}
\usepackage{color}
\usepackage[titletoc]{appendix}

\fancypagestyle{plain}
{
\fancyhf{}

}

\titleformat{\subsection}{\it}{\thesubsection.\enspace}{1.5pt}{}
\titleformat{\subsubsection}{\it}{\thesubsubsection.\enspace}{1.5pt}{}

\newtheorem{theo}{Theorem}[section]
\newtheorem{lemm}[theo]{Lemma}

\newtheorem{prop}[theo]{Proposition}
\newtheorem{rema}[theo]{Remark}
\numberwithin{equation}{section}
\def\bel{\begin{equation}\label}

\def\eeq{\end{equation}}

\newcommand\R{{\mathbb R}}
\newcommand\wf{{\widetilde f}}
\newcommand\wg{{\widetilde g}}
\newcommand\h{{\widetilde h}}
\newcommand{\al}{{\alpha}}
\newcommand{\be}{{\beta}}
\newcommand{\na}{{\nabla}}
\newcommand{\pa}{{\partial}}

\newcommand{\beq}{\begin{equation}}
\newcommand{\beno}{\begin{equation*}}
\newcommand{\eeno}{\end{equation*}}

\let\pa=\partial
\let\al=\alpha

\let\d=\delta

\let\lam=\lambda

\let\s=\sigma
\let\f=\frac

\let\tri=\triangle
\let\ep=\epsilon
\def\na{\nabla}
\def\dive{\mathop{\rm div}\nolimits}

\begin{document}
\title{Optimal decay of compressible Navier-Stokes equations with or without potential force \hspace{-4mm}}
\author{Jincheng Gao$^\dag$ \quad Minling Li$^\ddag$ \quad Zheng-an Yao$^\sharp$ \\[10pt]
\small {School of Mathematics, Sun Yat-Sen University,}\\
\small {510275, Guangzhou, P. R. China}\\[5pt]
}

\footnotetext{Email: \it $^\dag$gaojch5@mail.sysu.edu.cn,
\it $^\ddag$limling3@mail2.sysu.edu.cn,
\it $^\sharp$mcsyao@mail.sysu.edu.cn}
\date{}

\maketitle

\begin{abstract}
In this paper, we investigate the optimal decay rate for the higher order spatial derivative of global solution to the compressible Navier-Stokes (CNS) equations with or without potential force in three-dimensional whole space.
First of all, it has been shown in \cite{guo2012} that the $N$-th order spatial derivative of global small solution of the CNS equations without potential force tends to zero with the $L^2-$rate $(1+t)^{-(s+N-1)}$ when the initial perturbation around the constant equilibrium state belongs to $H^N(\mathbb{R}^3)\cap \dot H^{-s}(\mathbb{R}^3)(N \ge 3 \text{~and~} s\in [0, \frac32))$.
Thus, our first result improves this decay rate to $(1+t)^{-(s+N)}$.
Secondly, we establish the optimal decay rate for the global small solution of the CNS equations with potential force as time tends to infinity.
These decay rates for the solution itself and its spatial derivatives are really optimal since the upper bounds of decay rates coincide with the lower ones.

\vspace*{5pt}

\vspace*{5pt}
\end{abstract}

\tableofcontents

\section{Introduction}
In this paper, we are concerned with the optimal decay rate
of global small solution to the Cauchy problem for the compressible Navier-Stokes (CNS) equations with and without external force in three-dimensional whole space.
Thus, our first result is to investigate the optimal decay rate
for the CNS equations without external force as follows:
\beq\label{ns1}
\left\{\begin{array}{lr}
	\rho_t +\dive(\rho u)=0,\quad (t,x)\in \mathbb{R}^{+}\times \mathbb{R}^3,\\
	(\rho u)_t+\dive(\rho u \otimes u)+ \na p-\mu\tri u-(\mu+\lam)\na\dive u = 0,
	\quad (t,x)\in \mathbb{R}^{+}\times \mathbb{R}^3,
\end{array}\right.
\eeq
where $\rho$, $u$ and $p$ represent the unknown density, velocity and
pressure, respectively.
The initial data is given by
\beq\label{initial1}
(\rho,u)|_{t=0}=(\rho_0,u_0)(x),\quad x\in  \mathbb{R}^3.
\eeq
Furthermore, as the space variable tends to infinity, we assume
\beq\label{boundary1}
\lim\limits_{|x|\rightarrow+\infty}(\rho,u)=(\bar \rho,0),
\eeq
where $\bar\rho$ is a positive constant.
The pressure $p(\rho)$ here is assumed to be a smooth function
in a neighborhood of $\bar\rho$ with $p'(\bar\rho)>0$.
The constant viscosity coefficients $\mu$ and $\lambda$ satisfy the physical conditions\beq\label{physical-condition}
\mu>0,~~~~2\mu+3\lambda\geq 0. \eeq
The CNS system \eqref{ns1} is a well-known model which describes
the motion of compressible fluid.
In the following, we will introduce some mathematical results related to the CNS equations, including the local and global-in-time well-posedness, large time behavior and so on.

When the initial data are away from the vacuum, the local existence and
uniqueness of classical solutions have been obtained in \cite{{Serrin-1959},{Nash-1962}}.
If the initial density may vanish in open sets, the well-posedness theory also has been studied in \cite{{Cho-Choe-Kim-2004},{Choe-Kim-2003},{Cho-Kim-2006},{Li-Liang-2014},{Salvi-Straskraba-1993}} when the initial data satisfy some compatibility conditions.
The global classical solution was first obtained by Matsumura and Nishida \cite{Matsumura-Nishida-1980} as initial data is closed to a non-vacuum
equilibrium in some Sobolev space $H^s$.
This celebrated result requires that the solution has small oscillation from a uniform non-vacuum state such that the density is strictly away from vacuum.
For general data, one has to face a tricky problem of the possible appearance of vacuum.
As observed in \cite{Li2019,rozanova2008blow,Xin1998,Xin2013}, the strong (or smooth) solution for the CNS equations will blow up in finite time.
In order to solve this problem, it is important to study some blow-up criteria of strong solutions, refer to \cite{Huang2011blowcriter1,{Sun-Wang-Zhang-2011},
Huang2011blowcriter2,Wen2013}.
It is worth noting that the bounds established by the above papers are dependent on time.
Thus, under the assumption that $\sup_{t\in \mathbb R^+}\|\rho(t, \cdot)\|_{C^\alpha}\le  M$ for some $0<\alpha<1$, He, Huang and Wang \cite{he-huang-wang} proved global stability of large solution and built the decay rate for the global solution as it tends to the
constant equilibrium state.
Later, Gao, Wei and Yao studied the optimal decay rate for this class of
global large solution itself and its derivatives
in a series of articles \cite{{gao-wei-yao-D},{gao-wei-yao-NA},{gao-wei-yao-pre}}.
In the presence of vacuum, Huang, Li and Xin\cite{Huang-Li-Xin-2012} established the global existence and uniqueness of strong solution for the CNS equations in three-dimensional space in the condition that the initial energy is small.
Recently, Li and Xin \cite{Li-Xin-2019-pde} obtained similar results for the dimension two,
in addition, they also studied the large time behavior of the solution for the CNS system with small initial data but allowing large oscillations.
Some other related results with respect to the global well-posedness theory can be found in \cite{Li2020,Wen2017}.

The large time behavior of the solutions to the isentropic compressible Navier-Stokes system has been studied extensively.
The optimal decay rate for strong solution to the CNS system was first derived by Matsumura and Nishida \cite{nishida2}, and later by Ponce \cite{ponce} for the optimal $L^p(p\ge 2)$ decay rate.
With the help of the study of Green function, the optimal $L^p$ ($1\le p\le\infty$) decay rates in $\mathbb R^n(n\ge 2)$ were obtained in \cite{{hoff-zumbrun},{liu-wang}} when the small initial perturbation bounded in $H^s\cap L^1$ with the integer $s\ge[n/2]+3$.
All these decay results mentioned above are restricted to the perturbation framework, that is, if the initial data is a small perturbation of constant equilibrium in $L^1\cap H^3$,
the decay rate of global solution to system \eqref{ns1} in $L^2$-norm is
$$\|\rho(t)-\bar\rho\|_{L^2}+\|u (t)\|_{L^2}\le C(1+t)^{-\frac34}.$$
Furthermore, Gao, Tao and Yao \cite{gao2016} applied the Fourier splitting method,
developed by Schonbek \cite{Schonbek1985}, to establish optimal decay rate for
the higher-order spatial derivative of global small solution.
Specially, they built the decay rate as follows:
\begin{equation}\label{Decay-Gao-Tao-Yao}
\|\nabla^k(\rho-\bar\rho)(t)\|_{H^{N-k}}
+\|\nabla^k u(t)\|_{H^{N-k}}
\leq C_0 (1+t)^{-\f34-\f k2},\quad \text{for}~~ 0\leq k\leq N-1,
\end{equation}
if the initial perturbation belongs to $H^N(\mathbb{R}^3)\cap L^1(\mathbb{R}^3)$.
Obviously, the decay rate for the $N-$th order spatial derivative of global solution
in \eqref{Decay-Gao-Tao-Yao} is still not optimal.
\textbf{Recently, this tricky problem is addressed simultaneously in a series of articles
\cite{{chen2021},{Wang2020},{wu2020}} by using the spectrum analysis of
the linearized part.}
In the perturbation setting, the approach to proving the decay estimate for the solution of the CNS system relies heavily on the analysis of the linearization of the system.
More precisely, most of these decay results were proved by combining the linear optimal decay of spectral analysis with the energy method.
From another point of view, under the assumption that the initial perturbation is bounded in $\dot H^{-s}(s\in [0, \frac32))$, Guo and Wang \cite{guo2012} obtained the optimal decay rate of the solution and its spatial derivatives of system \eqref{ns1} under the $H^N(N\ge 3)-$framework by using pure energy method.
More precisely, they established the following decay estimate
\beq\label{Decay-Guo}
\|\nabla^l (\rho-\bar{\rho})(t)\|_{H^{N-l}}
+\|\nabla^l u(t)\|_{H^{N-l}}\le C_0(1+t)^{-(l+s)}, ~{\rm for~}-s< l \le N-1.
\eeq
This method in \cite{guo2012} mainly combined the energy estimates with the interpolation between negative and positive Sobolev norms, and hence do not use the analysis of the linearized part.
From the decay rate of \eqref{Decay-Guo}, it is easy to see that the decay rate of $N-th$ order derivative of solution $(\rho-\bar{\rho}, u)$ coincides with the lower one.
However, the $N-th$ order spatial derivative of heat equation has the optimal decay rate $(1+t)^{-(N+s)}$ rather than $(1+t)^{-(N-1+s)}$(see Theorem $1.1$ in \cite{guo2012}).
\textbf{\textit{Thus, the first purpose in this paper is to investigate the optimal decay rate for the quantity $\nabla^N (\rho, u)$ as it converges to zero in $L^2-$norm.}}

Now, we state the result of decay rate for the CNS equations that has been established before.

\begin{prop}\label{hs1}(\cite{guo2012})
Assume that $(\rho_0-\bar\rho,u_0)\in H^{N}$ for an integer $N\geq 3$.
Then there exists a constant $\delta_0$ such that if
	\[\|(\rho_0-\bar\rho,u_0)\|_{H^{[\f N2]+2}}\leq \delta_0, \]
	then the problem \eqref{ns1}--\eqref{boundary1} admits a unique global solution $(\rho, u)$
    satisfying that for all $t \ge 0$
	\begin{equation}\label{energy-01}
	\|(\rho-\bar\rho)(t)\|_{H^m}^2
    +\|u(t)\|_{H^m}^2
    +\int_0^t(\|\nabla\rho\|_{H^{m-1}}^2+\|\nabla u\|_{H^m}^2)d\tau
    \leq C (\|\rho_0-\bar\rho \|_{H^m}^2+\|u_0\|_{H^m}^2),
	\end{equation}
    where $[\f N2]+2\leq m\leq N$.
    If further, $(\rho_0-\bar\rho,u_0)\in \dot H^{-s}$ for some
    $s \in [0, \f32)$, then for all $t \ge 0$
		\begin{equation}
		\|\Lambda^{-s}(\rho-\bar\rho)(t)\|_{L^2}^2
        +\|\Lambda^{-s}u (t)\|_{L^2}^2\leq C_0,
		\end{equation}
		and
		\begin{equation}\label{sde1}
		\|\nabla^l(\rho-\bar\rho)(t)\|_{H^{N-l}}^2+\|\nabla^l u(t)\|_{H^{N-l}}^2
        \leq C_0 (1+t)^{-(l+s)}, {~\rm for~}-s<l \le N-1.
		\end{equation}
\end{prop}

Now, we state the first result, concerning the optimal decay rate of $N-th$ order spatial derivative of the solution $(\rho-\bar\rho,u)$ of problem \eqref{ns1}--\eqref{boundary1}.

\begin{theo}\label{hs2}
	Under all the assumptions in Proposition \ref{hs1}, then the global solution
    $(\rho, u)$ of problem \eqref{ns1}--\eqref{boundary1} admits the decay estimate for all $t\ge 0$
	\beq\label{sde3}
		\|\nabla^N(\rho-\bar\rho)(t)\|_{L^2}^2
        +\|\nabla^N u(t)\|_{L^2}^2\leq C_0 (1+t)^{-(s+N)}.
	\eeq
	where $C_0$ is a constant independent of time.
\end{theo}

\begin{rema}
The time decay estimate \eqref {sde3}, together with the decay estimate \eqref{sde1}, yields that
\begin{equation}\label{decay-NS}
\|\nabla^l(\rho-\bar\rho)(t)\|_{H^{N-l}}^2
+\|\nabla^l u(t)\|_{H^{N-l}}^2
\leq C_0 (1+t)^{-(l+s)}, {~\rm for~}-s<l \le N.
\end{equation}
Then, the decay estimate \eqref{decay-NS} provides with optimal decay rate
for the global solution $(\rho-\bar\rho, u)$ itself and
its any spatial derivative converging to zero.
Here the decay rate of global solution of original problem \eqref{ns1}--\eqref{boundary1} converging to equilibrium state is called optimal in the sense that
it coincides with the rate of solution of the linearized system.
\end{rema}
\begin{rema}
Combine the decay rate \eqref{Decay-Gao-Tao-Yao} and
the method as the proof of Theorem \ref{hs2}, we also have
\beno
\|\nabla^N(\rho-\bar\rho)(t)\|_{L^2}+\|\nabla^N u(t)\|_{L^2}
\leq C_0 (1+t)^{-\f34-\f N2}.
\eeno
This decay estimate of the $N$-th order spatial derivative is optimal in the sense that it coincides with the rate of solution of linearized part.
\end{rema}

Next, we consider the optimal decay rate of solution for the CNS equations
with potential force in three-dimensional whole space as follows:
\beq\label{ns3}
\left\{\begin{array}{lr}
	\rho_t +\dive(\rho u)=0,\quad (t,x)\in \mathbb{R}^{+}\times \mathbb{R}^3,\\
	(\rho u)_t+\dive(\rho u \otimes u)+ \na p-\mu\tri u-(\mu+\lam)\na\dive u +\rho\na \phi= 0,\quad (t,x)\in \mathbb{R}^{+}\times \mathbb{R}^3,
\end{array}\right.
\eeq
where $\rho$, $u$ and $p$ represent the density, velocity and pressure, respectively. And $-\na\phi$ is the time independent potential force. The constant viscosity coefficients $\mu$ and $\lambda$ satisfy the physical conditions \eqref{physical-condition}.
The initial data
\beq\label{initial-boundary}
(\rho,u)|_{t=0}=(\rho_0,u_0)(x)\rightarrow (\rho_{\infty},0),\quad |x|\rightarrow 0,
\eeq
where $\rho_{\infty}$ is a positive constant.
We assume that $p(\rho)$ is a smooth function in a neighborhood of $\rho_{\infty}$ with $p'(\rho_{\infty})>0$. The stationary solution $(\rho^*(x),u^*(x))$ for
the CNS equations \eqref{ns3} is given by $(\rho^*(x),0)$ satisfying
\beq\label{p}
\int_{\rho_{\infty}}^{\rho^*(x)}\f{p'(s)}{s}ds+\phi(x)=0.
\eeq
The details of derivation for the stationary solution can be found in \cite{mat1983}.
First, Matsumura and Nishida \cite{mat1983} obatined the global existence of solutions to system \eqref{ns3} near the steady state $(\rho^*(x),0)$ with initial perturbation 
under the $H^3-$framework.
In addition, they also showed that the global solution 
converges to the stationary state as time tends to infinity.
The first work to give explicit decay estimate for solution was 
represented by Deckelnick \cite{Deckelnick1992}.
Specifically, Deckelnick was concerned about the isentropic case and showed that
\beno
\sup_{x\in\mathbb{R}^3}|(\rho(t,x)-\rho^*(x),u(t,x))|\leq C(1+t)^{-\f14}.
\eeno
This was then improved by Shibata and Tanaka for more general external forces in \cite{Shibata2003,Shibata2007} to $(1+t)^{-\f12+\kappa}$ for any small positive constant $\kappa$ when the initial perturbation belongs to $H^3\cap L^{\f65}$.
Later, Duan, Liu, Ukai and Yang \cite{duan2007} investigated the optimal $L^p-L^q$ convergence rates for this system when the initial perturbation is also bounded in $L^p$ with $1\leq p<\f65$. Specifically, they established the decay rate as follows:
\beq\label{highdecayL1}
\|(\rho-\rho^*,u)(t)\|_{L^2}\leq  C(1+t)^{-\frac32(\frac1p-\frac12)},
\quad
\|\na^k(\rho-\rho^*,u)(t)\|_{L^2}
\leq C(1+t)^{-\frac32(\frac1p-\frac12)-\frac{k}{2}}, ~\text{for}~~k=1,2,3.
\eeq
For more result about the decay estimate for the CNS equations with potential force, 
one may refer to \cite{{Ukai2006},{duan-ukai-yang-zhao2007},{Okita2014},{Li2011},
{Wang2017},{Matsumura1992},{Matsumura2001}}.
Obviously, the decay rates of the second and third order spatial derivatives 
in \eqref{highdecayL1} are only the same as the first one.
\textbf{\textit{In this paper, our second target is to investigate the optimal decay rate for
the $k-th$ $(k\geq2)$ order spatial derivative of solution to the 
CNS equations with potential force.}}

Finally, we aim to investigate the lower bounds of decay rates for the solution
itself and its spatial derivatives.
The decay rate is called optimal in the sense that this rate 
coincides with the linearized part.
Thus, the study of the lower decay rate, which is the same as the upper one, can
help us obtain the optimal decay rate of solution.
Along this direction, Schonbek addressed the lower bound of decay rate for solution
of the classical incompressible Navier-Stokes equations \cite{Schonbek1986,Schonbek1991}
(see also MHD equations \cite{Schonbek-Schonbek-Suli}).
Based on so-called Gevrey estimates, Oliver and Titi \cite{Oliver2000} established
the lower and upper bounds of decay rate for the higher order derivatives
of solution to the incompressible Navier-Stokes equations in whole space.
For the case of compressible flow, there are many results of lower bound of
decay rate for the solution itself to the CNS equations and related
models, such as the CNS equations \cite{lizhang2010,Kagei2002},
compressible Navier-Stokes-Poisson equations \cite{limatsumura2010,Zhang2011},
and compressible viscoelastic flows \cite{Hu-Wu-2013}.
However, these lower bounds mentioned above only consider the solution itself
and do not involve the derivative of the solution.
Recently, some scholars are devoted to studying the lower bound
of decay rate for the derivative of solution,
which can be referred to \cite{{chen2021},{wu2020},{gao-lyu-Yao-2019},{gao-wei-yao-pre}}.
\textbf{\textit{Thus, our third target is to establish lower bound of decay rate for
the global solution itself and its spatial derivatives.}}
These lower bounds of decay estimates, which coincide with the upper ones,
show that they are really optimal.

Now, our second result can be stated as follows.

\begin{theo}\label{them3}
	Let $(\rho^*(x),0)$ be the stationary solution of initial value problem \eqref{ns3}--\eqref{initial-boundary}, if $(\rho_0-\rho^*,u_0)\in H^N$ for $N\geq3$, there exists a constant $\delta$ such that the potential function $\phi(x)$ satisfies
	\beq\label{phik}
	\sum_{k=0}^{N+1}\|(1+|x|)^{k}\na^k\phi\|_{L^2\cap L^\infty}\leq \delta,
	\eeq
	and the initial perturbation statisfies
	\beq\label{initial-H2}
	\|(\rho_0-\rho^*,u_0)\|_{H^{N}}\leq \delta.
	\eeq
	Then there exists a unique global solution $(\rho,u)$ of initial value problem \eqref{ns3}--\eqref{initial-boundary} satisfying
\beq\label{energy-thm}
\begin{split}
\|(\rho-\rho^*,u)(t)\|_{H^N}^2
+\int_0^t\big(\|\nabla(\rho-\rho^*)\|_{H^{N-1}}^2+\|\nabla u\|_{H^N}^2\big)ds
\leq C\|(\rho_0-\rho^*,u_0)\|_{H^N}^2 ,\quad t\geq0,
\end{split}
\eeq
where $C$ is a positive constant independent of time $t$.
If further
\beno
\|(\rho_0-\rho^*,u_0)\|_{L^1}<\infty,
\eeno
then there exists constans $\delta_0>0$ and $\bar C_0>0$ such that for any $0<\delta\leq\delta_0$, we have
\beq\label{kdecay}
\|\na^k(\rho-\rho^*)(t)\|_{L^2}+\|\na^k u(t)\|_{L^2}
\leq \bar C_0(1+t)^{-\f34-\f k2},\quad\text{for}~~0\leq k \leq N.
\eeq
\end{theo}

\begin{rema}
The global well-posedness theory of the CNS equations with potential force
in three-dimensional whole space was studied in \cite{duan2007} under the $H^3-$ framework.
Furthermore, they also established the decay estimate \eqref{highdecayL1} if
the initial data belongs to $L^p$ with $1\le p < \frac65$.
Thus, the advantage of the decay rate \eqref{kdecay} in Theorem \ref{them3} is that
the decay rate of the global solution $(\rho-\rho^*,u)$ itself and
its any order spatial derivative is optimal.
\end{rema}

Finally, we have the following result concerning the lower bounds of decay rates for solution and its spatial derivatives of the CNS equations with potential force.
\begin{theo}\label{them4}
Suppose the assumptions of Theorem \ref{them3} hold on. Furthermore, we assume that $\int_{\R^3}(\rho_0-\rho^*)(x) d x$ and $\int_{\R^3}u_0(x) d x$ are at least one nonzero.
Then, the global solution $(\rho,u)$ obtained in Theorem \ref{them4} has the decay rates for any large enough $t$,
\beq
\begin{aligned}\label{kdecaylow}
&{c_0}(1+t)^{-\f34-\f k2}\le \|\na^k(\rho-\rho^*)(t)\|_{L^2}\le {c_1}(1+t)^{-\f34-\f k2};\\
&{c_0}(1+t)^{-\f34-\f k2}\le \|\na^ku(t)\|_{L^2}\le {c_1}(1+t)^{-\f34-\f k2};
\end{aligned}
\eeq
for all $0\leq k\leq N$.
Here $c_0$ and $c_1$ are positive constants independent of time $t$.
\end{theo}

\begin{rema}
	The decay rates showed in \eqref{kdecay} and \eqref{kdecaylow} imply that the $k-th$ $(0\leq k\leq N)$ order spatial derivative of the solution converges to the equilibrium state $(\rho^*, 0)$ at the $L^2-$rate $(1+t)^{-\f34-\f k2}$. In other words, these decay rates of the solution itself and its spatial derivatives obtained in \eqref{kdecay} and \eqref{kdecaylow} are optimal.
\end{rema}

\textbf{Notation:} Throughout this paper, for $1\leq p\leq +\infty$ and $s\in\mathbb{R}$, we simply denote $L^p(\mathbb{R}^3)$ and $H^s(\mathbb{R}^3)$ by $L^p$  and $H^s$, respectively.
And the constant $C$ denotes a general constant which may vary in different estimates.
$\widehat{f}(\xi)=\mathcal F(f(x))$ represents the usual Fourier transform of the function $f(x)$ with respect to $x\in\mathbb{R}^3$. $\mathcal F^{-1}(\widehat{f}(\xi))$ means the inverse Fourier transform of $\widehat{f}(\xi)$ with respect to $\xi\in\mathbb{R}^3$. For the sake of simplicity, we write $\int f dx:=\int _{\mathbb{R}^3} f dx$.

The rest of the paper is organized as follows. In Section \ref{approach}, we introduce the difficulties and our approach to prove the results. In Section \ref{pre}, we recall some important lemmas, which will be used in later analysis. And Section \ref{h-s} is denoted to giving the proof of Theorem \ref{hs2}. Finally, Theorem \ref{them3} and Theorem \ref{them4} are proved in Section \ref{rhox}.

\section{Difficulties and outline of our approach}\label{approach}
The main goal of this section is to explain the main difficulties of
proving Theorems \ref{hs2}, \ref{them3} and \ref{them4} as well as our
strategies for overcoming them.
In order to establish optimal decay estimate for the CNS equations,
the main difficulty comes from the system \eqref{ns1} or \eqref{ns3}
satisfying hyperbolic-parabolic coupling equations, such that
the density only can obtain lower dissipation estimate.

First of all, let us introduce our strategy to prove the Theorem \ref{hs2}.
Indeed, applying the classical energy estimate, it is easy
to establish following estimate:
\begin{equation}\label{energy-N-1}
\begin{split}
\f{d}{dt}\|\nabla^N(n,u)\|_{L^2}^2+\|\nabla^{N+1}u\|_{L^2}^2
\leq \|\nabla (n, u)\|_{H^1}^2 \|\nabla^N(n, u)\|_{L^2}^2
+\text{some~good~terms}.
\end{split}
\end{equation}
In order to control the first term on the right handside of \eqref{energy-N-1},
the idea in \cite{guo2012} is to establish the dissipative for the density $\nabla^N n$,
which gives rise to the cross term $\frac{d}{dt}\int \nabla^{N-1} u\cdot\nabla^{N}n dx$
in energy part.
This is the reason why the decay rate of the $N-th$ order spatial derivative of
solution of the CNS equations can only attain the decay rate as the $(N-1)-th$ one.
In order to settle this problem, our strategy is to apply the time integrability
of the dissipative term of density rather than  absorbing it by the dissipative term.
More precisely, applying the weighted energy method to the estimate
\eqref{energy-N-1} and using decay \eqref{sde1}, it holds true
\begin{equation}\label{energy-N-2}
\begin{split}
&(1+t)^{N+\s+\ep_0}\|\nabla^N (n,u)\|_{L^2}^2
+\int_{0}^{t}(1+\tau)^{N+\s+\ep_0}\|\nabla^{N+1}u\|_{L^2}^2d\tau \\
\leq& \|\nabla^N (n_0,u_0)\|_{L^2}^2
+\int_0^t (1+\tau )^{N-1+\s+\ep_0}\|\nabla^N (n,u)\|_{L^2}^2 d\tau
+\text{some~good~terms}.
\end{split}
\end{equation}
Thus, we need to control the second term on the right handside of \eqref{energy-N-2}.
On the other hand, it is easy to check that
\begin{equation}\label{energy-02}
\begin{split}
\f{d}{dt}\mathcal{E}^{N-1}(t)
+C(\|\nabla^{N}n\|_{L^2}^2+\|\nabla^{N}u\|_{H^1}^2)\leq 0.
\end{split}
\end{equation}
Here $\mathcal{E}^{N-1}(t)$ is equivalent to $\|\nabla^{N-1}(n,u)\|_{H^1}^2$.
The combination of \eqref{energy-02} and decay estimate \eqref{sde1} yields directly
\begin{equation}\label{energy-N-3}
(1+t)^{N-1+\s+\ep_0}\mathcal{E}^{N-1}(t)
+\int_0^t(1+\tau)^{N-1+\s+\ep_0}
\big(\|\na^N n\|_{L^2}^2+\|\na^N u\|_{H^1}^2\big)d\tau
\leq C(1+t)^{\ep_0}.
\end{equation}
Thus, we apply the time integrability of the dissipative term of density in \eqref{energy-N-3}
to control the second term on the right handside of \eqref{energy-N-2}.
Therefore, we can obtain the optimal decay rate for $\nabla^N(n, u)$ as it converges to zero.

Secondly, we will establish the optimal decay rate, including in Theorem \ref{them3},
for the higher order spatial derivative of global solution to the CNS equations
with external potential force.
Due to the influence of potential force, the equilibrium state of global solution will
depend on the spatial variable. This will create some fundamental difficulties
as we establish the energy estimates, see Lemmas \ref{enn-1}, \ref{enn} and \ref{ennjc}.
Similar to the decay estimate \eqref{highdecayL1}(cf.\cite{duan2007}),
one can combine the energy estimate and the decay rate of linearized system
to obtain the following decay estimates:
\begin{equation}\label{basic-decay-d}
\|\nabla^k (\rho-\rho^*)(t)\|_{H^{N-k}}+\|\nabla^k u(t)\|_{H^{N-k}}
\le C(1+t)^{-(\frac34+\frac{k}{2})},\quad k=0,1,
\end{equation}
if the initial data $(\rho_0-\rho^*, u_0)$ belongs to $H^N \cap L^1$.
To prove that this decay is true for $k\in\{2,\cdots,N-1\}$, we are going to do it
by mathematical induction.
Thus, assume that decay \eqref{basic-decay-d} holds on for $k=l\in\{1,\cdots,N-2\}$,
our target is to prove the validity of \eqref{basic-decay-d} as $k=l+1$.
This logical relationship can be guaranteed by using the classical Fourier splitting
method(cf\cite{gao2016}). However, similar to the method of the proof of Theorem \ref{hs2},
we guarantee this logical relationship by using the time weighed method,
see Lemma \ref{N-1decay} more specifically.
Since the presence of potential force term $\rho \nabla \phi$, we can not
apply the time weighted method mentioned above to establish the optimal decay
rate for the $N-th$ order spatial derivative of global solution.
Motivated by \cite{wu2020}, we establish some energy estimate for the quantity
$\int_{|\xi|\geq\eta}\widehat{\na^{N-1}v}\cdot \overline{\widehat{\na^{N}n}}d\xi$,
namely the higher frequency part, rather than $\int \nabla^{N-1} u\cdot\nabla^{N}n dx$.
Here $\widehat{\na^{N-1}v}$ and $\widehat{\na^{N}n}$ stand for the Fourier part of
$\na^{N-1}v$ and $\na^{N}n$ respectively.
The advantage is that the quantity $\|\na^{N}(n,v)\|_{L^2}^2-\eta_3\int_{|\xi|\geq\eta}\widehat{\na^{N-1}v}\cdot \overline{\widehat{\na^{N}n}}d\xi$
is equivalent to $\|\na^{N}(n, v)\|_{L^2}^2$.
Then, the combination of some energy estimate and decay estimate can help us build
the following inequality:
\beq\label{highesthigh}
\begin{split}	&\f{d}{dt}\Big\{\|\na^{N}(n,v)\|_{L^2}^2-\eta_3\int_{|\xi|\geq\eta}\widehat{\na^{N-1}v}\cdot \overline{\widehat{\na^{N}n}}d\xi \Big\}+\|\na^{N}v^h\|_{L^2}^2+\eta_3\|\na^{N}n^h\|_{L^2}^2\\
	\leq&  C\|\na^{N}(n^l, v^l)\|_{L^2}^2+C(1+t)^{-3-N}.
\end{split}
\eeq
Then, one has to estimate the decay rate of the low-frequency term
$\|\na^{N}(n^l, v^l)\|_{L^2}^2$.
Indeed, Duhamel's principle and decay estimate of $k-th$ $(0\leq k\leq N)$ order spatial derivative of solution obtained above allow us to obtain that\beno
\|\na^N(n^l, v^l)\|_{L^2}\leq C\delta\sup_{0\leq\tau\leq t}\|\na^N (n,v)\|_{L^2}+C(1+t)^{-\f34-\f N2}.
\eeno
which, together with \eqref{highesthigh}, by using the smallness of $\d$,
we can obtain the optimal decay rate for the $N-th$ order spatial derivative
of global solution to the CNS equations with external potential force.

Finally, we will establish the lower decay estimate, coincided with the upper one,
for the global solution itself and its spatial derivative.
It is noticed that the system in terms of density and momentum is always adopted to establish the lower bounds of decay rates for the global solution and its spatial derivatives in many previous works, one may refer to (\cite{chen2021,lizhang2010}).
However, the appearance of potential force term(i.e., $\rho \nabla \phi$)
prevents us taking this method to solve the problem.
Thus, let $(n, v)$ and $(\tilde{n}, \tilde{v})$ be the solutions of nonlinear
and linearized problem respectively.
Define the difference: $(n_\d, v_\d)\overset{def}{=}(n-\widetilde{n}, v-\widetilde{v})$,
it holds on
\beno
\|n\|_{L^2}\ge \|\widetilde{n}\|_{L^2}-\|{n}_\d\|_{L^2}
\quad \text{and} \quad
\|v\|_{L^2} \ge \|\widetilde{v}\|_{L^2}-\|{v}_\d\|_{L^2}.
\eeno
If these quantities obey the assumptions:
$\|(n_\d,v_\d)\|_{L^2}\leq \widetilde C\d(1+t)^{-\f34}$
and
$\min\{\|\widetilde{n}\|_{L^2},\|\widetilde{v}\|_{L^2}\}\geq \widetilde{c}(1+t)^{-\f34}$,
moreover, the constant $\d$ is a small constant and independent of $\widetilde{c}$,
then we can choose the constant $\d$ to obtain the decay rate
$\min\{\|{n}\|_{L^2},\|{v}\|_{L^2}\}\geq {c_1}(1+t)^{-\f34}$.
Similarly, it is easy to check that the decay \eqref{kdecaylow} holds true for $k=1$.
Based on the lower bound of decay rate for first order spatial derivative of solution and the upper bound of decay rate for the solution itself, we can deduce the lower bound of decay rate for $k-th$ $(2\leq k\leq N)$ order spatial derivative of solution by using the following Sobolev interpolation inequality:\beno
\|\na^k(n,v)\|_{L^2}\geq C\|\na(n,v)\|_{L^2}^k\|(n,v)\|_{L^2}^{-(k-1)},\quad\text{for}~~2\leq k\leq N.
\eeno
And more proof details of Theorem \ref{them4} can be found in Section \ref{lower} below.

\section{Preliminary}\label{pre}
In this section, we collect some elementary inequalities, which will be extensively used in later sections.
First of all, in order to estimate the term about $\bar\rho(x)$ in the CNS equations with a potential force, we need the following Hardy inequality.
\begin{lemm}[Hardy inequality]\label{hardy}
	For $k\geq1$, suppose that $\f{\na\phi}{(1+|x|)^{k-1}}\in L^2$, then $\f{\phi}{(1+|x|)^{k}}\in L^2$, with the estimate
	\beno
	\begin{split}
		\|\f{\phi}{(1+|x|)^{k}}\|_{L^2}\leq C\|\f{\na\phi}{(1+|x|)^{k-1}}\|_{L^2}.
	\end{split}
	\eeno
\end{lemm}
The proof of Lemma \ref{hardy} is simply and we omit it here. We will use the following Sobolev interpolation of Gagliardo-Nirenberg inequality frequently in energy estimates, which can be found in \cite{guo2012} more details.
\begin{lemm}[Sobolev interpolation inequality]\label{inter}
	Let $2\leq p\leq +\infty$ and $0\leq l,k\leq m$. If $p=+\infty$, we require furthermore that $l\leq k+1$ and $m\geq k+2$. Then if $\na^l\phi\in L^2$ and $\na^m \phi\in L^2$, we have $\na^k\phi\in L^p$. Moreover, there exists a positive constant $C$ dependent only on $k,l,m,p$ such that
	\begin{equation}\label{Sobolev}
	\|\na^k\phi\|_{L^p}\leq C\|\na^l\phi\|_{L^2}^{\theta}\|\na^m\phi\|_{L^2}^{1-\theta},
	\end{equation}
	where $0\leq\theta\leq1$ satisfying
	\beno
	\f k3-\f1p=\Big(\f l3-\f12\Big)\theta+\Big(\f m3-\f12\Big)(1-\theta).
	\eeno
\end{lemm}

Then we recall the following commutator estimate, which is used frequently in energy estimates. The proof and more details may refer to \cite{majda2002}.
\begin{lemm}\label{commutator}
	Let $k\geq1$ be an integer and define the commutator\beno
       [\na^k,f]g=\na^k(fg)-f\na^kg.
    \eeno
    Then we have
    \beno
    \|[\na^k,f]g\|_{L^2}\leq C\|\na f\|_{L^\infty}\|\na^{k-1}g\|_{L^2}+C\|\na^k f\|_{L^2}\|g\|_{L^\infty},
    \eeno
    where $C$ is a positive constant dependent only on $k$.
\end{lemm}

Finally, we conclude this section with the following lemma. The proof and more details may refer to \cite{chen2021}.
\begin{lemm}\label{tt2}
Let $r_1,r_2>0$ be two real numbers, for any $0<\ep_0<1$, we have
\beno
\begin{split}
	\int_0^{\f t2}(1+t-\tau)^{-r_1}(1+\tau)^{-r_2}d\tau\leq C& \left\{\begin{array}{l}
		(1+t)^{-r_1},\quad \text{for}~~ r_2>1,\\
		(1+t)^{-r_1+\ep_0},\quad~~ \text{for}~~ r_2=1, \\
		(1+t)^{-(r_1+r_2-1)},\quad \text{for}~~ r_2<1,
	\end{array}\right.\\
	\int_{\f t2}^{t}(1+t-\tau)^{-r_1}(1+\tau)^{-r_2}d\tau\leq C& \left\{\begin{array}{l}
		(1+t)^{-r_2},\quad \text{for}~~ r_1>1,\\
		(1+t)^{-r_2+\ep_0},\quad~~ \text{for}~~ r_1=1, \\
		(1+t)^{-(r_1+r_2-1)},\quad \text{for}~~ r_1<1,
	\end{array}\right.
\end{split}
\eeno
where $C$ is a positive constant independent of $t$.
\end{lemm}

\section{The proof of Theorem \ref{hs2}}\label{h-s}

In this section, we study the optimal decay rate of the $N-th$
spatial derivative of global small solution for the initial value problem \eqref{ns1}--\eqref{boundary1}.
Thus, let us write $n\overset{def}{=} \rho-\bar\rho$, then the original system \eqref{ns1}--\eqref{initial1} can be rewritten in the perturbation form as
\beq\label{ns2}
\left\{\begin{array}{lr}
	n_t +\bar\rho\dive u=S_1,\quad (t,x)\in \mathbb{R}^{+}\times \mathbb{R}^3,\\
	u_t+\gamma \bar\rho\na n-\bar\mu\tri u-(\bar\mu+\bar\lam) \na \dive u=S_2,\quad (t,x)\in \mathbb{R}^{+}\times \mathbb{R}^3,\\
	(n,u)|_{t=0}=(\rho_0-\bar\rho,u_0),\quad x\in \mathbb{R}^3,
\end{array}\right.
\eeq
where the functions $f(n), g(n)$ and source terms $S_i(i=1,2)$ are defined by
\beno\label{fgdefine}
\begin{split}
    &f(n)\overset{def}{=}\f{n}{n+\bar\rho},\quad \quad
    g(n)\overset{def}{=}\f{p'(n+\bar\rho)}{n+\bar\rho}-\f{p'(\bar\rho)}{\bar\rho},\\
	&S_1 \overset{def}{=}-n\dive u-u\cdot\na n,\\
	&S_2 \overset{def}{=} -u\cdot \na u
         -f(n)\big(\bar\mu\tri u-(\bar\mu+\bar\lam) \na \dive u\big)-g(n)\na n.
\end{split}
\eeno
Here the coefficients $\bar\mu, \bar\lam$ and $\gamma$
are defined by
$\bar\mu=\f{\mu}{\bar\rho}, \
\bar\lam=\f{\lam}{\bar\rho}, \
\gamma=\f{p'(\bar\rho)}{\bar\rho^2}.$
Due the the uniform estimate \eqref{energy-01}, then there exists a positive
constant $C$ such that for any $1\leq k\leq N$,
\beno
\begin{split}
	|f(n)|\leq C|n|,\quad |g(n)|\leq C|n|,\quad
	|f^{(k)}(n)|\leq C,\quad |g^{(k)}(n)|\leq C.
\end{split}
\eeno
Next, in order to estimate the $L^2-$norm of the spatial derivatives of $f(n)$ and $g(n)$, we shall record the following lemma, which will be used frequently in later estimate.
\begin{lemm}\label{hrholem}
	Under the assumptions of Theorem \ref{hs2}, $f(n)$ and $g(n)$ are two functions of $n$ defined by \eqref{fgdefine}, then for any integer $1\leq m\leq N-1$, it holds true
	\beq\label{hrho}
	\|\na^mf(n)\|_{L^2}^2+\|\na^mg(n)\|_{L^2}^2\leq C(1+t)^{-(m+s)},
	\eeq
	where $C$ is a positive constant independent of time.
\end{lemm}
\begin{proof}
	We only control the first term on the left handside of \eqref{hrho}, and the other one can be controlled similarly.
	Notice that for $m\geq 1$,
	\[\na^mf(n)=\text{a sum of products}~~f^{\gamma_1,\gamma_2,\cdots,\gamma_j}(n)\na^{\gamma_1}n\cdots\na^{\gamma_j}n\]
	with the functions $f^{\gamma_1,\gamma_2,\cdots,\gamma_j}(n)$ are some derivatives of $f(n)$ and $1\leq \gamma_{i}\leq m$, $i=1,2,\cdots,j$, $\gamma_1+\gamma_2+\cdots+\gamma_j=m$, $j\geq 1$. Without loss of generality, we assume that $1\leq \gamma_1\leq\gamma_2\leq\cdots\leq\gamma_j\leq m$. Thus, if $j\geq2$, we have $\gamma_{j-1}\leq m-1\leq N-2$.
	It follows from the decay estimate \eqref{sde1} that for $1\le m\leq N-1$,\beno
	\|\na^mn\|_{L^2}\leq C(1+t)^{-(m+s)},
	\eeno
	which, together with Sobolev inequality and the fact that $j\geq1$, we deduce that
	\beno
	\begin{split}
	&\|f^{\gamma_1,\gamma_2,\cdots,\gamma_j}(n)\na^{\gamma_1}n\cdots\na^{\gamma_j}n\|_{L^2}\\
		\leq & C \|\na^{\gamma_1}n\|_{L^\infty}\cdots\|\na^{\gamma_{j-1}}n\|_{L^\infty}\|\na^{\gamma_{j}}n\|_{L^2}\\
		\leq &C \|\na^{\gamma_1+1}n\|_{H^1}\cdots\|\na^{\gamma_{j-1}+1}n\|_{H^1} \|\na^{\gamma_j}n\|_{L^2}\\
		\leq&C(1+t)^{-\f{m+j s+(j-1)}{2}}
		\leq C(1+t)^{-\f{m+ s}{2}}.
	\end{split}
	\eeno
   Consequently, the proof of this lemma is completed.
\end{proof}
Now, based on lemma \ref{hrholem}, we give the following lemma, which provides the time integrability of $N-th$ order spatial derivate of the solution $(n,u)$.
\begin{lemm}\label{integrability}
Under the assumptions in Theorem \ref{hs2},
for any fixed constant $0<\ep_0<1$, we have
\beq\label{Enc} (1+t)^{N+s-1}\|\na^{N-1}(n,u)\|_{H^1}^2
+(1+t)^{-\ep_0}\int_0^t(1+\tau)^{N+s+\ep_0-1}
\big(\|\na^N n\|_{L^2}^2+\|\na^N u\|_{H^1}^2\big)d\tau
\leq C,
\eeq
where $C$ is a positive constant independent of $t$.
\end{lemm}
\begin{proof}
From Lemmas 3.2, 3.3 and 3.4 in  \cite{guo2012},
the following estimates hold on for all $0\leq k\leq N-1$,
\begin{align}
\label{in01} \f{d}{dt}\|\nabla^k(n,u)\|_{L^2}^2+\|\nabla^{k+1}u\|_{L^2}^2\leq
& C\delta_0 \|\nabla^{k+1}(n,u)\|_{L^2}^2,\\
\label{in02} \f{d}{dt}\|\nabla^{k+1}(n,u)\|_{L^2}^2+\|\nabla^{k+2}u\|_{L^2}^2
\leq& C\delta_0 \big(\|\nabla^{k+1}(n,u)\|_{L^2}^2+\|\nabla^{k+2}u\|_{L^2}^2 \big),\\
\label{in03} \f{d}{dt}\int \nabla^ku\cdot\nabla^{k+1}n dx+\|\nabla^{k+1}n\|_{L^2}^2
\leq& C \big( \|\nabla^{k+1}u\|_{L^2}^2+\|\nabla^{k+2}u\|_{L^2}^2\big).
\end{align}
where the constant $C$ is a positive constant independent of time.
Based on the estimates \eqref{in01}, \eqref{in02} and \eqref{in03},
then it holds for all $0\leq k\leq N-1$,
\begin{equation}\label{e1}
\begin{split}
\f{d}{dt}\big(\|\nabla^k(n,u)\|_{H^1}^2
+\eta_1 \int \nabla^k u \cdot\nabla^{k+1}n dx\big)
+\eta_1\|\nabla^{k+1}n\|_{L^2}^2+\|\nabla^{k+1}u\|_{H^1}^2\leq 0.
\end{split}
\end{equation}
Here $\eta_1$ is a small positive constant.
Taking $k=N-1$ in inequality \eqref{e1}, then we have
\begin{equation}\label{e2}
\f{d}{dt}\mathcal{E}^{N-1}(t)
+\eta_1\|\nabla^{N}n\|_{L^2}^2+\|\nabla^{N}u\|_{H^1}^2\leq 0,
\end{equation}
where the energy $\mathcal{E}^{N-1}(t)$ is defined by
$$
\mathcal{E}^{N-1}(t)\overset{def}{=}
\|\nabla^{N-1}(n,u)(t)\|_{H^1}^2+\eta_1 \int \nabla^{N-1} u\cdot\nabla^{N}n dx.
$$
Then, due to the smallness of $\eta_1$, we have the following equivalent relation
\beq\label{xih}
c_1\|\nabla^{N-1}(n, u)(t)\|_{H^1}^2
\le \mathcal{E}^{N-1} (t)
\le  c_2\|\nabla^{N-1}(n, u)(t)\|_{H^1}^2,
\eeq
where the constants $c_1$ and $c_2$ are independent of time.
For any fixed $\ep_0(0<\ep_0<1)$,
multiplying the inequality \eqref{e2} by $(1+t)^{N+s+\ep_0-1}$,
it holds true
\begin{equation}\label{Ek1}
\begin{split}
\f{d}{dt}\big\{(1+t)^{N+s+\ep_0-1}\mathcal{E}^{N-1}(t)\big\}
          +(1+t)^{N+s+\ep_0-1}
          \big(\|\nabla^{N}n\|_{L^2}^2+\|\nabla^{N}u\|_{H^1}^2\big)
\leq C (1+t)^{N+s+\ep_0-2}\mathcal{E}^{N-1}(t).
\end{split}
\end{equation}
The decay estimate \eqref{sde1} and equivalent relation
\eqref{xih} lead us to get that for $0<k\leq N-1$,
\beno
(1+t)^{N+s+\ep_0-2}\mathcal{E}^{N-1}(t)
\leq C(1+t)^{N+s+\ep_0-2}\|\nabla^{N-1}(n,u)\|_{H^1}^2
\leq C(1+t)^{-1+\ep_0},
\eeno
	which, together with inequality \eqref{Ek1}, yields directly
    \beq\label{Ek}
	\f{d}{dt}\big\{(1+t)^{N+s+\ep_0-1}\mathcal{E}^{N-1}(t)\big\}
          +(1+t)^{N+s+\ep_0-1}
          \big(\|\nabla^{N}n\|_{L^2}^2+\|\nabla^{N}u\|_{H^1}^2\big)
    \leq C(1+t)^{{-1+\ep_0}}.
	\eeq
    Integrating the inequality \eqref{Ek} over $[0, t]$, it holds true
	\begin{equation}
    \begin{aligned}
    &(1+t)^{N+s+\ep_0-1}\mathcal{E}^{N-1}(t)
       +\int_0^t(1+\tau)^{N+s+\ep_0-1}\big(\|\na^N n\|_{L^2}^2
       +\|\na^{N}u\|_{H^1}^2\big)d\tau\\
    &\leq \mathcal{E}^{N-1}(0)+\int_0^t(1+\tau)^{-1+\ep_0}d\tau\leq C(1+t)^{\ep_0},
    \end{aligned}
    \end{equation}
	which, together with the equivalent relation \eqref{xih}, implies that
	\beno
	(1+t)^{N+s+\ep_0-1}\|\na^k(n,u)\|_{H^1}^2
     +\int_0^t(1+\tau)^{N+s+\ep_0-1}
     \big(\|\na^{N}n\|_{L^2}^2+\|\na^{N}u\|_{H^1}^2\big)d\tau\leq C(1+t)^{\ep_0}.
	\eeno
	Consequently, we thereupon obtain \eqref{Enc}.
\end{proof}
Finally, using the time integrability of the dissipative term of the $N-th$ order spatial derivative of $(n,u)$ obtained in Lemma \ref{integrability}, we can establish the optimal decay of $N-th$ one.
\begin{lemm}\label{ENN}
	Under the assumptions in Theorem \ref{hs2}, then the global solution
    $(n, u)$ has the decay estimate
\beq\label{Enn}	
(1+t)^{N+s}\|\na^N(n,u)\|_{L^2}^2
+(1+t)^{-\ep_0}\int_0^t(1+\tau)^{N+s+\ep_0}\|\na^{N+1}u\|_{L^2}^2d\tau\leq C,
\eeq
where $C$ is a positive constant independent of time.
\end{lemm}
\begin{proof}
Applying differential operator $\na^N$ to \eqref{ns2}
and multiplying the resulting equation by $\na^N (n,u)$, it holds
\beq\label{nes}
\begin{split}
	\f{d}{dt}\|\nabla^N(n,u)\|_{L^2}^2+\|\nabla^{N+1}u\|_{L^2}^2\leq C\int \na^N S_1\cdot \na^Nn dx+C\int \na^N S_2\cdot \na^Nu dx.
\end{split}
\eeq
Next, we estimate two terms on the right handside of \eqref{nes}.
The definition of $S_1$ leads us to get
\beq\label{s11}
\int \na^N S_1\cdot \na^Nn dx=-\int \na^N (n\dive u+u\cdot \na n)\cdot \na^Nn dx.
\eeq
Invoking Sobolev inequality, we deduce
\beq\label{s12}
\begin{split}
	\int \na^N (n\dive u)\cdot \na^Nn dx
	\leq& C \|\na^N(n\dive u)\|_{L^2}\|\na^N n\|_{L^2}\\
	\leq&C\big(\|n\|_{L^\infty}\|\na^N\dive u\|_{L^2}+\|\dive u\|_{L^\infty}\|\na^Nn\|_{L^2}\big)\|\na^N n\|_{L^2}\\
	\leq&\ep \|\na^{N+1} u\|_{L^2}+C_{\ep}\|\na n\|_{H^1}^2\|\na^N n\|_{L^2}^2+C\|\na^2u\|_{H^1}\|\na^N n\|_{L^2}^2\\
	\leq &\ep \|\na^{N+1}u\|_{L^2}^2
           +C_{\ep}(1+t)^{-(1+s)}\|\na^N(n,u)\|_{L^2}^2.
\end{split}
\eeq
Integrating by parts, then employing Sobolev inequality, Lemma \ref{commutator}
and decay estimate \eqref{sde1}, we find that
\beq\label{s13}
\begin{split}
	\int \na^N (u\cdot \na n)\cdot \na^Nn dx
	=&-\int\dive u|\na^Nn|^2 dx+\int\big([\na^{N},u]\cdot\na n\big)\na^N n dx\\
	\leq&C \big(\|\na u\|_{L^\infty}\|\na^{N} n\|_{L^2}+\|\na^N u\|_{L^2}\|\na n\|_{L^\infty}\big)\|\na^N n\|_{L^2}\\
	\leq&C \big(\|\na^2 u\|_{H^1}\|\na^{N} n\|_{L^2}+\|\na^N u\|_{L^2}\|\na^{2} n\|_{H^1}\big)\|\na^N n\|_{L^2}\\
	\leq&C (1+t)^{-(1+\frac{s}{2})}\|\na^N(n,u)\|_{L^2}^2.
\end{split}
\eeq
Substituting estimates \eqref{s12} and \eqref{s13} into \eqref{s11}, we have
\beq\label{S1}
\begin{split}
	\int \na^N S_1\cdot \na^Nn dx
    \leq\ep \|\na^{N+1}u\|_{L^2}^2+C_{\ep}(1+t)^{-1}\|\na^N(n,u)\|_{L^2}^2.
\end{split}
\eeq
Next integrating by parts, and using H{\"o}lder inequality, we see that
\beq\label{S22}
\begin{split}
	\int \na^N S_2\cdot \na^Nu dx\leq& C \|\na^{N-1} S_2\|_{L^2}\|\na^{N+1}u\|_{L^2}\\
	\leq&C\big(\|\na^{N-1}(u\cdot \na u)\|_{L^2}+\|\na^{N-1}\big[f(n)(\bar\mu\tri u+(\bar\mu+\bar \lam)\na \dive u)\big]\|_{L^2}\\
	&\quad+\|\na^{N-1}\big(g(n)\na n\big)\|_{L^2}\big)\|\na^{N+1}u\|_{L^2}\overset{def}{=}H_{1}+H_{2}+H_{3}.
\end{split}
\eeq
We use Sobolev inequality, H{\"o}lder inequality  and Lemma \ref{commutator}, to find
\beq\label{unau}
\begin{split}
	H_1
	\leq&C\big(\|u\|_{L^\infty}\|\na^{N} u\|_{L^2}+\|[\na^{N-1},u]\cdot\na n\|_{L^2}\big)\|\na^{N+1}u\|_{L^2}\\
	\leq&C \big(\|u\|_{L^\infty}\|\na^{N} u\|_{L^2}+\|\na u\|_{L^\infty}\|\na^{N-1} u\|_{L^2}\big)\|\na^{N+1}u\|_{L^2}\\
	\leq&\ep\|\na^{N+1}u\|_{L^2}^2+C_{\ep}\|\na u\|_{H^1}^2\|\na^{N} u\|_{L^2}^2+C_{\ep}\|\na^2 u\|_{H^1}^2\|\na^{N-1} u\|_{L^2}^2\\
	\leq&\ep\|\na^{N+1}u\|_{L^2}^2+C_{\ep}(1+t)^{-(1+s)}\|\na^{N} u\|_{L^2}^2+C_{\ep}(1+t)^{-{(N+1+2s)}},
\end{split}
\eeq
where we have used the decay estimate \eqref{sde1} in the last inequality.
We employ once again Sobolev inequality and decay \eqref{sde1}, combined with \eqref{hrho}, to deduce that
\beq\label{fn}
\begin{split}
	H_{2}
	\leq&C \big(\|f(n)\|_{L^\infty}\|\na^{N+1} u\|_{L^2}+\|\na^{N-1} f(n)\|_{L^2}\|\na^{2} u\|_{L^\infty}\big)\|\na^{N+1}u\|_{L^2}\\
	\leq&C\big(\|\na f(n)\|_{H^1}\|\na^{N+1} u\|_{L^2}+\|\na^{N-1} f(n)\|_{L^2}\|\na^{3} u\|_{L^2}^{\f12}\|\na^{4} u\|_{L^2}^{\f12}\big)\|\na^{N+1}u\|_{L^2}\\
	\leq
&C\delta_0\|\na^{N+1}u\|_{L^2}^2
  +C(1+t)^{-\f{N-1+\s}{2}}\|\na^{3} u\|_{L^2}^{\f12}\|\na^{4} u\|_{L^2}^{\f12}\|\na^{N+1}u\|_{L^2}.
\end{split}
\eeq
Since the initial data $(\rho_0-\bar\rho,u_0)\in H^{N}$ with integer $N\geq 3$,
the second term on the right handside of inequality \eqref{fn}
should be estimated in the following two cases.
If the integer $N=3$, we apply the Cauchy inequality to obtain
\beq\label{fn-01}
\begin{split}
		(1+t)^{-\f{3+s}{2}}\|\na^{3} u\|_{L^2}^{\f12}\|\na^{4} u\|_{L^2}^{\f32}
		\leq& \ep\|\na^4  u\|_{L^2}^2+C_\ep(1+t)^{-2(3+s)}\|\na^{3} u\|_{L^2}^2.
\end{split}
\eeq
If the integer $N\geq4$, then we apply  the decay \eqref{sde1} to get
\beq\label{fn-02}
\begin{split}
(1+t)^{-\f{N-1+s}{2}}\|\na^{3} u\|_{L^2}^{\f12}
\|\na^{4} u\|_{L^2}^{\f12}\|\na^{N+1}u\|_{L^2}
\leq \ep\|\na^{N+1}u\|_{L^2}^2+C_{\ep}(1+t)^{-(N+2+2s)}.
\end{split}
\eeq
Substituting the estimates \eqref{fn-01} and \eqref{fn-02}
into \eqref{fn}, we find that
\beq\label{fnu}
\begin{split}
	&H_{2}
	\leq (\ep+C\delta_0)\|\na^{N+1}u\|_{L^2}^2
    +C_{\ep}(1+t)^{-1}\|\na^{N}u\|_{L^2}^2+C_{\ep}(1+t)^{-(N+2+2s)}.
\end{split}
\eeq
Likewise, employing \eqref{sde1} and \eqref{hrho} once again, we calculate
\beq\label{gnu}
\begin{split}
	H_{3}
	\leq&C\Big(\|g(n)\|_{L^\infty}\|\na^{N} n\|_{L^2}+\|\na^{N-1} g(n)\|_{L^2}\|\na n\|_{L^\infty}\Big)\|\na^{N+1}u\|_{L^2}\\
	\leq&\ep\|\na^{N+1}u\|_{L^2}^2+C_{\ep}\Big(\|\na g(n)\|_{H^1}^2\|\na^{N} n\|_{L^2}^2+\|\na^2 n\|_{H^1}^2\|\na^{N-1} g(n)\|_{L^2}^2\Big)\\
	\leq&\ep\|\na^{N+1}u\|_{L^2}^2+C_{\ep}(1+t)^{-(1+\s)}\|\na^{N}n\|_{L^2}^2+C_{\ep}(1+t)^{-{(N+1+2\s)}}.
\end{split}
\eeq
Substituting the estimates \eqref{unau}, \eqref{fnu} and \eqref{gnu} into \eqref{S22}, then using Young inequality, we thereby obtain that
\beq\label{S2}
\begin{split}
	\int \na^N S_2\cdot \na^Nu dx\leq (\ep+C\delta_0)\|\na^{N+1}u\|_{L^2}^2
+C_{\ep}(1+t)^{-1}\|\na^{N}(n,u)\|_{L^2}^2
+C_{\ep}(1+t)^{-(N+1+s)}.
\end{split}
\eeq
We utilize the estimates \eqref{S1} and \eqref{S2} into \eqref{nes}, and choose $\ep$ and $\delta_0$ suitably small, to discover that
\beno
\begin{split}
	\f{d}{dt}\|\nabla^N(n,u)\|_{L^2}^2+\|\nabla^{N+1}u\|_{L^2}^2\leq C(1+t)^{-1}\|\na^{N}(n,u)\|_{L^2}^2+C(1+t)^{-(N+1+s)}.
\end{split}
\eeno
Multiplying the above inequality by $(1+t)^{{N+s+\ep_0}}$
and integrating with respect to $t$, we find
\beq\label{S21}
\begin{split} &(1+t)^{N+s+\ep_0}\|\nabla^N(n,u)\|_{L^2}^2
+\int_{0}^{t}(1+\tau)^{N+s+\ep_0}\|\nabla^{N+1}u\|_{L^2}^2d\tau\\
\leq &\|\nabla^N(n_0,u_0)\|_{L^2}^2
     +C\int_0^t(1+\tau)^{N+s-1+\ep_0}\|\na^{N}(n,u)\|_{L^2}^2d\tau
     +C\int_0^t(1+\tau)^{-1+\ep_0}d\tau\\
\leq&C(1+\|\nabla^N(n_0,u_0)\|_{L^2}^2)+C(1+t)^{\ep_0},
\end{split}
\eeq
where we have used the decay estimate \eqref{Enc}.
Thus, the estimate \eqref{S21} yields the estimate \eqref{Enn} directly.
Therefore, we complete the proof of this lemma.
\end{proof}

\underline{\noindent\textbf{The proof of Theorem \ref{hs2}.}}
With the help of estimate \eqref{Enn} in Lemma \ref{ENN}, it holds true
$$
\|\na^N n(t)\|_{L^2}^2+\|\na^N u(t)\|_{L^2}^2 \leq C(1+t)^{-(s+N)},
$$
where $C$ is a positive constant independent of time.
Thus, we obtain the estimate \eqref{sde3} immediately,
and hence, finish the proof of Theorem \ref{hs2}.

\section{The proof of Theorems \ref{them3} and \ref{them4}}\label{rhox}

In this section, we will give the proof for the Theorem \ref{them3}
that includes the global well-posedness theory and time decay estimate \eqref{kdecay}.
First of all, the global small solution of the CNS equations can be proven
just by taking the strategy of energy method in \cite{duan2007}
when the initial data is small perturbation near the equilibrium state.
Thus, we assume that the global solution $(\rho, u)$ in Theorem \ref{them3}
exists and satisfies the energy estimate \eqref{energy-thm},i.e.,
\beq\label{energy-thm-01}
\begin{split}
\|(\rho-\rho^*,u)\|_{H^N}^2
+\int_0^t\big(\|\nabla(\rho-\rho^*)\|_{H^{N-1}}^2+\|\nabla u\|_{H^N}^2\big)ds
\leq C\|(\rho_0-\rho^*,u_0)\|_{H^N}^2,
\end{split}
\eeq
for all $t\geq 0$.
Secondly, similar to the decay estimate \eqref{highdecayL1}, one can combine
the energy estimate and the decay rate of linearized system to obtain the following
decay estimates:
\begin{equation}\label{basic-decay}
\|\nabla^k (\rho-\rho^*)(t)\|_{H^{N-k}}+\|\nabla^k u(t)\|_{H^{N-k}}
\le C(1+t)^{-(\frac34+\frac{k}{2})},\quad k=0,1,
\end{equation}
if the initial data $(\rho_0-\rho^*, u_0)$ belongs to $L^1$ additionally.
Now, we focus on establishing the optimal decay rate for the higher order spatial
derivative of solution. In other words, we will prove that the decay rate
\eqref{basic-decay} holds on for the case $k=2,...,N$.
Thus, let us set
$$n(x,t)\overset{def}{=}\rho(x,t)-\rho^*(x),\quad
\bar\rho(x)\overset{def}{=} \rho^*(x)-\rho_{\infty}, \quad
v\overset{def}{=}\f{\rho_{\infty}}{\sqrt{p'(\rho_{\infty})}}u,
$$
then \eqref{ns3}--\eqref{initial-boundary} can be rewritten in the perturbation form
\beq\label{ns5}
\left\{\begin{array}{lr}
	n_t +\gamma\dive v=\widetilde S_1,\quad (t,x)\in \mathbb{R}^{+}\times \mathbb{R}^3,\\
	v_t+\gamma\na n-\mu_1\tri v-\mu_2\na\dive v =\widetilde  S_2,\quad (t,x)\in \mathbb{R}^{+}\times \mathbb{R}^3,\\
	(n,v)|_{t=0}\overset{def}{=}(n_0,v_0)=(\rho_0-\rho^*,\f{\rho_{\infty}}{\sqrt{p'(\rho_{\infty})}}u_0)\rightarrow (0,0)~~\text{as}~~|x|\rightarrow \infty,
\end{array}\right.
\eeq
where
$\mu_1=\f{\mu}{\rho_{\infty}}, \mu_2=\f{\mu+\lam}{\rho_{\infty}},
\gamma=\sqrt{p'(\rho_{\infty})}$, and
\beno
\begin{split}
	\widetilde S_1&=-\f{\mu_1\gamma}{\mu}\dive [(n+\bar\rho)v],\\
	\widetilde S_2&=-\f{\mu_1\gamma}{\mu}v\cdot \na v-\wf(n+\bar\rho)\big(\mu_1\tri v+\mu_2\na\dive v\big)-\wg(n+\bar\rho) \na n-\h(n,\bar\rho)\na\bar\rho.
\end{split}
\eeno
Here the nonlinear functions $\wf, \wg$ and $\h$ are defined by
\beno
\begin{split}
&\wf(n+\bar\rho)\overset{def}{=}\f{n+\bar\rho}{n+\bar\rho+\rho_{\infty}},\quad \wg(n+\bar\rho)\overset{def}{=}\f{\mu}{\mu_1\gamma}\Big(\f{p'(n+\bar\rho+\rho_{\infty})}{n+\bar\rho+\rho_{\infty}}-\f{p'(\rho_{\infty})}{\rho_{\infty}}\Big),\\ &\h(n,\bar\rho)\overset{def}{=}\f{\mu}{\mu_1\gamma}\Big(\f{p'(n+\bar\rho+\rho_{\infty})}{n+\bar\rho+\rho_{\infty}}-\f{p'(\bar\rho+\rho_{\infty})}{\bar\rho+\rho_{\infty}}\Big).
\end{split}
\eeno
In view of the definition of $\h$ and $\wg$, it holds true
\beq\label{hg-relation}
\h(n,\bar\rho)=\wg(n+\bar\rho)-\wg(\bar\rho),
\eeq
which will be used in this section.

\subsection{Energy estimates}
In this subsection, we will establish the following differential inequality that
will play important role for us to establish the optimal decay rate
for the higher order spatial derivative of solution.
First of all, let us define the energy $\mathcal{E}^N_l(t)$ as
$$
\mathcal{E}^N_l(t)\overset{def}{=}
\sum_{m=l}^{N}\|\na^m(n,v)\|_{L^2}^2
+\eta_2\sum_{m=l}^{N-1}\int\na^{m}v\cdot\na^{m+1}n dx, \quad 0\le l \le N,
$$
where $\eta_2$ is a small positive constant.
Due to the smallness of parameter $\eta_2$, we have the following equivalent realtion
\begin{equation}\label{emleq}
c_3 \|\nabla^l(n, v)\|_{H^{N-l}}^2
\le \mathcal{E}^N_l(t)
\le c_4 \|\nabla^l(n, v)\|_{H^{N-l}}^2 .
\end{equation}
where $c_3$ and $c_4$ are positive constant independent of time.
Finally, the relation \eqref{p} and the condition \eqref{phik} in Theorem \ref{them3}
lead us to obtain
\beno
\begin{split}
\sum_{k=0}^{N+1}\|(1+|x|)^{k}\na^k(\rho^*-\rho_{\infty})\|_{L^2\cap L^\infty}\leq \delta.
\end{split}
\eeno
This, together with the Sobolev interpolation inequality, yields
\begin{equation}\label{density-control}
\sum_{k=0}^{N+1}\|(1+|x|)^{k}\na^k(\rho^*-\rho_{\infty})\|_{L^p}\leq \delta,
\quad 2\leq p\leq+\infty.
\end{equation}
This inequality will be used frequently in Section \ref{rhox}.
Now we state the main result in this subsection.

\begin{prop}\label{nenp}
Under the assumptions in Theorem \ref{them3}, for any $1\leq l\leq N-1$, we have
\beq\label{eml}
\begin{split}	\f{d}{dt}\mathcal{E}^N_l(t)+\eta_2\sum_{m=l}^{N-1}\|\na^{m+1}n\|_{L^2}^2
+\sum_{m=l}^{N}\|\na^{m+1}v\|_{L^2}^2\leq 0.
\end{split}
\eeq
Here $\eta_2$ is a small positive constant.
\end{prop}
Due to the energy estimate \eqref{energy-thm-01},
there exists a positive constant $C$ such that for any $k\geq 1$,
\beno
\begin{split}
	|\wf(n+\bar\rho)|\leq C|n+\bar\rho|,\quad |\wg(n+\bar\rho)|\leq C|n+\bar\rho|,\quad
	|\wf^{(k)}(n+\bar\rho)|\leq C,\quad |\wg^{(k)}(n+\bar\rho)|\leq C.
\end{split}
\eeno

Next we give three lemmas as follows, which is the key to prove Proposition \ref{nenp}. The first one is the basic energy estimate for $k-th$ $(1\leq k\leq N-1)$ order spatial derivative of solution.

\begin{lemm}\label{enn-1}
Under the assumptions in Theorem \ref{them3}, for $1\leq k\leq N-1$, we have
\beq\label{en1}
\f{d}{dt}\|\na^k(n,v)\|_{L^2}^2+\|\na^{k+1}v\|_{L^2}^2
\leq C\delta \|\na^{k+1}(n,v)\|_{L^2}^2,
\eeq
where $C$ is a positive constant independent of time.
\end{lemm}
\begin{proof}
Applying differential operator $\na^k$ to $\eqref{ns5}_1$
and $\eqref{ns5}_2$, multiplying the resulting equations by $\na^kn$ and $\na^kv$, respectively, then integrating over $\mathbb{R}^3$, one easily shows that\beq\label{ehk}
	\begin{split}
		\f12\f{d}{dt}\|\na^k(n,v)\|_{L^2}^2+\mu_1\|\na^{k+1} v\|_{L^2}^2+\mu_2\|\na^k\dive v\|_{L^2}^2
		= \int \na^k\widetilde{S}_1\cdot \na^kn dx+\int \na^k\widetilde{S}_2\cdot \na^kv dx.
	\end{split}
	\eeq
Integrating by part and using the definition of $\widetilde{S}_1$, it holds
\beq\label{ks1}
\begin{aligned}
\int \na^k\widetilde{S}_1\cdot \na^kn dx
\leq &C\int \na^{k-1}( v\cdot\na n )\cdot \na^{k+1}n dx
     +C\int \na^{k-1}(n\dive v )\cdot \na^{k+1}n dx\\
    & +C\int \na^{k-1}(v\cdot\na\bar\rho )\cdot \na^{k+1}n dx
     +C\int \na^{k-1}(\bar\rho \dive v)\cdot \na^{k+1}n dx\\
    \overset{def}{=}&I_1+I_2+I_3+I_4.
\end{aligned}
\eeq
It follows from H{\"o}lder and Sobolev interpolation inequalities that
\beq\label{I1}
\begin{split}
|I_1| \leq C
\sum_{0\leq l\leq k-1}\|\na^{k-1-l}v\|_{L^3}\|\na^{l+1}n\|_{L^6}\|\na^{k+1}n\|_{L^2}.
\end{split}
\eeq
The Sobolev interpolation inequality \eqref{Sobolev} in Lemma \ref{inter} yields directly
\beq\label{nanak}
\|\na^{k-1-l}v\|_{L^3}\|\na^{l+1}v\|_{L^6}
\leq C\|\na^{\al}v\|_{L^2}^{\f{l+2}{k+1}}\|\na^{k+1}v\|_{L^2}^{1-\f{l+2}{k+1}}
     \|v\|_{L^2}^{1-\f{l+2}{k+1}}\|\na^{k+1}v\|_{L^2}^{\f{l+2}{k+1}},
\eeq
where
$$\al=\f{k+1}{2(l+2)}\leq \f{k+1}{4}\leq \f{N}{4}.$$
Substituting estimate \eqref{nanak} into \eqref{I1}, it holds true
\beq\label{i1}
\begin{split}
|I_1| \leq C\sum_{0\leq l\leq k-1}
\|\na^{\al}v\|_{L^2}^{\f{l+2}{k+1}}\|n\|_{L^2}^{1-\f{l+2}{k+1}}
\|\na^{k+1}(n,v)\|_{L^2}^2
\leq C\delta \|\na^{k+1}(n,v)\|_{L^2}^2.
\end{split}
\eeq
Similarly, it is easy to check that
\beq\label{i2}
\begin{split}
|I_2| \leq C\sum_{0\leq l\leq k-1}
\|\na^{k-1-l}n\|_{L^3}\|\na^{l+1}v\|_{L^6}\|\na^{k+1}n\|_{L^2}\leq C\delta \|\na^{k+1}(n,v)\|_{L^2}^2.
	\end{split}
	\eeq
The application of  H{\"o}lder inequality implies directly
\beq\label{ii3}
\begin{split}
|I_3|\leq C\sum_{0\leq l\leq k-1}
         \|(1+|x|)^{l+1}\na^{l+l}\bar\rho\|_{L^3}
         \|\f{\na^{k-1-l}v}{(1+|x|)^{l+1}}\|_{L^6}\|\na^{k+1}n\|_{L^2}.
\end{split}
\eeq
Since
\beq\label{vl6}
\|\f{\na^{k-1-l}v}{(1+|x|)^{l+1}}\|_{L^6}\leq C \|\na\Big(\f{\na^{k-1-l}v}{(1+|x|)^{l+1}}\Big)\|_{L^2}\leq C \|\f{\na^{k-l}v}{(1+|x|)^{l+1}}\|_{L^2}+C\|\f{\na^{k-1-l}v}{(1+|x|)^{l+2}}\|_{L^2},
\eeq
one can deduce from \eqref{ii3} and Hardy inequality that
\beq\label{i3}
\begin{split}
|I_3|\leq C\delta\sum_{0\leq l\leq k-1}
\Big(\|\f{\na^{k-l}v}{(1+|x|)^{l+1}}\|_{L^2}
+\|\f{\na^{k-1-l}v}{(1+|x|)^{l+2}}\|_{L^2}\Big)\|\na^{k+1}n\|_{L^2}\leq C\delta\|\na^{k+1}(n,v)\|_{L^2}^2.
\end{split}
\eeq
Similarly, it is easy to check that
\beq\label{i4}
\begin{split}
|I_4| \leq C\sum_{0\leq l\leq k-1}
\|(1+|x|)^{k-1-l}\na^{k-1-l}\bar\rho\|_{L^3}
\|\f{\na^{l+1}v}{(1+|x|)^{k-1-l}}\|_{L^6}\|\na^{k+1}n\|_{L^2}
\leq C\delta\|\na^{k+1}(n,v)\|_{L^2}^2.
\end{split}
\eeq
Substituting \eqref{i1}, \eqref{i2}, \eqref{i3} and \eqref{i4}
into \eqref{ks1}, it holds
\beq\label{s1}	
\begin{split}
\left|\int \na^k\widetilde{S}_1\cdot \na^kn dx \right|
\leq & C\delta\|\na^{k+1}(n,v)\|_{L^2}^2.
\end{split}
\eeq
Recall the definition of $\widetilde{S}_2$,
we obtain after integration by part that
\beno
\begin{split}
\int \na^k\widetilde{S}_2\cdot \na^kv dx
=&-\int \na^{k-1}\Big(\f{\mu_1\gamma}{\mu}v\cdot \na v \Big)
     \cdot \na^{k+1}v  dx
  -\int \na^{k-1}\Big\{ \wf(n+\bar\rho)\big(\mu_1\tri v+\mu_2\na\dive v\big)\Big\}
      \cdot \na^{k+1}v dx\\
&-\int \na^{k-1}\Big(\wg(n+\bar\rho) \na n \Big)\cdot \na^{k+1}v dx
-\int \na^{k-1}\Big(\h(n,\bar\rho)\na\bar\rho\Big)\cdot \na^{k+1}v dx\\
\overset{def}{=}&J_1+J_2+J_3+J_4.
\end{split}
\eeno
The Sobolev inequality and \eqref{nanak} yield directly
\beq
\begin{split}\label{j1}
|J_1| \leq C\sum_{0\leq l\leq k-1}
       \|\na^{k-1-l}v\|_{L^3}\|\na^{l+1}v\|_{L^6}\|\na^{k+1}v\|_{L^2}
		\leq C\delta \|\na^{k+1}v\|_{L^2}^2.
\end{split}
\eeq
Next, we begin to give the estimate for the term $J_2$.
Notice that for any $1 \le l \le k-1$,
\beq\label{widef}
\begin{split}
\na^l\wf(n+\bar\rho)=\text{a sum of products}~~\wf^{\gamma_1,\gamma_2,\cdots,\gamma_{i+j}}(n+\bar\rho)
\na^{\gamma_1}n\cdots\na^{\gamma_i}n\na^{\gamma_{i+1}}\bar\rho\cdots\na^{\gamma_{i+j}}\bar\rho
\end{split}
\eeq
with the functions $\wf^{\gamma_1,\gamma_2,\cdots,\gamma_{i+j}}(n+\bar\rho)$ are some derivatives of $\wf(n+\bar\rho)$ and $1\leq \gamma_{\be}\leq l$, $\be=1,2,\cdots,i+j$; $\gamma_1+\gamma_2+\cdots+\gamma_i=m$, $0\leq m\leq l$ and $\gamma_1+\gamma_2+\cdots+\gamma_{i+j}=l$.
Then $J_{2}$ is split up into three terms as follows:\beq\label{J2}
\begin{split}
|J_2|\leq&C\int|\wf(n+\bar\rho)||\na^{k+1}v|^2dx
       +C\sum_{1\leq l\leq k-1}\sum_{\gamma_1+\cdots+\gamma_j=l}
       \int|\na^{\gamma_1}\bar\rho|\cdots|\na^{\gamma_j}\bar\rho||\na^{k+1-l}v||\na^{k+1}v|dx\\
	   &+C\sum_{1\leq l\leq k-1}\sum_{\substack{\gamma_1+\cdots+\gamma_{i+j}=l
        \\\gamma_1+\cdots+\gamma_{i}=m\\1\leq m\leq l}}
        \int|\na^{\gamma_1}n|\cdots|\na^{\gamma_i}n|
         |\na^{\gamma_{i+1}}\bar\rho|\cdots|\na^{\gamma_{i+j}}\bar\rho|
         |\na^{k+1-l}v||\na^{k+1}v|dx\\
	 \overset{def}{=} & J_{21}+J_{22}+J_{23}.
\end{split}
\eeq
Then it is easy to verify that
\beno
\begin{split}
|J_{21}| \leq C\big(\|n\|_{L^\infty}+\|\bar\rho \|_{L^\infty}\big)\|\na^{k+1}v\|_{L^2}^2
		 \leq C\delta \|\na^{k+1}v\|_{L^2}^2.
\end{split}
\eeno
We address the second term $J_{22}$ by using Hardy inequality
\beno
\begin{split}
|J_{22}|
\leq& C\sum_{1\leq l\leq k-1}\sum_{\gamma_1+\cdots+\gamma_j=l}
\|(1+|x|)^{\gamma_1}\na^{\gamma_1}\bar\rho\|_{L^\infty}
\cdots\|(1+|x|)^{\gamma_j}\na^{\gamma_j}\bar\rho\|_{L^\infty}
\|\f{\na^{k+1-l}v}{(1+|x|)^{l}}\|_{L^2}\|\na^{k+1}v\|_{L^2}
\leq C\delta\|\na^{k+1}v\|_{L^2}^2.
\end{split}
\eeno
In view of the Sobolev interpolation inequality \eqref{Sobolev}
in Lemma \ref{inter} and Hardy inequality, we deduce that
\beno
\begin{split}
|J_{23}| \leq& C\sum_{1\leq l\leq k-1}\sum_{\substack{\gamma_1+\cdots+\gamma_{i+j}=l\\\gamma_1+\cdots+\gamma_{i}=m\\1\leq m\leq l}}\|\na^{\gamma_1}n\|_{L^\infty}\cdots\|\na^{\gamma_i}n\|_{L^\infty}\|(1+|x|)^{\gamma_{i+1}}\na^{\gamma_{i+1}}\bar\rho\|_{L^\infty}\cdots\|(1+|x|)^{\gamma_{i+j}}\na^{\gamma_{i+j}}\bar\rho\|_{L^\infty}\\
		&\times\|\f{\na^{k+1-l}v}{(1+|x|)^{l-m}}\|_{L^2}\|\na^{k+1}v\|_{L^2}\\
		\leq&C\delta\sum_{1\leq l\leq k-1}\sum_{\substack{\gamma_1+\cdots+\gamma_{i+j}=l\\\gamma_1+\cdots+\gamma_{i}=m\\1\leq m\leq l}}\|n\|_{L^2}^{1-\f{3+2\gamma_1}{2(k+1)}}\|\na^{k+1}n\|_{L^2}^{\f{3+2\gamma_1}{2(k+1)}}\cdots\|n\|_{L^2}^{1-\f{3+2\gamma_{i}}{2(k+1)}}\|\na^{k+1}n\|_{L^2}^{\f{3+2\gamma_i}{2(k+1)}}\|\na^{k+1-m}v\|_{L^2}\|\na^{k+1}v\|_{L^2}\\
		\leq&C\delta\sum_{1\leq l\leq k-1}\sum_{\substack{\gamma_1+\cdots+\gamma_{i+j}=l\\\gamma_1+\cdots+\gamma_{i}=m\\1\leq m\leq l}}\|n\|_{L^2}^{1-\f{3i+2m}{2(k+1)}}\|\na^{k+1}n\|_{L^2}^{\f{3i+2m}{2(k+1)}}\|\na^{\al}v\|_{L^2}^{\f{3i+2m}{2(k+1)}}\|\na^{k+1}v\|_{L^2}^{1-\f{3i+2m}{2(k+1)}}\|\na^{k+1}v\|_{L^2}\\
		\leq&C\delta\|\na^{k+1}(n,v)\|_{L^2}^{2},
\end{split}
\eeno
where
$$\al=\f{3i(k+1)}{3i+2m}\leq \f35(k+1)\leq \f35N\leq N,$$
provided $i\geq 1$, $m\geq 1$ and $i\leq m$.
Then, substituting the estimates of $J_{21}, J_{22}$ and $J_{23}$
into \eqref{J2}, we have
\beq\label{j2}
|J_2|\leq C\delta\|\na^{k+1}(n,v)\|_{L^2}^{2}.
\eeq
It is easy to check that
\beq\label{J3}
\begin{split}
|J_3|
\leq &C\int|n+\bar\rho||\na^{k}n||\na^{k+1}v|dx
       +C\sum_{1\leq l\leq k-1}\sum_{\gamma_1+\cdots+\gamma_j=l}
       \int|\na^{\gamma_1}\bar\rho|\cdots|\na^{\gamma_j}\bar\rho||\na^{k-l}n||\na^{k+1}v|dx\\
		&+C\sum_{1\leq l\leq k-1}
           \sum_{\substack{\gamma_1+\cdots+\gamma_{i+j}=l
                  \\\gamma_1+\cdots+\gamma_{i}=m\\1\leq m\leq l}}
         \int|\na^{\gamma_1}n|\cdots|\na^{\gamma_i}n||\na^{\gamma_{i+1}}\bar\rho|
         \cdots|\na^{\gamma_{i+j}}\bar\rho||\na^{k-l}n||\na^{k+1}v|dx\\
		\overset{def}{=}&J_{31}+J_{32}+J_{33}.
\end{split}
\eeq
By Sobolev inequality, it is easy to check that
\beno
\begin{split}
|J_{31}|
\leq C\|(n+\bar\rho)\|_{L^3}\|\na^{k}n\|_{L^6}\|\na^{k+1}v\|_{L^2}
\leq C\delta\|\na^{k+1}(n,v)\|_{L^2}^2.
\end{split}
\eeno
Similar to estimate \eqref{vl6}, we apply Hardy inequality to obtain
\beno
\begin{split}
|J_{32}|
\leq&C \sum_{1\leq l\leq k}\sum_{\gamma_1+\cdots+\gamma_j=l}\|(1+|x|)^{\gamma_1}\na^{\gamma_1}\bar \rho\|_{L^{3}}\|(1+|x|)^{\gamma_2}\na^{\gamma_2}\bar \rho\|_{L^{\infty}}\cdots\|(1+|x|)^{\gamma_j}\na^{\gamma_j}\bar \rho\|_{L^{\infty}}\|\f{\na^{k-l}n}{(1+|x|)^{l}}\|_{L^6}\|\na^{k+1}v\|_{L^2}\\
		\leq& C\delta\|\na^{k+1}(n,v)\|_{L^2}^2 .
\end{split}
\eeno
The application of Sobolev interpolation inequality \eqref{Sobolev}
in Lemma \ref{inter} yields directly
\beno
\begin{split}
|J_{33}|\leq& C\sum_{1\leq l\leq k-1}\sum_{\substack{\gamma_1+\cdots+\gamma_{i+j}=l\\\gamma_1+\cdots+\gamma_{i}=m\\1\leq m\leq l}}\|\na^{\gamma_1}n\|_{L^\infty}\cdots\|\na^{\gamma_i}n\|_{L^\infty}\|(1+|x|)^{\gamma_{i+1}}\na^{\gamma_{i+1}}\bar\rho\|_{L^\infty}\cdots\|(1+|x|)^{\gamma_{i+j}}\na^{\gamma_{i+j}}\bar\rho\|_{L^\infty}\\
		&\times\|\f{\na^{k-l}n}{(1+|x|)^{l-m}}\|_{L^2}\|\na^{k+1}v\|_{L^2}\\
		\leq&C\delta\sum_{1\leq l\leq k-1}\sum_{\substack{\gamma_1+\cdots+\gamma_{i+j}=l\\\gamma_1+\cdots+\gamma_{i}=m\\1\leq m\leq l}}\|n\|_{L^2}^{1-\f{3+2\gamma_1}{2(k+1)}}\|\na^{k+1}n\|_{L^2}^{\f{3+2\gamma_1}{2(k+1)}}\cdots\|n\|_{L^2}^{1-\f{3+2\gamma_{i}}{2(k+1)}}\|\na^{k+1}n\|_{L^2}^{\f{3+2\gamma_i}{2(k+1)}}\|\na^{k-m}n\|_{L^2}\|\na^{k+1}v\|_{L^2}\\
		\leq&C\delta\sum_{1\leq l\leq k-1}\sum_{\substack{\gamma_1+\cdots+\gamma_{i+j}=l\\\gamma_1+\cdots+\gamma_{i}=m\\1\leq m\leq l}}\|n\|_{L^2}^{1-\f{3i+2m}{2(k+1)}}\|\na^{k+1}n\|_{L^2}^{\f{3i+2m}{2(k+1)}}\|\na^{\al}n\|_{L^2}^{\f{3i+2m}{2(k+1)}}\|\na^{k+1}n\|_{L^2}^{1-\f{3i+2m}{2(k+1)}}\|\na^{k+1}v\|_{L^2}\\
		\leq&C\delta\|(\na^{k+1}n,\na^{k+2}v)\|_{L^2}^{2},
\end{split}
\eeno
where
$$\al=\f{(3i-2)(k+1)}{3i+2m}\leq \f35(k+1)\leq \f35N\leq N,$$
provided $i\geq 1$, $m\geq 1$ and $i\leq m$.
Substituting the estimates of term $J_{31}, J_{32}$
and $J_{33}$ into \eqref{J3}, it holds
\beq\label{j3}
|J_3|\leq C\delta \|\na^{k+1}(n,v)\|_{L^2}^{2}.
\eeq
	Before estimating the last term $J_4$, we want to show that $\h(n,\bar\rho)$ and its spatial derivatives can be controlled by a sum of products of some terms of $n$ and its derivatives.
	For this purpose, it follows from Taylor expansion that \beno
	\h(n,\bar\rho)=\f{p'(\rho_{\infty})+p''(\rho_{\infty})\rho_{\infty}}{(\bar\rho+\rho_{\infty})(n+\bar\rho+\rho_{\infty})}n+o(|n|).
	\eeno
	Using the fact that $p(\rho)$ is smooth in a neighborhood of $\rho_{\infty}$ with $p'(\rho_{\infty})>0$, one obtains that\beq\label{h}
	|\h(n,\bar\rho)|\leq C|n|.
	\eeq
	Next, let us to deal with the derivatives of $\h$. In view of the definition of $\h$ and $\wg$, it then follows from \eqref{hg-relation} that for any $l\geq 1$,\beq\label{wideh}
	\begin{split}
		&\na^l\h(n,\bar\rho)
		=\na^l\wg(n+\bar\rho)-\na^l\wg(\bar\rho)\\
		=&\text{a sum of products}~~\big\{\wg^{\gamma_1,\gamma_2,\cdots,\gamma_{i+j}}(n+\bar\rho)\na^{\gamma_1}n\cdots\na^{\gamma_i}n\na^{\gamma_{i+1}}\bar\rho\cdots\na^{\gamma_{i+j}}\bar\rho\\
		&\quad-\wg^{\gamma_1,\gamma_2,\cdots,\gamma_{i+j}}(\bar\rho)\na^{\gamma_1}\bar\rho\cdots\na^{\gamma_{i+j}}\bar\rho\big\}
	\end{split}
    \eeq
	with the functions $\wg^{\gamma_1,\gamma_2,\cdots,\gamma_{i+j}}$ are some derivatives of $\wg$ and $1\leq \gamma_{\be}\leq l$, $\be=1,2,\cdots,i+j$; $\gamma_1+\gamma_2+\cdots+\gamma_i=m$, $0\leq m\leq l$ and $\gamma_1+\gamma_2+\cdots+\gamma_{i+j}=l$.
	For the case that $m=0$, we use mean value theorem to find that there exists a $\xi$ between $\bar\rho$ and $n+\bar\rho$, such that\beno
	\begin{split}
		&\wg^{\gamma_1,\gamma_2,\cdots,\gamma_{j}}(n+\bar\rho)\na^{\gamma_1}\bar\rho\cdots\na^{\gamma_{j}}\bar\rho-\wg^{\gamma_1,\gamma_2,\cdots,\gamma_{j}}(\bar\rho)\na^{\gamma_1}\bar\rho\cdots\na^{\gamma_{j}}\bar\rho
		=\wg^{(\gamma_1,\gamma_2,\cdots,\gamma_{j})+1}(\xi)n\na^{\gamma_1}\bar\rho\cdots\na^{\gamma_{j}}\bar\rho,
	\end{split}
	\eeno
    which, together with \eqref{wideh}, yields the following estimate
    \beq\label{h2}
	\begin{split}
		|\na^l\h(n,\bar\rho)|\leq C|n|\sum_{\gamma_1+\cdots+\gamma_{j}=l}|\na^{\gamma_{1}}\bar\rho|\cdots|\na^{\gamma_{j}}\bar\rho|+ C\sum_{\substack{\gamma_1+\cdots+\gamma_{i+j}=l\\\gamma_1+\cdots+\gamma_{i}=m\\1\leq m\leq l}}|\na^{\gamma_1}n|\cdots|\na^{\gamma_i}n||\na^{\gamma_{i+1}}\bar\rho|\cdots|\na^{\gamma_{i+j}}\bar\rho|.
	\end{split}
	\eeq
	This, together with the estimate \eqref{h}, gives\beq\label{j4}
	\begin{split}
		|J_4|\leq&C\int|\h(n,\bar\rho)||\na^{k}\bar\rho||\na^{k+1}v|dx+C\sum_{1\leq l\leq k-1}\int|\na^{l}\h(n,\bar\rho)||\na^{k-l}\bar\rho||\na^{k+1}v|dx\\
		\leq &C\int|n||\na^{k}\bar\rho||\na^{k+1}v|dx+C\sum_{1\leq l\leq k-1}\sum_{\gamma_1+\cdots+\gamma_{j}=l}\int|n||\na^{\gamma_1}\bar\rho|\cdots|\na^{\gamma_j}\bar\rho||\na^{k-l}\bar\rho||\na^{k+1}v|dx\\
		&\quad+C\sum_{1\leq l\leq k-1}\sum_{\substack{\gamma_1+\cdots+\gamma_{i+j}=l\\\gamma_1+\cdots+\gamma_{i}=m\\1\leq m\leq l}}\int|\na^{\gamma_1}n|\cdots|\na^{\gamma_i}n||\na^{\gamma_{i+1}}\bar\rho|\cdots|\na^{\gamma_{i+j}}\bar\rho||\na^{k-l}\bar\rho||\na^{k+1}v|dx\\
		\overset{def}{=}&J_{41}+J_{42}+J_{43}.
	\end{split}
	\eeq
	By virtue of Sobolev and Hardy inequalities, it is easy to deduce
	\beq\label{j41}
	\begin{split}
		J_{41}+J_{42}\leq& C \|\f{n}{(1+|x|)^{k}}\|_{L^6}\Big(\|(1+|x|)^{k}\na^{k}\bar\rho\|_{L^3}+\sum_{1\leq l\leq k-1}\sum_{\gamma_1+\cdots+\gamma_{j}=l}\|(1+|x|)^{\gamma_{1}}\na^{\gamma_{1}}\bar\rho\|_{L^\infty}\\
		&\quad\cdots\|(1+|x|)^{\gamma_{j}}\na^{\gamma_{j}}\bar\rho\|_{L^\infty}\|(1+|x|)^{k-l}\na^{k-l}\bar\rho\|_{L^3}\Big)\|\na^{k+1}v\|_{L^2}\\
		\leq&C\delta\Big(\|\f{\na n}{(1+|x|)^{k}}\|_{L^2}+\|\f{n}{(1+|x|)^{k+1}}\|_{L^2}\Big)\|\na^{k+1}v\|_{L^2}\\
		\leq& C\delta \|\na^{k+1}(n,v)\|_{L^2}^2.
	\end{split}
	\eeq
	It then follows from Sobolev inequality that
	\beq\label{J43}
	\begin{split}
		J_{43}\leq& C\sum_{1\leq l\leq k-1}\sum_{\substack{\gamma_1+\cdots+\gamma_{1+j}=l\\\gamma_1=m\\1\leq m\leq l}}\|\f{\na^{\gamma_1}n}{(1+|x|)^{k-m}}\|_{L^6}\|(1+|x|)^{\gamma_{2}}\na^{\gamma_{2}}\bar\rho\|_{L^\infty}\cdots\|(1+|x|)^{\gamma_{1+j}}\na^{\gamma_{1+j}}\bar\rho\|_{L^\infty}\\
		&\times\|(1+|x|)^{k+1-l}\na^{k-l}\bar\rho\|_{L^3}\|\na^{k+1}v\|_{L^2}+C\sum_{1\leq l\leq k}\sum_{\substack{\gamma_1+\cdots+\gamma_{i+j}=l\\\gamma_1+\cdots+\gamma_{i}=m\\1\leq m\leq l,~i\geq2}}\|\f{\na^{\gamma_1}n}{(1+|x|)^{k-m}}\|_{L^6}\|\na^{\gamma_2}n\|_{L^\infty}\cdots\\
		&\times\|\na^{\gamma_i}n\|_{L^\infty}\|(1+|x|)^{\gamma_{i+1}}\na^{\gamma_{i+1}}\bar\rho\|_{L^\infty}\cdots\|(1+|x|)^{\gamma_{i+j}}\na^{\gamma_{i+j}}\bar\rho\|_{L^\infty}\|(1+|x|)^{k+1-l}\na^{k-l}\bar\rho\|_{L^3}\|\na^{k+1}v\|_{L^2}\\
		\overset{def}{=}&J_{431}+J_{432}.
	\end{split}
	\eeq
	Thanks to Hardy inequality, one can deduce that
	\beno
	\begin{split}
		J_{431}\leq& C\delta\sum_{1\leq l\leq k-1}\sum_{\substack{\gamma_1+\cdots+\gamma_{1+j}=l\\\gamma_1=m\\1\leq m\leq l}}\Big(\|\f{\na^{\gamma_1+1}n}{(1+|x|)^{k-m}}\|_{L^2}+\|\f{\na^{\gamma_1}n}{(1+|x|)^{k-m+1}}\|_{L^2}\Big)\|\na^{k+1}v\|_{L^2}\\
		\leq& C\delta\|\na^{k+1}n\|_{L^2}\|\na^{k+1}v\|_{L^2}\leq C\delta\|\na^{k+1}(n,v)\|_{L^2}^2.
	\end{split}
	\eeno
	With the help of Sobolev interpolation inequality \eqref{Sobolev}
	in Lemma \ref{inter} and Hardy ineuqlity, one obtains
	\beno
	\begin{split}
		J_{432}\leq&C\delta\sum_{1\leq l\leq k-1}\sum_{\substack{\gamma_1+\cdots+\gamma_{i+j}=l\\\gamma_1+\cdots+\gamma_{i}=m\\1\leq m\leq l,~i\geq2}}\|\na^{k+1-m+\gamma_1}n\|_{L^2}\|n\|_{L^2}^{1-\f{3+2\gamma_2}{2(k+1)}}\|\na^{k+1}n\|_{L^2}^{\f{3+2\gamma_2}{2(k+1)}}\cdots\|n\|_{L^2}^{1-\f{3+2\gamma_i}{2(k+1)}}\|\na^{k+1}n\|_{L^2}^{\f{3+2\gamma_i}{2(k+1)}}\|\na^{k+1}v\|_{L^2}\\
		\leq&C\delta\sum_{1\leq l\leq k-1}\sum_{\substack{\gamma_1+\cdots+\gamma_{i+j}=l\\\gamma_1+\cdots+\gamma_{i}=m\\1\leq m\leq l,~i\geq2}}\|\na^{\al}n\|_{L^2}^{\theta}\|\na^{k+1}n\|_{L^2}^{1-\theta}\|n\|_{L^2}^{1-\theta}\|\na^{k+1}n\|_{L^2}^{\theta}\|\na^{k+1}v\|_{L^2}
		\leq C\delta\|\na^{k+1}(n,v)\|_{L^2}^2,
	\end{split}
	\eeno
	where
	$$\theta=\f{3(i-1)+2(m-\gamma_1)}{2(k+1)},\quad\al=\f{3(i-1)(k+1)}{3(i-1)+2(m-\gamma_1)}\leq \f{3}{5}(k+1)\leq \f35 N\leq N,$$
	provided that $i\geq2$ and $i-1\leq m-\gamma_1$.
    Inserting the estimates of $J_{431}$ and $J_{432}$ into $\eqref{J43}$, it follows immediately
	\beq\label{j43}
	J_{43}\leq C\delta\|\na^{k+1}(n,v)\|_{L^2}^2.
	\eeq
	Thus, substituting \eqref{j41} and \eqref{j43} into \eqref{j4}, we deduce that\beno
	\begin{split}
		|J_4|\leq C\delta\|\na^{k+1}(n,v)\|_{L^2}^2,
	\end{split}
	\eeno
	which, together with \eqref{j1}, \eqref{j2} and \eqref{j3}, gives
	\beq\label{s2}
	\begin{split}
		\int \na^k\widetilde{S}_2\cdot \na^kv dx
		\leq C\delta\|\na^{k+1}(n,v)\|_{L^2}^2.
	\end{split}
	\eeq
	We finally utilize \eqref{s1} and \eqref{s2} in \eqref{ehk}, to obtain \eqref{en1} directly.
	Thus, we complete the proof of this lemma.
\end{proof}

We then derive the energy estimate for $N-th$ order spatial derivative of solution.
\begin{lemm}\label{enn}
	Under the assumptions in Theorem \ref{them3}, we have
	\beq\label{en2}
	\f{d}{dt}\|\na^{N}(n,v)\|_{L^2}^2+\|\na^{N+1}v\|_{L^2}^2\leq C\delta \|(\na^{N}n,\na^{N+1}v)\|_{L^2}^2,
	\eeq
	where $C$ is a positive constant independent of time.
\end{lemm}
\begin{proof}
Applying differential operator $\na^{N}$ to $\eqref{ns5}_1$ and $\eqref{ns5}_2$,
multiplying the resulting equations by $\na^{N}n$ and $\na^{N}v$, respectively,
and integrating over $\mathbb{R}^3$, it holds
\beq\label{ehk1}
\begin{split}
		\f12\f{d}{dt}\|\na^{N}(n,v)\|_{L^2}^2+\mu_1\|\na^{N+1} v\|_{L^2}^2+\mu_2\|\na^{N}\dive v\|_{L^2}^2
		= \int \na^{N}\widetilde{S}_1\cdot \na^{N}n dx+\int \na^{N}\widetilde{S}_2\cdot \na^{N}v dx.
\end{split}
\eeq
Now we estimate two terms on the right handside of \eqref{ehk1} separately.
In view of the definition of $\widetilde{S}_1$, we have
\beq\label{ks2}
\begin{split}
\int \na^{N}\widetilde{S}_1\cdot \na^{N}n dx
=
&C\int \na^{N}( v\cdot\na n )\cdot \na^{N}n dx
  +C\int \na^{N}( n\dive v)\cdot \na^{N}n dx\\
&  +C\int \na^{N}( v\cdot\na\bar\rho )\cdot \na^{N}n dx
  +C\int \na^{N}( \bar\rho \dive v)\cdot \na^{N}n dx\\
\overset{def}{=} &K_1+K_2+K_3+K_4.
\end{split}
\eeq
Sobolev inequality and integration by parts yield\beno
\begin{split}
K_1=&C\sum_{0\leq l\leq N}\int\na^{l+1}n\na^{N-l}v\na^{N}n dx\\
=&C\int\na n\na^{N}v\na^{N}ndx
  +C\sum_{1\leq l\leq N-2}\int\na^{l+1} n\na^{N-l}v\na^{N}n dx
  +C\int \na^{N}n\na v\na^{N}ndx-C\int \dive v|\na^{N}n|^2dx\\
\leq&C\|\na v\|_{L^\infty}\|\na^{N}n\|_{L^2}^2
     +C\|\na n\|_{L^3}\|\na^{N}v\|_{L^6}\|\na^{N}n\|_{L^2}
     +C\sum_{2\leq l\leq N-1}\|\na^ln\|_{L^6}\|\na^{N+1-l}v\|_{L^3}\|\na^{N}n\|_{L^2}\\
\leq&C \|\na^2 v\|_{H^1}\|\na^{N}n\|_{L^2}^2
      +C\|\na n\|_{H^1}\|\na^{N+1}v\|_{L^2}\|\na^{N}n\|_{L^2}
      +C\sum_{2\leq l\leq N-1}\|\na^ln\|_{L^6}\|\na^{N+1-l}v\|_{L^3}\|\na^{N}n\|_{L^2}.
\end{split}
\eeno
Using the Sobolev interpolation inequality \eqref{Sobolev}
in Lemma \ref{inter}, the third term on the right handside
of above inequality can be estimated as follows
\beq\label{lnv}
\begin{aligned}
&\|\na^ln\|_{L^6}\|\na^{N+1-l}v\|_{L^3}\|\na^{N}n\|_{L^2}\\
\leq &C\|n\|_{L^2}^{1-\f{l+1}{N}}\|\na^{N}n\|_{L^2}^{\f{l+1}{N}}
      \|\na^\al v\|_{L^2}^{\f{l+1}{N}}
      \|\na^{N+1}v\|_{L^2}^{1-\f{l+1}{N}}\|\na^{N}n\|_{L^2}\\
\leq &C\delta \|(\na^{N}n,\na^{N+1}v)\|_{L^2}^2,
\end{aligned}
\eeq
here we denote $\al$ that
$$\al=\f{3N}{2(l+1)}+1\leq \f{N}{2}+1,$$
since $l\geq 2$.
Thus, we have
\beno
|K_1| \leq C\delta \|(\na^{N}n,\na^{N+1}v)\|_{L^2}^2.
\eeno
It follows in a similar way to $K_1$ that
\beno
\begin{split}
K_2=   &C\sum_{0\leq l\leq N}\int\na^l n\na^{N+1-l}v\na^{N}n dx\\
  \leq &C\Big(\|n\|_{L^\infty}\|\na^{N+1} v\|_{L^2}+\|\na n\|_{L^3}\|\na^{N} v\|_{L^6}
          +\sum_{2\leq l\leq N-1}\|\na^ln\|_{L^6}\|\na^{N+1-l}v\|_{L^3}
          +\|\na v\|_{L^\infty}\|\na^{N}n\|_{L^2}\Big)\|\na^{N}n\|_{L^2}\\
  \leq&C\delta \|(\na^{N}n,\na^{N+1}v)\|_{L^2}^2,
\end{split}
\eeno
where we have used \eqref{lnv} in the last inequality above.
Thanks to Hardy inequality, we compute
\beno
\begin{split}
|K_3| \leq& C\sum_{0\leq l\leq N}\|(1+|x|)^{l+1}\na^{l+1}\bar\rho\|_{L^{\infty}}
             \|\f{\na^{N-l}v}{(1+|x|)^{l+1}}\|_{L^2}\|\na^{N}n\|_{L^2}
		\leq C\delta \|(\na^{N}n,\na^{N+1}v)\|_{L^2}^2,\\
|K_4| \leq& C\sum_{0\leq l\leq N}\|(1+|x|)^{l}\na^{l}\bar\rho\|_{L^{\infty}}
            \|\f{\na^{N+1-l}v}{(1+|x|)^{l}}\|_{L^2}\|\na^{N}n\|_{L^2}
		\leq C\delta \|(\na^{N}n,\na^{N+1}v)\|_{L^2}^2.
\end{split}
\eeno
Substituting the estimates of term from $K_1$ to $K_4$ into \eqref{ks2}, we have
\beq\label{ws11}
\begin{split}
\left|\int \na^{N}\widetilde{S}_1\cdot \na^{N}n dx \right|
\leq&C \delta \|(\na^{N}n,\na^{N+1}v)\|_{L^2}^2.
\end{split}
\eeq
Remebering the definition of $\widetilde{S}_2$ and integrating by part, we find
\beq\label{ks22}
\begin{split}
&\int \na^{N}\widetilde{S}_2\cdot \na^{N}v dx\\
=&\int \na^{N-1}\Big(\f{\mu_1\gamma}{\mu}v\cdot \na v \Big)\cdot \na^{N+1}v dx
+\int \na^{N-1}\Big\{\wf(n+\bar\rho)\big(\mu_1\tri v+\mu_2\na\dive v\big) \Big\}\cdot \na^{N+1}v dx\\
&+\int \na^{N-1}\Big(\wg(n+\bar\rho) \na n \Big)\cdot \na^{N+1}v dx
+\int \na^{N-1}\Big(\h(n,\bar\rho)\na\bar\rho\Big)\cdot \na^{N+1}v dx\\
\overset{def}{=}&L_1+L_2+L_3+L_4.
\end{split}
\eeq
According to the Sobolev interpolation inequality \eqref{Sobolev}
in Lemma \ref{inter}, we deduce that
\beq\label{l1}
\begin{split}
|L_1|
\leq  &C\sum_{0\leq l\leq N-1}\|\na^{N-1-l}v\|_{L^3}\|\na^{l+1}v\|_{L^6}\|\na^{N+1}v\|_{L^2}\\
\leq  &C\sum_{0\leq l\leq N-1}\|\na^{\al}v\|_{L^2}^{\f{l+2}{N+1}}
      \|\na^{N+1}v\|_{L^2}^{1-\f{l+2}{N+1}}
      \|v\|_{L^2}^{1-\f{l+2}{N+1}}\|\na^{N+1}v\|_{L^2}^{\f{l+2}{N+1}}
      \|\na^{N+1}v\|_{L^2}\\
\leq &C\delta \|\na^{N+1}v\|_{L^2}^2,
\end{split}
\eeq
where $\al$ is given by
$$
\al=\f{N+1}{2(l+2)}\leq \f{N+1}{4}.
$$
Next, we estimate the term $L_2$ and
divide it into two terms as follows
\beq\label{L2}
|L_2|
\leq C\int|\wf(n+\bar\rho)||\na^{N+1}v|^2dx
+C\sum_{1\leq l\leq N-1}\int|\na^{l}\wf(n+\bar\rho)||\na^{N+1-l}v||\na^{N+1}v|dx
\overset{def}{=}L_{21}+L_{22}.
\eeq
	By Sobolev inequality, one can deduce directly that\beq\label{l21}
	\begin{split}
		L_{21}\leq C\|(n+\bar\rho)\|_{L^\infty}\|\na^{N+1}v\|_{L^2}^2\leq C\delta\|\na^{N+1}v\|_{L^2}^2.
	\end{split}
	\eeq
It follows from \eqref{widef} that $L_{22}$ can be
split up into three terms as follows:
\beq\label{l22}
\begin{split}
L_{22} \leq
&C\sum_{1\leq l\leq N-1}\sum_{\gamma_1+\cdots+\gamma_i=l}
 \int|\na^{\gamma_1}n|\cdots|\na^{\gamma_i}n||\na^{N+1-l}v||\na^{N+1}v|dx\\
&+C\sum_{1\leq l\leq N-1}\sum_{\gamma_1+\cdots+\gamma_j=l}
 \int|\na^{\gamma_1}\bar\rho|\cdots|\na^{\gamma_j}\bar\rho||\na^{N+1-l}v||\na^{N+1}v|dx\\
&+C\sum_{1\leq l\leq N-1}
\sum_{\substack{\gamma_1+\cdots+\gamma_{i+j}=l\\\gamma_1+\cdots+\gamma_{i}=m\\1\leq m\leq l-1}}
\int|\na^{\gamma_1}n|\cdots|\na^{\gamma_i}n|
    |\na^{\gamma_{i+1}}\bar\rho|\cdots|\na^{\gamma_{i+j}}\bar\rho|
    |\na^{N+1-l}v||\na^{N+1}v|dx\\
\overset{def}{=} & L_{221}+L_{222}+L_{223}.
\end{split}
\eeq
We then divide $L_{221}$ into two terms as follows:
\beno
\begin{aligned}
L_{221}
\leq & C\sum_{1\leq l\leq N-1}
            \int|\na^{l}n||\na^{N+1-l}v||\na^{N+1}v|dx\\
     &+C\sum_{1\leq l\leq N-1}\sum_{\substack{\gamma_1+\cdots+\gamma_i=l\\i\geq2}}
       \int|\na^{\gamma_1}n|\cdots|\na^{\gamma_i}n||\na^{N+1-l}v||\na^{N+1}v|dx\\
\overset{def}{=}&L_{2211}+L_{2212}.
\end{aligned}
\eeno
Using the Sobolev inequality and estimate \eqref{lnv}, it holds
\beno
\begin{split}
L_{2211}
\leq& C\Big(\|\na n\|_{L^3}\|\na^{N}v\|_{L^6}
       +\sum_{2\leq l\leq N-1}\|\na^{l}n\|_{L^6}\|\na^{N+1-l}v\|_{L^3}\Big)
       |\na^{N+1}v\|_{L^2}\\
\leq& C\Big(\|\na n\|_{H^1}\|\na^{N+1}v\|_{L^2}
       +\sum_{2\leq l\leq N-1}\|n\|_{L^2}^{1-\f{l+1}{N}}\|\na^{N}n\|_{L^2}^{\f{l+1}{N}}
       \|\na^\al v\|_{L^2}^{\f{l+1}{N}}\|\na^{N+1}v\|_{L^2}^{1-\f{l+1}{N}}\Big)
       \|\na^{N+1}v\|_{L^2}\\
\leq& C\delta \|(\na^{N}n,\na^{N+1}v)\|_{L^2}^2.
\end{split}
\eeno
Since $i\geq2$ in term $L_{2212}$, then $\gamma_\be\leq N-2$ for $\be=1,2,\cdots,i$.
Thus, we can use the Sobolev interpolation inequality \eqref{Sobolev} in
Lemma \ref{inter} to find
\beno
	\begin{split}
		L_{2212}\leq&C\sum_{1\leq l\leq N-1}\sum_{\substack{\gamma_1+\cdots+\gamma_i=l\\i\geq2}}\|\na^{\gamma_1}n\|_{L^\infty}\cdots\|\na^{\gamma_i}n\|_{L^\infty}\|\na^{N+1-l}v\|_{L^2}\|\na^{N+1}v\|_{L^2}\\
		\leq&C\delta\sum_{1\leq l\leq N-1}\sum_{\substack{\gamma_1+\cdots+\gamma_i=l\\i\geq2}}\|n\|_{L^2}^{1-\f{3+2\gamma_1}{2N}}\|\na^{N}n\|_{L^2}^{\f{3+2\gamma_1}{2N}}\cdots\|n\|_{L^2}^{1-\f{3+2\gamma_{i}}{2N}}\|\na^{N}n\|_{L^2}^{\f{3+2\gamma_i}{2N}}\|\na^{N+1-l}v\|_{L^2}\|\na^{N+1}v\|_{L^2}\\
		\leq&C\sum_{1\leq l\leq N-1}\sum_{\substack{\gamma_1+\cdots+\gamma_i=l\\i\geq2}}\|n\|_{L^2}^{1-\f{3i+2l}{2N}}\|\na^{N}n\|_{L^2}^{\f{3i+2l}{2N}}\|\na^{\al}v\|_{L^2}^{\f{3i+2l}{2N}}\|\na^{N+1}v\|_{L^2}^{1-\f{3i+2l}{2N}}\|\na^{N+1}v\|_{L^2}\\
		\leq&C\delta\|(\na^{N}n,\na^{N+1}v)\|_{L^2}^{2},
	\end{split}
	\eeno
	where $\al$ is given by
	$$\al=1+\f{3iN}{3i+2l}\leq 1+\f35N\leq N,$$
	provided $i\geq 1$, $m\geq 1$, $i\leq m$ and $N\geq3$. Therefore, it follows from the estimates above that
\beno
L_{221}\leq C\delta\|(\na^{N}n,\na^{N+1}v)\|_{L^2}^{2}.
\eeno
The application of Hardy inequality yields directly
\beno
\begin{split}
L_{222}
\leq &C \sum_{1\leq l\leq N-1}\sum_{\gamma_1+\cdots+\gamma_j=l}
        \|(1+|x|)^{\gamma_1}\na^{\gamma_1}\bar \rho\|_{L^{\infty}}\cdots\|(1+|x|)^{\gamma_j}\na^{\gamma_j}\bar \rho\|_{L^{\infty}}
        \|\f{\na^{N+1-l}v}{(1+|x|)^{l}}\|_{L^2}\|\na^{N+1}v\|_{L^2}\\
\leq &  C\delta \|\na^{N+1}v\|_{L^2}^2.
\end{split}
\eeno
Using Sobolev interpolation inequality \eqref{Sobolev} in
Lemma \ref{inter} again, we have
\beno
\begin{split}
L_{223}
\leq& C\sum_{1\leq l\leq N-1}\sum_{\substack{\gamma_1+\cdots+\gamma_{i+j}=l\\\gamma_1+\cdots+\gamma_{i}=m\\1\leq m\leq l-1}}\|\na^{\gamma_1}n\|_{L^\infty}\cdots\|\na^{\gamma_i}n\|_{L^\infty}\|(1+|x|)^{\gamma_{i+1}}\na^{\gamma_{i+1}}\bar\rho\|_{L^\infty}\cdots\|(1+|x|)^{\gamma_{i+j}}\na^{\gamma_{i+j}}\bar\rho\|_{L^\infty}\\
		&\times\|\f{\na^{N+1-l}v}{(1+|x|)^{l-m}}\|_{L^2}\|\na^{N+1}v\|_{L^2}\\
		\leq&C\delta\sum_{1\leq l\leq N-1}\sum_{\substack{\gamma_1+\cdots+\gamma_{i+j}=l\\\gamma_1+\cdots+\gamma_{i}=m\\1\leq m\leq l-1}}\|n\|_{L^2}^{1-\f{3+2\gamma_1}{2N}}\|\na^{N}n\|_{L^2}^{\f{3+2\gamma_1}{2N}}\cdots\|n\|_{L^2}^{1-\f{3+2\gamma_{i}}{2N}}\|\na^{N}n\|_{L^2}^{\f{3+2\gamma_i}{2N}}\|\na^{N+1-m}v\|_{L^2}\|\na^{N+1}v\|_{L^2}\\
		\leq&C\delta\sum_{1\leq l\leq N-1}\sum_{\substack{\gamma_1+\cdots+\gamma_{i+j}=l\\\gamma_1+\cdots+\gamma_{i}=m\\1\leq m\leq l-1}}\|n\|_{L^2}^{1-\f{3i+2m}{2N}}\|\na^{N}n\|_{L^2}^{\f{3i+2m}{2N}}\|\na^{\al}v\|_{L^2}^{\f{3i+2m}{2N}}\|\na^{N+1}v\|_{L^2}^{1-\f{3i+2m}{2N}}\|\na^{N+1}v\|_{L^2}\\
		\leq&C\delta\|(\na^{N}n,\na^{N+1}v)\|_{L^2}^{2},
	\end{split}
	\eeno
where $\al$ is given by
$$\al=1+\f{3iN}{3i+2m}\leq 1+\f35N\leq N,$$
provided $i\geq 1$, $m\geq 1$ and $i\leq m$.
Then, substituting estimates of terms $L_{221}, L_{222}$ and $L_{223}$
into \eqref{l22}, we have
\beq\label{L22}
\begin{split}
L_{22}\leq C\delta \|(\na^{N}n,\na^{N+1}v)\|_{L^2}^{2}.
\end{split}
\eeq
The combination of estimate \eqref{L2}, \eqref{l21} and \eqref{L22} yields directly
\beq\label{l2}
|L_2|\leq C\delta \|(\na^{N}n,\na^{N+1}v)\|_{L^2}^{2}.
\eeq
Now we give the estimate for the term $L_3$. Indeed, it is easy to check that
\beq\label{L3}
\begin{split}
|L_3|
   &\leq C\int|\wg(n+\bar\rho)||\na^{N}n||\na^{N+1}v|dx
     +C\sum_{1\leq l\leq N-1}\sum_{\gamma_1+\cdots+\gamma_i=l}
     \int|\na^{\gamma_1}n|\cdots|\na^{\gamma_i}n||\na^{N-l}n||\na^{k+2}v|dx\\
	&\quad+C\sum_{1\leq l\leq N-1}\sum_{\gamma_1+\cdots+\gamma_j=l}
       \int|\na^{\gamma_1}\bar\rho||\na^{N-l}n|\cdots
       |\na^{\gamma_j}\bar\rho||\na^{N-l}n||\na^{N+1}v|dx\\
	&\quad+C\sum_{1\leq l\leq N-1}\sum_{\substack{\gamma_1+\cdots+\gamma_{i+j}=l\\
            \gamma_1+\cdots+\gamma_{i}=m\\1\leq m\leq l-1}}
           \int|\na^{\gamma_1}n|\cdots|\na^{\gamma_i}n|
               |\na^{\gamma_{i+1}}\bar\rho|\cdots|\na^{\gamma_{i+j}}\bar\rho|
               |\na^{N-l}n||\na^{N+1}v|dx\\
	&\overset{def}{=}L_{31}+L_{32}+L_{33}+L_{34}.
\end{split}
\eeq
By Sobolev inequality, it is easy to check that
\beq\label{L31}
\begin{split}
L_{31}
\leq C\|(n+\bar\rho)\|_{L^\infty}\|\na^{N}n\|_{L^2}\|\na^{N+1}v\|_{L^2}
\leq C\delta\|(\na^{N}n,\na^{N+1}v)\|_{L^2}^2.
\end{split}
\eeq
For $L_{32}$, we divide it into the following two terms: \beno
\begin{split}
L_{32}\leq& C\sum_{1\leq l\leq N-1}\int|\na^{l}n||\na^{N-l}n||\na^{N+1}v|dx\\
   &+C\sum_{1\leq l\leq N-1}\sum_{\substack{\gamma_1+\cdots+\gamma_i=l\\i\geq2}}
     \int|\na^{\gamma_1}n|\cdots|\na^{\gamma_i}n||\na^{N-l}n||\na^{N+1}v|dx\\
	\overset{def}{=}&L_{321}+L_{322}.
	\end{split}
	\eeno
By Sobolev inequality, it holds
\beno
\begin{split}
L_{321}
\leq& C\Big(\|\na n\|_{L^3}\|\na^{N-1}n\|_{L^6}
           +\sum_{2\leq l\leq N-1}\|\na^{l}n\|_{L^6}\|\na^{N-l}n\|_{L^3}\Big)
          \|\na^{N+1}v\|_{L^2}\\
\leq& C\Big(\|\na n\|_{L^3}\|\na^{N-1}n\|_{L^6}
           +\sum_{2\leq l\leq N-1}
           \|n\|_{L^2}^{1-\f{l+1}{N}}
           \|\na^{N}n\|_{L^2}^{\f{l+1}{N}}
           \|\na^{\al}n\|_{L^2}^{\f{l+1}{N}}
           \|\na^{N}n\|_{L^2}^{1-\f{l+1}{N}}\Big)
          \|\na^{N+1}v\|_{L^2}\\
\leq& C\delta \|(\na^{N}n,\na^{N+1}v)\|_{L^2}^2,
\end{split}
\eeno
where $\al$ is defined by $\al=\f{3N}{2(l+1)}\leq \f{N}{2}$.
For the term $L_{322}$, the fact $i\geq2$ implies
$\gamma_\be+2\leq l+1\leq N$ for $\be=1,2,\cdots,i$.
Then, it follows from Sobolev interpolation inequality \eqref{Sobolev} in Lemma \ref{inter} that
\beno
	\begin{split}
		L_{322}\leq&C\sum_{1\leq l\leq N-1}\sum_{\substack{\gamma_1+\cdots+\gamma_{i}=l\\i\geq2}}\|\na^{\gamma_1}n\|_{L^\infty}\cdots\|\na^{\gamma_i}n\|_{L^\infty}\|\f{\na^{N-l}n}{(1+|x|)^{l-m}}\|_{L^2}\|\na^{N+1}v\|_{L^2}\\
		\leq&C\sum_{1\leq l\leq N-1}\sum_{\substack{\gamma_1+\cdots+\gamma_{i}=l\\i\geq2}}\|n\|_{L^2}^{1-\f{3+2\gamma_1}{2N}}\|\na^{N}n\|_{L^2}^{\f{3+2\gamma_1}{2N}}\cdots\|n\|_{L^2}^{1-\f{3+2\gamma_{i}}{2N}}\|\na^{N}n\|_{L^2}^{\f{3+2\gamma_i}{2N}}\|\na^{N-m}n\|_{L^2}\|\na^{N+1}v\|_{L^2}\\
		\leq&C\delta\sum_{1\leq l\leq N-1}\sum_{\substack{\gamma_1+\cdots+\gamma_{i}=l\\i\geq2}}\|n\|_{L^2}^{1-\theta}\|\na^{N}n\|_{L^2}^{\theta}\|\na^{\al}n\|_{L^2}^{\theta}\|\na^{N}n\|_{L^2}^{1-\theta}\|\na^{N+1}v\|_{L^2}\\
		\leq&C\delta\|(\na^{N}n,\na^{N+1}v)\|_{L^2}^{2},
	\end{split}
	\eeno
where we denote
$$\theta=\f{3i+2l}{2N},\quad\al=\f{3iN}{3i+2l}\leq \f35N\leq N,$$
provided $l\geq1$, $i\leq l$ and $i\geq2$.
Then, the combination of estimates of terms $L_{321}$ and $L_{322}$ implies directly
\beq\label{L32}
	L_{32}\leq C\delta\|(\na^{N}n,\na^{N+1}v)\|_{L^2}^{2}.
\eeq
    For the term $L_{33}$, by Hardy inequality, we obtain\beq\label{L33}
	\begin{split}
		L_{33}
		\leq&C \sum_{1\leq l\leq N-1}\sum_{\gamma_1+\cdots+\gamma_j=l}\|(1+|x|)^{\gamma_1}\na^{\gamma_1}\bar \rho\|_{L^{\infty}}\cdots\|(1+|x|)^{\gamma_j}\na^{\gamma_j}\bar \rho\|_{L^{\infty}}\|\f{\na^{N-l}n}{(1+|x|)^{l}}\|_{L^2}\|\na^{N+1}v\|_{L^2}\\
		\leq& C\delta \|(\na^{N}n,\na^{N+1}v)\|_{L^2}^2.
	\end{split}
	\eeq
	In view of Sobolev interpolation inequality \eqref{Sobolev} in
	Lemma \ref{inter} and Hardy inequality, one deduces that\beq\label{L34}
	\begin{split}
		L_{34}\leq& C\sum_{1\leq l\leq N-1}\sum_{\substack{\gamma_1+\cdots+\gamma_{i+j}=l\\\gamma_1+\cdots+\gamma_{i}=m\\1\leq m\leq l-1}}\|\na^{\gamma_1}n\|_{L^\infty}\cdots\|\na^{\gamma_i}n\|_{L^\infty}\|(1+|x|)^{\gamma_{i+1}}\na^{\gamma_{i+1}}\bar\rho\|_{L^\infty}\cdots\|(1+|x|)^{\gamma_{i+j}}\na^{\gamma_{i+j}}\bar\rho\|_{L^\infty}\\
		&\times\|\f{\na^{N-l}n}{(1+|x|)^{l-m}}\|_{L^2}\|\na^{N+1}v\|_{L^2}\\
		\leq&C\delta\sum_{1\leq l\leq N-1}\sum_{\substack{\gamma_1+\cdots+\gamma_{i+j}=l\\\gamma_1+\cdots+\gamma_{i}=m\\1\leq m\leq l-1}}\|n\|_{L^2}^{1-\f{3+2\gamma_1}{2N}}\|\na^{N}n\|_{L^2}^{\f{3+2\gamma_1}{2N}}\cdots\|n\|_{L^2}^{1-\f{3+2\gamma_{i}}{2N}}\|\na^{N}n\|_{L^2}^{\f{3+2\gamma_i}{2N}}\|\na^{N-m}n\|_{L^2}\|\na^{N+1}v\|_{L^2}\\
		\leq&C\delta\sum_{1\leq l\leq N-1}\sum_{\substack{\gamma_1+\cdots+\gamma_{i+j}=l\\\gamma_1+\cdots+\gamma_{i}=m\\1\leq m\leq l-1}}\|n\|_{L^2}^{1-\theta}\|\na^{N}n\|_{L^2}^{\theta}\|\na^{\al}n\|_{L^2}^{\theta}\|\na^{N}n\|_{L^2}^{1-\theta}\|\na^{N+1}v\|_{L^2}\\
		\leq&C\delta\|(\na^{N}n,\na^{N+1}v)\|_{L^2}^{2},
	\end{split}
	\eeq
	where
	$$\theta=\f{3i+2m}{2N},\quad\al=\f{3iN}{3i+2m}\leq \f35N\leq N,$$
	provided $i\geq 1$, $m\geq 1$ and $i\leq m$. We substitute \eqref{L31}-\eqref{L34} into \eqref{L3}, to find that
	\beq\label{l3}
	\begin{split}
		|L_3|\leq C\delta \|(\na^{N}n,\na^{N+1}v)\|_{L^2}^{2}.
	\end{split}
	\eeq
	For the last term $L_4$, with the aid of the estimates \eqref{h} and \eqref{h2} of $\h$, it is easy to deduce\beq\label{LL4}
	\begin{split}
		|L_4|\leq&C\int|\h(n,\bar\rho)||\na^{N}\bar\rho||\na^{N+1}v|dx+C\sum_{1\leq l\leq N-1}\int|\na^{l}\h(n,\bar\rho)||\na^{N-l}\bar\rho||\na^{N+1}v|dx\\
		\leq &C\int |n||\na^{N}\bar\rho||\na^{N+1}v|dx+C\sum_{1\leq l\leq N-1}\sum_{\gamma_1+\cdots+\gamma_{j}=l}\int|n||\na^{\gamma_1}\bar\rho|\cdots|\na^{\gamma_j}\bar\rho||\na^{N-l}\bar\rho||\na^{N}v|dx\\
		&\quad+C\sum_{0\leq l\leq N-1}\sum_{\gamma_1+\cdots+\gamma_i=l}\int|\na^{\gamma_1}n|\cdots|\na^{\gamma_i}n||\na^{N-l}\bar\rho||\na^{N+1}v|dx\\
		&\quad+C\sum_{1\leq l\leq N-1}\sum_{\substack{\gamma_1+\cdots+\gamma_{i+j}=l\\\gamma_1+\cdots+\gamma_{i}=m\\1\leq m\leq l-1}}\int|\na^{\gamma_1}n|\cdots|\na^{\gamma_i}n||\na^{\gamma_{i+1}}\bar\rho|\cdots|\na^{\gamma_{i+j}}\bar\rho||\na^{N-l}\bar\rho||\na^{N+1}v|dx\\
		\overset{def}{=}&L_{41}+L_{42}+L_{43}+L_{44}.
	\end{split}
	\eeq
	According to Hardy inequality, we obtain immediately
	\beno
	\begin{split}
		L_{41}\leq& \sum_{0\leq l\leq N-1} \|\f{n}{(1+|x|)^{N}}\|_{L^2}\|(1+|x|)^{N}\na^{N}\bar\rho\|_{L^\infty}\|\na^{N+1}v\|_{L^2}\leq C\delta \|(\na^{N}n,\na^{N+1}v)\|_{L^2}^2,\\
		L_{42}\leq& \sum_{0\leq l\leq N-1}\sum_{\gamma_1+\cdots+\gamma_j=l} \|\f{n}{(1+|x|)^{N}}\|_{L^2}\|(1+|x|)^{\gamma_1}\na^{\gamma_1}\bar\rho\|_{L^\infty}\cdots\|(1+|x|)^{\gamma_j}\na^{\gamma_j}\bar\rho\|_{L^\infty}\\
		&\quad\times\|(1+|x|)^{N-l}\na^{k+1}\bar\rho\|_{L^\infty}\|\na^{N+1}v\|_{L^2}\\
		\leq& C\delta \|(\na^{N}n,\na^{N+1}v)\|_{L^2}^2.
	\end{split}
	\eeno
	Next, let us to deal with $L_{43}$. It is easy to deduce that\beno
    \begin{split}
    	L_{43}\leq& C\sum_{0\leq l\leq N-1}\int|\na^{l}n||\na^{N-l}\bar\rho||\na^{N+1}v|dx+C\sum_{0\leq l\leq k}\sum_{\substack{\gamma_1+\cdots+\gamma_i=l\\i\geq2}}\int|\na^{\gamma_1}n|\cdots|\na^{\gamma_i}n||\na^{N-l}\bar\rho||\na^{N+1}v|dx\\
    	\overset{def}{=}&L_{431}+L_{432}.
    \end{split}
    \eeno
    We employ Hardy inequality once again, to get
    \beno
    \begin{split}
    	L_{431}\leq \sum_{0\leq l\leq k} \|\f{\na^ln}{(1+|x|)^{N-l}}\|_{L^2}\|(1+|x|)^{N-l}\na^{N-l}\bar\rho\|_{L^\infty}\|\na^{N+1}v\|_{L^2}\leq C\delta \|(\na^{N}n,\na^{N+1}v)\|_{L^2}^2.
    \end{split}
    \eeno
	The fact $i\geq2$ implies that $\gamma_\be+2\leq l+1\leq N$, for $\be=1,\cdots,i$. Applying Sobolev interpolation inequality \eqref{Sobolev} in Lemma \ref{inter} and Hardy inequality, we obtain\beno
	\begin{split}
		L_{432}\leq&C\sum_{1\leq l\leq N-1}\sum_{\substack{\gamma_1+\cdots+\gamma_{i}=l\\i\geq2}}\|\f{\na^{\gamma_1}n}{(1+|x|)^{N-l}}\|_{L^2}\|\na^{\gamma_2}n\|_{L^\infty}\cdots\|\na^{\gamma_i}n\|_{L^\infty}\|(1+|x|)^{N-l}\na\bar\rho\|_{L^2}\|\na^{N+1}v\|_{L^2}\\
		\leq&C\delta\sum_{1\leq l\leq N-1}\sum_{\substack{\gamma_1+\cdots+\gamma_{i}=l\\i\geq2}}\|\na^{N-l+\gamma_1}n\|_{L^2}\|n\|_{L^2}^{1-\f{3+2\gamma_2}{2N}}\|\na^{N}n\|_{L^2}^{\f{3+2\gamma_2}{2N}}\cdots\|_{L^2}\|n\|_{L^2}^{1-\f{3+2\gamma_i}{2N}}\|\na^{N}n\|_{L^2}^{\f{3+2\gamma_i}{2N}}\|\na^{N+1}v\|_{L^2}\\
		\leq&C\delta\sum_{1\leq l\leq N-1}\sum_{\substack{\gamma_1+\cdots+\gamma_{i+j}=l\\i\geq2}}\|\na^{\al}n\|_{L^2}^{\theta}\|\na^{N}n\|_{L^2}^{1-\theta}\|n\|_{L^2}^{1-\theta}\|\na^{N}n\|_{L^2}^{\theta}\|\na^{N+1}v\|_{L^2}\\
		\leq&C\delta\|(\na^{N}n,\na^{N+1}v)\|_{L^2}^2,
	\end{split}
	\eeno
	where
	$$\theta=\f{3(i-1)+2(l-\gamma_1)}{2N},\quad \al=\f{3(i-1)N}{3(i-1)+2(l-\gamma_1)}\leq \f{3}{5}N\leq N,$$
	provided that $i\geq2$ and $i-1\leq l-\gamma_1$.
	Two estimates of terms from $L_{431}$ and $L_{432}$ gives\beno
	L_{43}\leq C\delta\|(\na^{N}n,\na^{N+1}v)\|_{L^2}^2.
	\eeno
	The last term on the right handside of \eqref{LL4} can be divided into two terms:
	\beno
	\begin{split}
	L_{44}\leq&C\sum_{1\leq l\leq N-1}\sum_{\substack{\gamma_1+\cdots+\gamma_{1+j}=l\\\gamma_1=m\\1\leq m\leq l-1}}\|\f{\na^{\gamma_1}n}{(1+|x|)^{N-m}}\|_{L^2}\|(1+|x|)^{\gamma_{2}}\na^{\gamma_{2}}\bar\rho\|_{L^\infty}\cdots\|(1+|x|)^{\gamma_{1+j}}\na^{\gamma_{1+j}}\bar\rho\|_{L^\infty}\\
	&\times\|(1+|x|)^{N-l}\na^{N-l}\bar\rho\|_{L^\infty}\|\na^{N+1}v\|_{L^2}+C\sum_{1\leq l\leq N-1}\sum_{\substack{\gamma_1+\cdots+\gamma_{i+j}=l\\\gamma_1+\cdots+\gamma_{i}=m\\1\leq m\leq l-1,~i\geq2}}\|\f{\na^{\gamma_1}n}{(1+|x|)^{N-m}}\|_{L^2}\|\na^{\gamma_2}n\|_{L^\infty}\cdots\|\na^{\gamma_i}n\|_{L^\infty}\\
	&\times\|(1+|x|)^{\gamma_{i+1}}\na^{\gamma_{i+1}}\bar\rho\|_{L^\infty}\cdots\|(1+|x|)^{\gamma_{i+j}}\na^{\gamma_{i+j}}\bar\rho\|_{L^\infty}\|(1+|x|)^{N-l}\na^{N-l}\bar\rho\|_{L^\infty}\|\na^{N+1}v\|_{L^2}\\
	\overset{def}{=}&L_{441}+L_{442}.
	\end{split}
    \eeno
    We use Hardy inequaly, to find\beno
    \begin{split}
    	L_{441}\leq C\delta\sum_{1\leq l\leq N-1}\sum_{\substack{\gamma_1+\cdots+\gamma_{1+j}=l\\1\leq m\leq l-1}}\|\na^{N}n\|_{L^2}\|\na^{N+1}v\|_{L^2}\leq C\delta\|(\na^{N}n,\na^{N+1}v)\|_{L^2}^2,
    \end{split}
    \eeno
    To deal with term $L_{442}$, by virtue of Sobolev interpolation inequality \eqref{Sobolev} in Lemma \ref{inter} and Hardy inequality, we arrive at
	\beno
	\begin{split}
	L_{442}\leq&C\delta\sum_{1\leq l\leq N-1}\sum_{\substack{\gamma_1+\cdots+\gamma_{i+j}=l\\\gamma_1+\cdots+\gamma_{i}=m\\1\leq m\leq l-1,~i\geq2}}\|\na^{N-m+\gamma_1}n\|_{L^2}\|n\|_{L^2}^{1-\f{3+2\gamma_2}{2N}}\|\na^{N}n\|_{L^2}^{\f{3+2\gamma_2}{2N}}\cdots\|n\|_{L^2}^{1-\f{3+2\gamma_i}{2N}}\|\na^{N}n\|_{L^2}^{\f{3+2\gamma_i}{2N}}\|\na^{N+1}v\|_{L^2}\\
	\leq&C\delta\sum_{1\leq l\leq N-1}\sum_{\substack{\gamma_1+\cdots+\gamma_{i+j}=l\\\gamma_1+\cdots+\gamma_{i}=m\\1\leq m\leq l-1,~i\geq2}}\|\na^{\al}n\|_{L^2}^{\theta}\|\na^{N}n\|_{L^2}^{1-\theta}\|n\|_{L^2}^{1-\theta}\|\na^{N}n\|_{L^2}^{\theta}\|\na^{N+1}v\|_{L^2}\\
	\leq&C\delta\|(\na^{N}n,\na^{N+1}v)\|_{L^2}^2,
	\end{split}
	\eeno
	where
	$$\theta=\f{3(i-1)+2(m-\gamma_1)}{2N},\quad\al=\f{3(i-1)N}{3(i-1)+2(m-\gamma_1)}\leq \f{3}{5}N\leq N,$$
	provided that $i\geq2$ and $i-1\leq m-\gamma_1$.
	Hence, the combination of estimates of terms $L_{441}$ and $L_{442}$ implies directly
	\beno
	L_{44}\leq C\delta\|(\na^{N}n,\na^{N+1}v)\|_{L^2}^2.
	\eeno
	We substitute estimates of terms from $L_{41}$ to $L_{44}$ into \eqref{LL4} to find\beq\label{l4}
	\begin{split}
		|L_4|\leq C\delta\|(\na^{N}n,\na^{N+1}v)\|_{L^2}^2.
	\end{split}
	\eeq
	Inserting \eqref{l1}, \eqref{l2}, \eqref{l3} and \eqref{l4} into \eqref{ks22}, we thereby deduce that\beq\label{ws22}
	\begin{split}
		\Big|\int \na^{N}\widetilde{S}_2\cdot \na^{N}v dx\Big|\leq C\delta\|(\na^{N}n,\na^{N+1}v)\|_{L^2}^2.
	\end{split}
	\eeq
	Plugging \eqref{ws11} and \eqref{ws22} into \eqref{ehk1} gives \eqref{en2} directly. Therefore, the proof of this lemma is completed.
\end{proof}
Finally, we aim to recover the dissipation estimate for $n$.
\begin{lemm}\label{ennjc}
	Under the assumptions in Theorem \ref{them3}, for $1\leq k\leq N-1$, we have
	\beq\label{en3}
	\f{d}{dt}\int \na^k v\cdot\na^{k+1}ndx+\|\na^{k+1}n\|_{L^2}^2\leq C_2\|(\na^{k+1}v,\na^{k+2}v)\|_{L^2}^2,
	\eeq
	where $C_2$ is a positive constant independent of $t$.
\end{lemm}
\begin{proof}
	Applying differential operator $\na^k$ to $\eqref{ns5}_2$, multiplying the resulting equation by $\na^{k+1}n$, and integrating over $\mathbb{R}^3$, one arrives at\beq\label{vnjc}
	\begin{split}
		\int \na^{k}v_t\cdot\na^{k+1}n dx+\|\na^{k+1}n\|_{L^2}^2\leq C \|\na^{k+2}v\|_{L^2}^2+\int \na^k\widetilde{S}_2\cdot\na^{k+1}n dx.
	\end{split}
	\eeq
	In order to deal with $\int \na^{k}v_t\cdot\na^{k+1}n dx$, we turn
	the time derivative of velocity to the density.
	Then, applying differential operator $\na^k$ to the mass equation $\eqref{ns5}_1$, we find
	\[\na^kn_t+\gamma\na^{k}\dive v=\na^k \widetilde{S}_1.\]
	Hence, we can transform time derivative to the spatial derivative, i.e.,\beno
	\begin{split}
		\int \na^{k}v_t\cdot\na^{k+1}n dx=&\f{d}{dt}\int \na^{k}v\cdot\na^{k+1}n dx-\int \na^{k}v\cdot\na^{k+1}n_t dx\\
		=&\f{d}{dt}\int \na^{k}v\cdot\na^{k+1}n dx+\gamma\int \na^{k} v\cdot\na^{k+1}\dive v dx-\int \na^{k} v\cdot\na^{k+1}\widetilde{S}_1 dx\\
		=&\f{d}{dt}\int \na^{k}v\cdot\na^{k+1}n dx-\gamma\|\na^{k}\dive v\|_{L^2}^2-\int \na^{k+1}\dive v\cdot\na^{k-1}\widetilde{S}_1 dx
	\end{split}
	\eeno
	Substituting the identity above into \eqref{vnjc} and integrating by parts yield\beq\label{nvjc2}
	\begin{split}
		&\f{d}{dt}\int \na^{k}v\cdot\na^{k+1}n dx+\|\na^{k+1}n\|_{L^2}^2\\
		\leq& C \|(\na^{k+1}v,\na^{k+2}v)\|_{L^2}^2+C\int \na^{k+1}\dive v\cdot\na^{k-1}\widetilde{S}_1dx-C\int \na^k\widetilde{S}_2\cdot\na^{k+1}n dx.
	\end{split}
	\eeq
	As for the term of $\widetilde{S}_1$.
		It then follows in a similar way to the estimates of term from $I_1$ to $I_4$ in Lemma \ref{enn-1} that
	\beq\label{ss1}
	\begin{split}
		\Big|\int \na^{k+1}\dive v\cdot\na^{k-1}\widetilde{S}_1 dx\Big|\leq C\|\na^{k+2}v\|_{L^2}\|\na^{k-1}\widetilde{S}_1\|_{L^2}
		\leq C \delta \|\na^{k+1}n\|_{L^2}^2+C\|(\na^{k+1}v,\na^{k+2}v)\|_{L^2}^2.
	\end{split}
	\eeq
	To deal with the term of $\widetilde{S}_2$, we only need to follow the idea as the estimates of term from $L_1$ to $L_4$ in Lemma \ref{enn}. Hence, we
	give the estimates as follow
	\beq\label{ss2}
	\begin{split}
		\Big|\int \na^k\widetilde{S}_2\cdot\na^{k+1}n dx\Big|\leq C \|\na^k\widetilde{S}_2\|_{L^2}\|\na^{k+1}n\|_{L^2}\leq C \delta\|(\na^{k+1}n,\na^{k+2}v)\|_{L^2}^2.
	\end{split}
	\eeq
	We then utilize \eqref{ss1} and \eqref{ss2} in \eqref{nvjc2}, to deduce \eqref{en3} directly.
\end{proof}

\underline{\noindent\textbf{The proof of Proposition \ref{nenp}.}}
With the help of Lemmas \ref{enn-1}-\ref{ennjc},
it is easy to establish the estimate \eqref{eml}.
Therefore, we complete the proof of Proposition \ref{nenp}.

\subsection{Optimal decay of higher order derivative}
In this subsection, we will study the optimal decay rate for the
$k-th$ $(2 \leq k\leq N-1)$ order spatial derivatives of global solution.
In order to achieve this target, the optimal decay rate of higher order spatial
derivative will be established by the lower one.
In this aspect, the Fourier splitting method, developed by Schonbek(see \cite{Schonbek1985}),
is applied to establish the optimal decay rate for higher order derivative
of global solution in \cite{{Schonbek-Wiegner},{gao2016}}.
However, We're going to use time weighted and mathematical induction to solve this problem.

\begin{lemm}\label{N-1decay}
	Under the assumption of Theorem \ref{them3}, for $0\leq k\leq N-1$, we have
	\beq\label{n1h1}
	\|\na^k(n,v)\|_{H^{N-k}}\leq C (1+t)^{-\f34-\f k2},
	\eeq
	where $C$ is a positive constant independent of time.
\end{lemm}
\begin{proof}
We will take the strategy of induction to give the proof for
the decay rate \eqref{n1h1}. In fact, the decay rate \eqref{basic-decay}
implies \eqref{n1h1} holds true for the the case $k=0, 1$.
By the general step of induction, assume that the decay rate \eqref{n1h1}
holds on for the case $k=m$, i.e.,
	\beq\label{inducass}
	\|\na^m (n,v)\|_{H^{N-m}}\leq C (1+t)^{-\f34-\f m2},
	\eeq
for $m=1,...,N-2$.
Choosing the integer $l=m$ in \eqref{eml} and
multiplying it by $(1+t)^{\f32+m+\ep_0} (0<\ep_0<1)$, we have
\beno
\begin{split}
\f{d}{dt}\Big\{(1+t)^{\f32+m+\ep_0} \mathcal{E}^N_m(t)\Big\}
+(1+t)^{\f32+m+\ep_0}\big(\|\na^{m+1}n\|_{H^{N-m-1}}^2
+\|\na^{m+1}v\|_{H^{N-m}}^2\big)\leq C(1+t)^{\f12+m+\ep_0} \mathcal{E}^N_m(t).
\end{split}
\eeno
Integrating with respect to $t$, using the equivalent relation \eqref{emleq} and the decay estimate \eqref{inducass}, one obtains
\beq\label{energy2}
\begin{split}
&(1+t)^{\f32+m+\ep_0} \mathcal{E}^N_m(t)	
    +\int_0^t(1+\tau)^{\f32+m+\ep_0}\big(\|\na^{m+1}n\|_{H^{N-m-1}}^2
    +\|\na^{m+1}v\|_{H^{N-m}}^2\big)d\tau\\
\leq& \mathcal{E}^N_m(0)+C\int_0^t(1+\tau)^{\f12+m+\ep_0} \mathcal{E}^N_m(\tau)d\tau\\
\leq&C\|\na^m(n_0,v_0)\|_{H^{N-m}}^2
     +C\int_0^t(1+\tau)^{\f12+m+\ep_0} \|\na^m(n,v)\|_{H^{N-m}}^2d\tau\\
\leq&C\|\na^m(n_0,v_0)\|_{H^{N-m}}^2
     +C\int_0^t(1+\tau)^{-1+\ep_0}d\tau\leq C(1+t)^{\ep_0}.
\end{split}
\eeq
On the other hand, taking $l=m+1$ in \eqref{eml}, we have
\beq\label{Ekk}
\begin{split}	\f{d}{dt}\mathcal{E}^{N}_{m+1}(t)
+\|\na^{m+2}n\|_{H^{N-m-2}}^2+\|\na^{m+2}v\|_{H^{N-m-1}}^2\leq 0.
\end{split}
\eeq
Multiplying \eqref{Ekk} by $(1+t)^{\f52+m+\ep_0}$,
integrating over $[0, t]$ and using  estimate \eqref{energy2}, we find
\beno
\begin{split}
&(1+t)^{\f52+m+\ep_0}\mathcal{E}^{N}_{m+1}(t)
+\int_0^t(1+\tau)^{\f52+m+\ep_0}\big(\|\na^{m+2}n\|_{H^{N-m-2}}^2
+\|\na^{k+2}v\|_{H^{N-m-1}}^2\big)d\tau\\
\leq& \mathcal{E}^{N}_{m+1}(0)
+\int_0^t(1+\tau)^{\f32+m+\ep_0}\mathcal{E}^{N}_{m+1}(\tau)d\tau\\
\leq&C\|\na^{m+1}(n_0,v_0)\|_{H^{N-m-1}}^2
+\int_0^t(1+\tau)^{\f32+m+\ep_0}\|\na^{m+1}(n,v)\|_{H^{N-m-1}}^2d\tau
    	\leq C(1+t)^{\ep_0},
\end{split}
\eeno
which, togeter with the equivalent relation \eqref{emleq}, yields immediately
\beno
\|\na^{m+1}(n,v)\|_{H^{N-m-1}}\leq C(1+t)^{-\f34-\f {m+1}2}.
\eeno
Then, the decay estimate \eqref{n1h1} holds on for case of $k=m+1$.
By the general step of induction, we complete the proof of this lemma.
\end{proof}

\subsection{Optimal decay of critical derivative}

In this subsection, our target is to establish the optimal decay rate for the
$N-th$ order spatial derivative of global solution $(n, v)$ as it tends to zero.
The decay rate of $N-th$ order derivative of global solution $(n, v)$ in Lemma \ref{N-1decay}
is not optimal since it only has the same decay rate as the lower one.
The loss of time decay estimate comes from the appearance of cross term
$\frac{d}{dt}\int \nabla^{N-1} v \cdot \nabla^N n dx$ in energy
when we set up the dissipation estimate for the density.
Now let us introduce some notations that will be used frequently in this subsection.
Let $0\leq\varphi_0(\xi)\leq1$ be a function in $C_0^{\infty}(\mathbb{R}^3)$ such that\begin{equation*}
\begin{split}
\varphi_0(\xi)=\left\{
\begin{array}{ll}
1,\quad \text{for}~~|\xi|\leq \f{\eta}{2},\\[1ex]
0,\quad\text{for}~~|\xi|\geq \eta,   \\[1ex]
\end{array}
\right.
\end{split}
\end{equation*}
where $\eta$ is a fixed positive constant. Based on the Fourier transform, we can define a low-medium-high-frequency decomposition $(f^l(x),f^h(x))$ for a function $f(x)$ as follows:
\beq\label{def-h-l}
f^l(x)\overset{def}{=}\mathcal{F}^{-1}(\varphi_0(\xi)\widehat{f}(\xi))~~\text{and}~~f^h(x)\overset{def}{=}f(x)-f^l(x).
\eeq

\begin{lemm}\label{highfrequency}
Under the assumptions of Theorem \ref{them3},
there exists a positive small constant $\eta_3$, such that
\beq\label{en6}
\begin{split}	&\f{d}{dt}\Big\{\|\na^{N}(n,v)\|_{L^2}^2-\eta_3\int_{|\xi|\geq\eta}\widehat{\na^{N-1}v}\cdot \overline{\widehat{\na^{N}n}}d\xi \Big\}+\|\na^{N}v^h\|_{L^2}^2+\eta_3\|\na^{N}n^h\|_{L^2}^2\\
		\leq&  C_4\|\na^{N}(n^l, v^l)\|_{L^2}^2+C(1+t)^{-3-N},
\end{split}
\eeq
	where $C_4$ is a positive constant independent of time.
\end{lemm}
	\begin{proof}
Taking differential operating $\na^{N-1}$ to the equation \eqref{ns5}, it holds true
\beq\label{ns6}
\left\{\begin{array}{lr}
\na^{N-1}n_t +\gamma\na^{N-1}\dive v=\na^{N-1}\widetilde S_1,\\
\na^{N-1}v_t+\gamma\na^{N} n-\mu_1\na^{N-1}\tri v-\mu_2\na^{N}\dive v
=\na^{N-1}\widetilde  S_2.
\end{array}\right.
\eeq
Taking the Fourier transform of $\eqref{ns6}_2$, multiplying the resulting equation by $\overline{\widehat{\na^{N}n}}$ and integrating on $\{\xi||\xi|\geq \eta\}$, it holds true
\beq\label{f1}
\begin{split}
&\int_{|\xi|\geq\eta}\widehat{\na^{N-1}v_t}\cdot \overline{\widehat{\na^{N}n}}d\xi+\gamma\int_{|\xi|\geq\eta}|\widehat{\na^Nn}|^2d\xi\\
=&\int_{|\xi|\geq\eta}\big(\mu_1\widehat{\na^{N-1}\tri v}+\mu_2\widehat{\na^N\dive v}\big)\cdot \overline{\widehat{\na^{N}n}}d\xi+\int_{|\xi|\geq\eta}\widehat{\na^{N-1}\widetilde  S_2}\cdot \overline{\widehat{\na^{N}n}}d\xi.
\end{split}
\eeq
It is easy to deduce from $\eqref{ns6}_1$ that
\beno
\begin{split}
		\widehat{\na^{N-1}v_t}\cdot \overline{\widehat{\na^{N}n}}=&-i\xi\widehat{\na^{N-1}v_t}\cdot \overline{\widehat{\na^{N-1}n}}=-\widehat{\na^Nv_t}\cdot \overline{\widehat{\na^{N-1}n}}\\
		=&-\pa_t(\widehat{\na^Nv}\cdot \overline{\widehat{\na^{N-1}n}})+\widehat{\na^Nv}\cdot \overline{\widehat{\na^{N-1}n_t}}\\
		=&-\pa_t(\widehat{\na^Nv}\cdot \overline{\widehat{\na^{N-1}n}})-\gamma\widehat{\na^Nv}\cdot
\overline{\widehat{\na^{N-1}\dive v}}+\widehat{\na^Nv}\cdot \overline{\widehat{\na^{N-1}\widetilde{S}_1}}.
	\end{split}
	\eeno
Then, substituting this identity into identity \eqref{f1}, we have
\beq\label{f2}
\begin{split}
		&-\f{d}{dt}\int_{|\xi|\geq\eta}\widehat{\na^{N}v}\cdot \overline{\widehat{\na^{N-1}n}}d\xi+\gamma\int_{|\xi|\geq\eta}|\widehat{\na^Nn}|^2d\xi\\
		=&\int_{|\xi|\geq\eta}\big(\mu_1\widehat{\na^{N-1}\tri v}+\mu_2\widehat{\na^N\dive v}\big)\cdot \overline{\widehat{\na^{N}n}}d\xi+\gamma\int_{|\xi|\geq\eta}\widehat{\na^{N}v}\cdot \overline{\widehat{\na^{N-1}\dive v}}d\xi \\
		&\quad-\int_{|\xi|\geq\eta}\widehat{\na^{N}v}\cdot \overline{\widehat{\na^{N-1}\widetilde  S_1}}d\xi +\int_{|\xi|\geq\eta}\widehat{\na^{N-1}\widetilde  S_2}\cdot \overline{\widehat{\na^{N}n}}d\xi\\
		\overset{def}{=}&M_1+M_2+M_3+M_4.
\end{split}
\eeq
The application of Cauchy inequality yields directly
\beq\label{N1}
\begin{split}
|M_1|
\leq& C\int_{|\xi|\geq\eta}|\xi|^{2N+1}|\widehat{v}||\widehat{n}|d\xi
\leq\ep \int_{|\xi|\geq\eta}|\xi|^{2N}|\widehat{n}|^2d\xi
        +C_{\ep}\int_{|\xi|\geq\eta}|\xi|^{2(N+1)}|\widehat{v}|^2d\xi,
\end{split}
\eeq
for some small $\ep$, which will be determined later.
Obviously, it holds true
\beq\label{N2}
|M_2| \leq  C\int_{|\xi|\geq\eta}|\xi|^{2N}|\widehat{v}|^2d\xi.
\eeq
Using the Cauchy inequality and definition of $\widetilde{S}_1$, it is easy to check that
\beq\label{m3}
\begin{split}
|M_3|
\leq&C \int_{|\xi|\geq\eta}|\xi|^{2(N+1)}|\widehat{v}|^2d\xi
       +C\int_{|\xi|\geq\eta}|\xi|^{2(N-2)}|\widehat{\widetilde  S_1}|^2 d\xi \\
\leq&C \int_{|\xi|\geq\eta}|\xi|^{2(N+1)}|\widehat{v}|^2d\xi
       +C\int_{|\xi|\geq\eta}|\xi|^{2(N-2)}|\widehat{\na nv+n\na v}|^2d\xi\\
    & +C\int_{|\xi|\geq\eta}|\xi|^{2(N-2)}|\widehat{\na\bar\rho v+\bar\rho\na v}|^2d\xi\\
\overset{def}{=}&\int_{|\xi|\geq\eta}|\xi|^{2(N+1)}|\widehat{v}|^2d\xi+M_{31}+M_{32}.
\end{split}
\eeq
Using Plancherel Theorem  and Sobolev inequality, it is easy to check that
\beq\label{m31}
\begin{split}
M_{31}
\le  &C\|\na^{N-2}(\na nv+n\na v)\|_{L^2}^2\\
\leq &C\big(\|\na n\|_{L^\infty}^2\|\na^{N-2}v\|_{L^2}^2
      +\|\na^{N-1}n\|_{L^2}^2\|v\|_{L^\infty}^2
      +\|n\|_{L^\infty}^2\|\na^{N-1}v\|_{L^2}^2
      +\|\na^{N-2}n\|_{L^2}^2\|\na v\|_{L^\infty}^2\big)\\		
\leq&C\big(\|\na^{2}(n,v)\|_{H^1}^2\|\na^{N-2}(n,v)\|_{L^2}^2
      +\|\na(n,v)\|_{H^1}^2\|\na^{N-1}v\|_{L^2}^2\big)\\
\leq&C(1+t)^{-3-N},
\end{split}
\eeq
where we have used the decay \eqref{n1h1} in the last inequality.
Similarly, we also apply Hardy's inequality to obtain
\beq\label{m32}
\begin{split}
M_{32}
\leq &C\int_{|\xi|\geq\eta}|\xi|^{2(N-1)}|\widehat{\na\bar\rho v+\bar\rho\na v}|^2d\xi
\le  C\|\na^{N-1}(\na \bar\rho v+\bar\rho\na v)\|_{L^2}^2 \\
		\leq&C\sum_{0\leq l\leq N-1}\Big(\|(1+|x|)^{l+1}\na^{l+1}\bar\rho\|_{L^\infty}\|\f{\na^{N-1-l}v}{(1+|x|)^{l+1}}\|_{L^2}+\|(1+|x|)^{l}\na^{l}\bar\rho\|_{L^\infty}\|\f{\na^{N-l}v}{(1+|x|)^{l}}\|_{L^2}\Big)\|\na^{N+1}v\|_{L^2}\\
		\leq&C\delta\|(\na^{N}v,\na^{N+1}v)\|_{L^2}^2,
\end{split}
\eeq
where we have used the fact that for any suitable function $\phi$,
there exists a positive constant $C$ dependent only on $\eta$ such that
\beno
\int_{|\xi|\geq\eta}|\xi|^{2(N-2)}|\widehat{\phi}|^2d\xi\leq C\int_{|\xi|\geq\eta}|\xi|^{2(N-1)}|\widehat{\phi}|^2d\xi.
\eeno
Substituting the estimates \eqref{m31} and \eqref{m32} into \eqref{m3},
one can get that
\beq\label{N3}
|M_3|\leq C\delta\|(\na^{N}v,\na^{N+1}v)\|_{L^2}^2 +C(1+t)^{-3-N}.
\eeq
Using Cauchy inequality and the definition of $\widetilde  S_2$, we have
\beq\label{m4}
\begin{split}
|M_4|
\leq &C\int_{|\xi|\geq\eta}|\xi|^{2N-1}|\widehat{\widetilde  S_2}|| {\widehat{n}}|d\xi\\
		\leq&C\int_{|\xi|\geq\eta}|\xi|^{2N-1}| {\widehat{n}}||\widehat{v\cdot \na v}|d\xi
     +C\int_{|\xi|\geq\eta}|\xi|^{2N-1}| {\widehat{n}}
        |\big(\mu_1|\widehat{\wf(n+\bar\rho)\tri v}|
         +\mu_2|\widehat{\wf(n+\bar\rho)\na\dive v}|\big)d\xi\\
&\quad+C\int_{|\xi|\geq\eta}|\xi|^{2N-1}| {\widehat{n}}|
          |\widehat{\wg(n+\bar\rho)\na n}|d\xi
      +C\int_{|\xi|\geq\eta}|\xi|^{2N-1}| {\widehat{n}}|
          |\widehat{\h(n,\bar\rho)\na \bar\rho}|d\xi\\
\overset{def}{=}&M_{41}+M_{42}+M_{43}+M_{44}.
\end{split}
\eeq
In view of Plancherel Theorem, Sobolev inequality and
commutator estimate in Lemma \ref{commutator}, we find
\beq\label{M41}
\begin{split}
M_{41}
\leq &\ep \|\na^Nn\|_{L^2}^2+C_{\ep}\|\na^{N-1}(v\cdot\na v)\|_{L^2}^2\\
\leq &\ep \|\na^Nn\|_{L^2}^2+C_{\ep}\|v\|_{L^\infty}^2\|\na^Nv\|_{L^2}^2
      +C_{\ep}\|[\na^{N-1},v]\cdot\na v\|_{L^2}^2\\
\leq &\ep \|\na^Nn\|_{L^2}^2+C_{\ep}\|\na v\|_{H^1}^2\|\na^Nv\|_{L^2}^2
      +C_{\ep}\|\na v\|_{L^\infty}^2\|\na^{N-1} v\|_{L^2}^2\\
\leq &\ep \|\na^Nn\|_{L^2}^2+C_{\ep}(1+t)^{-3-N},
\end{split}
\eeq
where we have used the estimate \eqref{n1h1} in the last inequality.
Similarly, it is easy to check that
\beq\label{m42}
\begin{split}
M_{42}
\leq &\ep \|\na^Nn\|_{L^2}^2+C_{\ep}\|\na^{N-1}\big(\wf(n+\bar\rho)
        (\mu_1\tri v+\mu_2\na\dive v)\big)\|_{L^2}^2\\
\leq &\ep\|\na^Nn\|_{L^2}^2+C_{\ep}\|(n,\bar\rho)\|_{L^\infty}^2
      \|\na^{N+1}v\|_{L^2}^2+C_{\ep}\|\na n\|_{L^3}^2\|\na^{N}v\|_{L^6}^2
      +C_{\ep}\|(1+|x|)\na \bar\rho\|_{L^\infty}^2 \|\f{\na^{N}v}{1+|x|}\|_{L^2}^2\\
     &\quad+C_{\ep}\sum_{2\leq l\leq N-1}\|\na^l\wf(n+\bar\rho)\na^{N+1-l}v\|_{L^2}^2\\
\leq &\ep \|\na^Nn\|_{L^2}^2+C_{\ep}\delta\|\na^{N+1}v\|_{L^2}^2
      +C_{\ep}\sum_{2\leq l\leq N-1}\|\na^l\wf(n+\bar\rho)\na^{N+1-l}v\|_{L^2}^2.
\end{split}
\eeq
Now let us deal with the last term on the right handside
of \eqref{m42}. Indeed, it is easy to deduce that
\beno
\begin{split}
&\sum_{2\leq l\leq N-1}\|\na^l\wf(n+\bar\rho)\na^{N+1-l}v\|_{L^2}\\
\leq& C\sum_{2\leq l\leq N-1}\sum_{\substack{\gamma_1+\cdots+\gamma_{i+j}=l\\\gamma_1+\cdots+\gamma_{i}=m\\0\leq m\leq l}}\||\na^{\gamma_1}n|\cdots|\na^{\gamma_{i}}n||\na^{\gamma_{i+1}}\bar\rho|\cdots|\na^{\gamma_{i+j}}\bar\rho||\na^{N+1-l}v|\|_{L^2}.
\end{split}
\eeno
For $m=0$, we apply the Hardy inequality to obtain
\beno
\begin{split}
\sum_{2\leq l\leq N-1}
\sum_{\gamma_1+\cdots+\gamma_{j}=l}\|(1+|x|)^{\gamma_{i+1}}\na^{\gamma_{i+1}}\bar\rho\|_{L^\infty}
\cdots\|(1+|x|)^{\gamma_{i+j}}\na^{\gamma_{i+j}}\bar\rho\|_{L^\infty}
\|\f{\na^{N+1-l}v}{(1+|x|)^{l}}\|_{L^2}
\leq C\delta\|\na^{N+1}v\|_{L^2}.
\end{split}
\eeno
For $1\leq m\leq l-1$,  the Sobolev inequality
and decay estimate \eqref{n1h1} imply
\beno
\begin{split}
&\sum_{2\leq l\leq N-1}
\sum_{\substack{\gamma_1+\cdots+\gamma_{i+j}=l\\\gamma_1+\cdots+\gamma_{i}=m\\1\leq m\leq l-1}}
  \|\na^{\gamma_1}n\|_{L^\infty}\cdots\|\na^{\gamma_{i}}n\|_{L^\infty}
  \|(1+|x|)^{\gamma_{i+1}}\na^{\gamma_{i+1}}\bar\rho\|_{L^\infty}
  \cdots\|(1+|x|)^{\gamma_{i+j}}\na^{\gamma_{i+j}}\bar\rho\|_{L^\infty}
  \|\f{\na^{N+1-l}v}{(1+|x|)^{l-m}}\|_{L^2}\\
\leq&C\sum_{2\leq l\leq N-1}\sum_{\substack{\gamma_1+\cdots+\gamma_{i+j}=l\\\gamma_1+\cdots
                                    +\gamma_{i}=m\\1\leq m\leq l-1}}
\|\na^{\gamma_1+1}n\|_{H^1}\cdots\|\na^{\gamma_{i}+1}n\|_{H^1}\|\na^{N+1-m}v\|_{L^2}
\leq C(1+t)^{-\f{4+N}{2}}.
\end{split}
\eeno
For $m=l$, the Sobolev inequality
and decay estimate \eqref{n1h1} yield directly
\beno
\begin{split}
&\sum_{2\leq l\leq N-1}\sum_{\gamma_1+\cdots+\gamma_{i}=l}
\|\na^{\gamma_1}n\|_{L^\infty}
\cdots\|\na^{\gamma_{i-1}}n\|_{L^\infty}
\|\na^{\gamma_{i}}n\|_{L^3}\|\na^{N+1-l}v\|_{L^6}\\
\leq
&\sum_{2\leq l\leq N-1}\sum_{\gamma_1+\cdots+\gamma_{i}=l}\|\na^{\gamma_1+1}n\|_{H^1}\cdots
\|\na^{\gamma_{i-1}+1}n\|_{H^1}\|\na^{\gamma_{i}}n\|_{H^1}\|\na^{N+1-l}v\|_{H^1}
\leq C(1+t)^{-\f{4+N}{2}}.
\end{split}
\eeno
Therefore, we can obtain the following estimate
\beq\label{m421}
\begin{split}
	\sum_{2\leq l\leq N-1}\|\na^l\wf(n+\bar\rho)\na^{N+1-l}v\|_{L^2}
	\leq C\delta\|\na^{N+1}v\|_{L^2}+C(1+t)^{-\f{4+N}{2}},
\end{split}
\eeq
which, together with the estimate \eqref{m42}, yields directly
\beq\label{M42}
M_{42}
	\leq \ep \|\na^Nn\|_{L^2}^2+C_{\ep}\delta \|\na^{N+1}v\|_{L^2}^2+C_{\ep}(1+t)^{-4-N}.
\eeq
One can deal with the term $M_{43}$ in the manner of $M_{42}$. It holds true
\beq\label{m43}
\begin{split}
	M_{43}
	\leq&\ep \|\na^Nn\|_{L^2}^2+C_{\ep}\|(n,\bar\rho)\|_{L^\infty}^2\|\na^{N}n\|_{L^2}^2+C_{\ep}\sum_{1\leq l\leq N-1}\|\na^l\wg(n+\bar\rho)\na^{N-l}n\|_{L^2}^2\\
	\leq&(\ep+C_\ep\delta) \|\na^Nn\|_{L^2}^2+C_{\ep}\sum_{2\leq l\leq N-1}\|\na^l\wg(n+\bar\rho)\na^{N-l}n\|_{L^2}^2.
\end{split}
\eeq
Similar to the estimate \eqref{m421}, it is easy to check that
\beno
\begin{split}
	\sum_{2\leq l\leq N-1}\|\na^l\wg(n+\bar\rho)\na^{N+1-l}v\|_{L^2}
	\leq C\delta\|\na^{N}n\|_{L^2}+C(1+t)^{-\f{3+N}{2}},
\end{split}
\eeno
which, together with the inequality \eqref{m43}, yields directly
\beq\label{M43}
\begin{split}
	M_{43}
	\leq&(\ep+C_\ep\delta) \|\na^Nn\|_{L^2}^2+C_{\ep}(1+t)^{-3-N}.
\end{split}
\eeq
Finally, let us deal with the term $M_{44}$.
Indeed, the Hardy inequality and identity \eqref{h} yield directly
\beq\label{m44}
	\begin{split}
		M_{44}		\leq&\ep\|\na^{N}n\|_{L^2}^2+C_{\ep}\|\h(n,\bar\rho)\na^{N}\bar\rho\|_{L^2}^2+C_{\ep}\sum_{1\leq l\leq N-1}\|\na^l\h(n,\bar\rho)\na^{N-l}\bar\rho\|_{L^2}^2\\
		\leq&\ep\|\na^{N}n\|_{L^2}^2+C_{\ep}\|\f{n}{(1+|x|)^{N}}\|_{L^2}^2\|(1+|x|)^{N}\na^{N}\bar\rho\|_{L^\infty}^2+C_{\ep}\sum_{1\leq l\leq N-1}\|\na^l\h(n,\bar\rho)\na^{N-l}\bar\rho\|_{L^2}^2\\
		\leq&(\ep+C\delta)\|\na^{N}n\|_{L^2}^2+C_{\ep}\sum_{1\leq l\leq N-1}\|\na^l\h(n,\bar\rho)\na^{N-l}\bar\rho\|_{L^2}^2.
	\end{split}
	\eeq
	To deal with the last term on the right side of \eqref{m44}, we employ the estimate \eqref{h2} of $\h$, to find\beno
	\begin{split}
		&\|\na^l\h(n,\bar\rho)\na^{N-l}\bar\rho\|_{L^2}\\
		\leq&C\sum_{\substack{\gamma_1+\cdots+\gamma_{1+j}=l\\\gamma_1=m\\0\leq m\leq l}}\|\f{\na^{\gamma_1}n}{(1+|x|)^{N-m}}\|_{L^2}\|(1+|x|)^{\gamma_{2}}\na^{\gamma_{2}}\bar\rho\|_{L^\infty}\cdots\|(1+|x|)^{\gamma_{1+j}}\na^{\gamma_{1+j}}\bar\rho\|_{L^\infty}\|(1+|x|)^{N-l}\na^{N-l}\bar\rho\|_{L^\infty}\\
		&\quad+C\sum_{\substack{\gamma_1+\cdots+\gamma_{i+j}=l\\\gamma_1+\cdots+\gamma_{i}=m\\1\leq m\leq l,i\geq2}}\|\f{\na^{\gamma_1}n}{(1+|x|)^{N-m}}\|_{L^2}\|\na^{\gamma_2}n\|_{L^\infty}
\cdots\|\na^{\gamma_i}n\|_{L^\infty}
\|(1+|x|)^{\gamma_{i+1}}\na^{\gamma_{i+1}}\bar\rho\|_{L^\infty}\\
&\quad  \quad  \quad  \quad  \quad  \quad \quad \quad \quad \cdots\|(1+|x|)^{\gamma_{i+j}}\na^{\gamma_{i+j}}\bar\rho\|_{L^\infty}\|(1+|x|)^{N-l}\na^{N-l}\bar\rho\|_{L^\infty}\\
		\overset{def}{=}&M_{441}+M_{442}.
	\end{split}
	\eeno
In view of Hardy inequality, we can obtain the estimate
\beno
\begin{split}
M_{441}
	\leq& C\delta\|\na^N n\|_{L^2}.
\end{split}
\eeno
To deal with the term $M_{442}$, we can apply Hardy inequality and the decay estimate \eqref{n1h1} to obtain
		\beno
		\begin{split}
			M_{442}
			\leq& C\delta\sum_{\substack{\gamma_1+\cdots+\gamma_{i+j}=l\\\gamma_1+\cdots+\gamma_{i}=m\\0\leq m\leq l,~i\geq2}}\|\na^{N-m+\gamma_1}n\|_{L^2}\|\na^{\gamma_2+1}n\|_{H^1}\cdots\|\na^{\gamma_i+1}n\|_{H^1}
			\leq C(1+t)^{-\f{4+N}{2}}.
		\end{split}
		\eeno
The combination of estimates from term $M_{441}$ to $M_{442}$ yields directly
\beno
\sum_{1\leq l\leq N-1}\|\na^l\h(n,\bar\rho)\na^{N-l}\bar\rho\|_{L^2}
\leq C\delta\|\na^{N}n\|_{L^2}+C(1+t)^{-\f{4+N}{2}},
\eeno
which, together with the inequality \eqref{m44}, yields directly
	\beq\label{M44}
	\begin{split}
		M_{44}
		\leq(\ep+C_{\ep}\delta)\|\na^{N}n\|_{L^2}^2+C_{\ep}(1+t)^{-4-N}.
	\end{split}
	\eeq
Consequently, by virtue of the estimates \eqref{m4}, \eqref{M41},
\eqref{M42}, \eqref{M43} and \eqref{M44}, it holds true
\beq\label{M4}
M_4 \leq (\ep+C_{\ep}\delta) \|\na^N n\|_{L^2}^2
    +C_{\ep}\delta \|\na^{N+1}v\|_{L^2}^2+C_{\ep}(1+t)^{-3-N}.
\eeq
Substituting the estimates \eqref{N1}-\eqref{N3} and \eqref{M4} into \eqref{f2}, we find
\beno
\begin{split}
&-\f {d}{dt}\int_{|\xi|\geq\eta}\widehat{\na^{N}v}\cdot \overline{\widehat{\na^{N-1}n}}d\xi+\gamma\int_{|\xi|\geq\eta}|\widehat{\na^Nn}|^2d\xi\\
&\leq(\ep+C_{\ep}\delta)\|\na^N n\|_{L^2}^2+C_{\ep}\|(\na^N v,\na^{N+1}v)\|_{L^2}^2+C_{\ep}(1+t)^{-3-N}.
	\end{split}
	\eeno
Due to the definition \eqref{def-h-l}, there exists a positive constant $C$ such that
	\beq\label{vhvl}
	\begin{split}
		\|\na^N v^h\|_{L^2}^2\leq C\|\na^{N+1}v^h\|_{L^2}^2,\quad \|\na^{N+1} v^l\|_{L^2}^2\leq C \|\na^{N}v^l\|_{L^2}^2,
	\end{split}
	\eeq
	and choosing $\ep$ and $\delta$ suitably small, we deduce that
	\beq\label{en5}
	\begin{split}
		&-\f{d}{dt}\int_{|\xi|\geq\eta}\widehat{\na^{N}v}\cdot \overline{\widehat{\na^{N-1}n}}d\xi+\gamma\int_{|\xi|\geq\eta}|\widehat{\na^Nn}|^2d\xi
		\leq C\|\na^N (n,v)^l\|_{L^2}^2+C_{3} \|\na^{N+1}v^h\|_{L^2}^2+C(1+t)^{-3-N}.
	\end{split}
	\eeq
Recalling the estimate \eqref{en2} in Lemma \ref{enn}, one has the following estimate
\beq\label{en4}
\f{d}{dt}\|\na^{N}(n,v)\|_{L^2}^2+\|\na^{N+1}v\|_{L^2}^2\leq C\delta \|(\na^{N}n,\na^{N+1}v)\|_{L^2}^2.
\eeq
Multiplying \eqref{en5} by $\eta_3$, then adding to \eqref{en4}, and choosing $\delta$ and $\eta_3$ suitably small, then we have
\begin{equation*}
\begin{aligned}		&\f{d}{dt}\Big\{\|\na^{N}(n,v)\|_{L^2}^2-\eta_3\int_{|\xi|\geq\eta}\widehat{\na^{N-1}v}\cdot \overline{\widehat{\na^{N}n}}d\xi \Big\}
+\|\na^{N+1}v\|_{L^2}^2+\eta_3\|\na^{N}n^h\|_{L^2}^2\\
\leq &C_4\|\na^{N}(n^l,v^l)\|_{L^2}^2+C(1+t)^{-3-N}.
\end{aligned}
\end{equation*}
Using \eqref{vhvl} once again, we obtain that
\begin{equation*}
\begin{aligned}	
&\f{d}{dt}\Big\{\|\na^{N}(n,v)\|_{L^2}^2
  -\eta_3\int_{|\xi|\geq\eta}\widehat{\na^{N-1}v}\cdot
  \overline{\widehat{\na^{N}n}}d\xi \Big\}
  +\|\na^{N}v^h\|_{L^2}^2+\eta_3\|\na^{N}n^h\|_{L^2}^2\\
\leq & C_4\|\na^{N}(n^l, v^l)\|_{L^2}^2+C(1+t)^{-3-N}.
\end{aligned}
\end{equation*}
Therefore, we complete the proof of this lemma.
\end{proof}

    Observe the right handside of the estimate \eqref{en6} in Lemma \ref{highfrequency}, we need to estimate the low frequency of $\na^N(n,v)$.
    For this purpose, we need to analyze the initial value problem for linearized system of \eqref{ns5}:
    \beq\label{linear}
    \left\{\begin{array}{lr}
	\widetilde n_t +\gamma\dive\widetilde v=0,\quad (t,x)\in \mathbb{R}^{+}\times \mathbb{R}^3,\\
	\widetilde u_t+\gamma\na \widetilde n-\mu_1\tri \widetilde v-\mu_2\na\dive \widetilde v =0,\quad (t,x)\in \mathbb{R}^{+}\times \mathbb{R}^3,\\
	(\widetilde n,\widetilde v)|_{t=0}=(n_0,v_0),\quad x\in \mathbb{R}^3.
    \end{array}\right.
    \eeq
    In terms of the semigroup theory for evolutionary equations, the solution $(\widetilde n,\widetilde v)$ of the linearized system \eqref{linear} can be expressed as
    \beq\label{U}
    \left\{\begin{array}{lr}
    \widetilde U_t=A\widetilde U,\quad t\geq0,\\
    \widetilde U(0)=U_0,
\end{array}\right.\eeq
    where $\widetilde U \overset{def}{=}(\widetilde{n},\widetilde{v})^t$,
    $U_0 \overset{def}=(n_0,v_0)$
    and the matrix-valued differential operator $A$ is given by
    \beno
    A = {\left(
	\begin{matrix}
		0 & -\gamma \dive \\
		-\gamma \na & \mu_1\tri+\mu_2\na\dive
	\end{matrix}
	\right).}
    \eeno
    Denote $S(t)\overset{def} =e^{tA}$, then the system \eqref{U} gives rise to
    \beq\label{uexpress}\widetilde U(t)=S(t)U_0=e^{tA} U_0,\quad t\geq0. \eeq
    Then, it is easy to check that the following estimate holds
    \beq\label{linearlow}
    \|\na^N(S(t)U_0)\|_{L^2}\leq C(1+t)^{-\f34-\f N2}\|U_0\|_{L^1\cap H^N},
    \eeq
    where $C$ is a positive constant independent of time.
    The estimate \eqref{linearlow} can be found in \cite{chen2021,duan2007}.
    Finally, let us denote $F(t)=(\widetilde{S_1}(t),\widetilde{S_2}(t))^{tr}$, then
    the system \eqref{ns5} can be rewritten as follows:
    \beq\label{nonlinear}
    \left\{\begin{array}{lr}
    	U_t=A U+F,\\
    	U(0)=U_0.
    \end{array}\right.
    \eeq
    Then we can use Duhamel's principle to represent the solution of system \eqref{ns5} in term of the semigroup
    \beq\label{Uexpress}
    U(t)=S(t)U_0+\int_0^tS(t-\tau)F(\tau)d\tau.
    \eeq
    Now, one can establish the estimate for the low frequency of $\na^N(n,v)$ as follows:
\begin{lemm}\label{lowfrequency}
Under the assumption of Theorem \ref{them3}, we have
\beq\label{lowfre}
\|\na^N(n^l, v^l)(t)\|_{L^2}
\leq C \delta \sup_{0\leq s\leq t}\|\na^N(n,v)(s)\|_{L^2}+C(1+t)^{-\f34-\f N2},
\eeq
where $C$ is a positive constant independent of time.
\end{lemm}
\begin{proof}
	It follows from the formula \eqref{Uexpress} that
    \beno
	\na^N(n, v)=\na^N(S (t)U_0)+\int_0^t\na^{N}[S(t-\tau)F(\tau)]d\tau,
	\eeno
	which yields directly
	\beq\label{nvexpress}
	\|\na^N(n^l, v^l)\|_{L^2}
    \leq \|\na^N(S (t)U_0)^l\|_{L^2}+\int_0^t\|\na^{N}[S(t-\tau)F(\tau)]^l\|_{L^2}d\tau.
	\eeq
	Since the initial data $U_0=(n_0,v_0)\in L^1\cap H^N$, it follows from the
    estimate \eqref{linearlow} that
    \beq\label{U0es}
	\|\na^N(S (t)U_0)^l\|_{L^2}\leq C(1+t)^{-\f34-\f N2}\|U_0\|_{L^1\cap H^N}.
	\eeq
	Then we compute by means of Sobolev inequality that\beq\label{nvl}
	\begin{split}
		&\int_0^t\|\na^{N}[S(t-\tau)F(\tau)]^l\|_{L^2}d\tau
		\leq\int_0^t\||\xi|^{N}|\widehat{S}(t-\tau)||\widehat F(\tau)|\|_{L^2(|\xi|\leq \eta)}d\tau\\
		\leq&\int_0^{\f t2}\||\xi|^{N}|\widehat{S}(t-\tau)|\|_{L^2(|\xi|\leq \eta)}\|\widehat F(\tau)\|_{L^\infty(|\xi|\leq \eta)}d\tau+\int_{\f t2}^t\||\xi||\widehat{S}(t-\tau)|\|_{L^2(|\xi|\leq \eta)}\||\xi|^{N-1}\widehat F(\tau)\|_{L^\infty(|\xi|\leq \eta)}d\tau\\
		\leq&\int_0^{\f t2}(1+t-\tau)^{-\f34-\f N2}\|\widehat F(\tau)\|_{L^\infty(|\xi|\leq \eta)}d\tau+\int_{\f t2}^t(1+t-\tau)^{-\f54}\||\xi|^{N-1}\widehat F(\tau)\|_{L^\infty(|\xi|\leq \eta)}d\tau\\
		\overset{def}{=}&N_1+N_2.
	\end{split}
	\eeq
	Now we estimate the first term on the right handside of \eqref{nvl} as follows:
\beq\label{ss1N}
\begin{split}
N_1=
\int_0^{\f t2}(1+t-\tau)^{-\f34-\f N2}\|F\|_{L^1}d\tau
\leq C\int_0^{\f t2}(1+t-\tau)^{-\f34-\f N2} \big(\|\widetilde{S}_1\|_{L^1}+\|\widetilde{S}_2\|_{L^1}\big)d\tau.
\end{split}
\eeq
In view of the definitions of $\widetilde{S}_i(i=1,2)$
and decay estimate \eqref{n1h1}, we have
\beq\label{SS1}
\|\widetilde{S}_1\|_{L^1}\leq C\|\na (n,v)\|_{L^2}\|(n,v)\|_{L^2}+C\|(1+|x|)\na \bar\rho\|_{L^2}\|\f{v}{1+|x|}\|_{L^2}+C\|\bar\rho\|_{L^2}\|\na v\|_{L^2}\leq C\delta (1+t)^{-\f54},
\eeq
and
\beq\label{SS2}
\begin{split}
		\|\widetilde{S}_2\|_{L^1}\leq& C\|v\|_{L^2}\|\na v\|_{L^2}+C\|\wf(n+\bar\rho)\|_{L^2}\|\na^2 v\|_{L^2}+C\|\wg(n+\bar\rho)\|_{L^2}\|\na n\|_{L^2}+C\|\h(n,\bar\rho)\|_{L^2}\|\na \bar\rho\|_{L^2}\\
		\leq&C\|v\|_{L^2}\|\na v\|_{L^2}+C\|n\|_{L^2}\|\na^2 v\|_{L^2}+C\|\bar\rho\|_{L^2}\|\na^2 v\|_{L^2}+C\|n\|_{L^2}\|\na n\|_{L^2}\\
		&\quad+C\|\bar\rho\|_{L^2}\|\na n\|_{L^2}+C\|\f{n}{1+|x|}\|_{L^2}\|(1+|x|)\na \bar\rho\|_{L^2}\\
		\leq&C\delta (1+t)^{-\f54}.
	\end{split}
	\eeq
Substituting the estimates \eqref{SS1} and \eqref{SS2} into \eqref{ss1N},
and using the estimate in Lemma \ref{tt2}, it holds
\beq\label{F1}
	\begin{split}
		N_1\leq C \int_0^{\f t2}(1+t-\tau)^{-\f34-\f N2}(1+\tau)^{-\f54}d\tau\leq C (1+t)^{-\f34-\f N2}.
\end{split}
\eeq
Next, let us deal with the $N_2$ term.
For any smooth function $\phi$, there exists a positive constant $C$
dependent only on $\eta$, such that
$$\||\xi|^{N-1}\widehat \phi\|_{L^\infty(|\xi|\leq \eta)}\leq C\||\xi|^{N-2}\widehat \phi\|_{L^\infty(|\xi|\leq \eta)},$$
	then we find that
\beq\label{L}
	\begin{split}
		\||\xi|^{N-1}\widehat F\|_{L^\infty(|\xi|\leq \eta)}
		\leq& C \|[\na^{N-1}\widetilde {S}_1]^l\|_{L^1}+C\|[\na^{N-1}\widetilde {S}_2]^l\|_{L^1}\\
		\leq&C\|[\na^{N-1}(\na nv+n\na v)]^l\|_{L^1}+C\|[\na^{N-1}(\na \bar\rho v+\bar\rho\na v)]^l\|_{L^1}+C\|[\na^{N-1}(v\na v)]^l\|_{L^1}\\
		&\quad+C\|[\na^{N-2}\big(\wf(n+\bar\rho)(\mu_1\tri v+\mu_2\na \dive v)\big)]^l\|_{L^1}+C\|[\na^{N-1}(\wg(n+\bar\rho)\na n)]^l\|_{L^1}\\
		&\quad+C\|[\na^{N-1}(\h(n,\bar\rho)\na \bar\rho)]^l\|_{L^1}\\
        \overset{def}{=}&N_{21}+N_{22}+N_{23}+N_{24}+N_{25}+N_{26}.
	\end{split}
	\eeq
	First of all, applying the decay estimate \eqref{n1h1}, then the term $N_{21}$ can be estimated as follows
    \beq\label{L1}
	\begin{split}
		N_{21}\leq C\sum_{0\leq l\leq N-1}\big(\|\na^{l+1}n\|_{L^2}\|\na^{N-l-1}v\|_{L^2}+\|\na^{l}n\|_{L^2}\|\na^{N-l}v\|_{L^2}\big)
		\leq C (1+t)^{-1-\f N2}.
	\end{split}
	\eeq
	By virtue of Hardy inequality, we have\beq\label{L222}
	\begin{split}
		N_{22}\leq&C\sum_{0\leq l\leq N-1}\Big(\|(1+|x|)^{l+1}\na^{l+1}\bar\rho\|_{L^2}\|\f{\na^{N-l-1}v}{(1+|x|)^{l+1}}\|_{L^2}+\|(1+|x|)^{l}\na^{l}\bar\rho\|_{L^2}\|\f{\na^{N-l}v}{(1+|x|)^{l}}\|_{L^2} \Big)\\
		\leq&C \delta\|\na^N v\|_{L^2}.
	\end{split}
	\eeq
	In view of the decay estimate \eqref{n1h1}, it follows directly
	\beq\label{LL3}
	\begin{split}
		N_{23}\leq C\sum_{0\leq l\leq N-1}\|\na^{l}v\|_{L^2}\|\na^{N-l}v\|_{L^2}
		\leq C (1+t)^{-1-\f N2}.
	\end{split}
	\eeq
	Applying the estimate \eqref{widef} of function $\wf$, we deduce that\beq\label{L4}
	\begin{split}
		N_{24}
		\leq&C\int|\wf(n+\bar\rho)||\na^{N}v|dx+C\sum_{1\leq l\leq N-2}\sum_{\gamma_1+\cdots+\gamma_{j}=l}\int|\na^{\gamma_{1}}\bar\rho|\cdots|\na^{\gamma_{j}}\bar\rho||{\na^{N-l}v}|dx\\
		&\quad+C\sum_{1\leq l\leq N-2}\sum_{\substack{\gamma_1+\cdots+\gamma_{i+j}=l\\\gamma_1+\cdots+\gamma_{i}=m\\1\leq m\leq l}}\int|\na^{\gamma_1}n|\cdots|\na^{\gamma_{i}}n||\na^{\gamma_{i+1}}\bar\rho|\cdots|\na^{\gamma_{i+j}}\bar\rho||{\na^{N-l}v}|dx\\
		\overset{def}{=}&N_{241}+N_{242}+N_{243}.
	\end{split}
    \eeq
	It follows from H{\"o}lder inequality that\beno
	\begin{split}
		N_{241}\leq&C\|(n+\bar\rho)\|_{L^2}\|\na^N v\|_{L^2}\leq C\delta\|\na^N v\|_{L^2}.
	\end{split}
    \eeno
	By Hardy inequality, it is easy to deduce\beno
	\begin{split}
		N_{242}\leq& C\sum_{1\leq l\leq N-2}\sum_{\gamma_1+\cdots+\gamma_{j}=l}\|(1+|x|)^{\gamma_{1}}\na^{\gamma_{1}}\bar\rho\|_{L^2}\|(1+|x|)^{\gamma_{2}}\na^{\gamma_{2}}\bar\rho\|_{L^\infty}\cdots\|(1+|x|)^{\gamma_{j}}\na^{\gamma_{j}}\bar\rho\|_{L^\infty}\|\f{\na^{N-l}v}{(1+|x|)^{l}}\|_{L^2}\\
		\leq& C\delta\|\na^N v\|_{L^2}.
	\end{split}
	\eeno
	Without loss of generality, we assume that $1\leq \gamma_1\leq\cdots\leq \gamma_{i}\leq N-2$. The fact $i\geq2$ implies $\gamma_{i-1}\leq N-3\leq N-2$. Thus, we can exploit Hardy inequality, Sobolev interpolation inequality \eqref{Sobolev} in Lemma \ref{inter} and the decay estimate \eqref{n1h1} to obtain\beno\begin{split}
		N_{243}\leq&C\sum_{1\leq l\leq N-2}\sum_{\substack{\gamma_1+\cdots+\gamma_{i+j}=l\\\gamma_1+\cdots+\gamma_{i}=m\\1\leq m\leq l}}\|\na^{\gamma_1}n\|_{L^\infty}\cdots\|\na^{\gamma_{i-1}}n\|_{L^\infty}\|\na^{\gamma_{i}}n\|_{L^2}\|(1+|x|)^{\gamma_{i+1}}\na^{\gamma_{i+1}}\bar\rho\|_{L^\infty}\\
		&\quad\cdots\|(1+|x|)^{\gamma_{i+j}}\na^{\gamma_{i+j}}\bar\rho\|_{L^\infty}\|\f{\na^{N-l}v}{(1+|x|)^{l-m}}\|_{L^2}\\
		\leq&C\delta\sum_{1\leq l\leq N-2}\sum_{\substack{\gamma_1+\cdots+\gamma_{i+j}=l\\\gamma_1+\cdots+\gamma_{i}=m\\1\leq m\leq l}}\|\na^{\gamma_1+1}n\|_{H^1}\cdots\|\na^{\gamma_{i-1}+1}n\|_{H^1}\|\na^{\gamma_i}n\|_{L^2}\|\na^{N-m}v\|_{L^2}\\
		\leq&C\delta\sum_{1\leq l\leq N-2}\sum_{\substack{\gamma_1+\cdots+\gamma_{i+j}=l\\\gamma_1+\cdots+\gamma_{i}=m\\1\leq m\leq l}}(1+t)^{-\f N2-\f{5i}{4}-\f74}
		\leq C (1+t)^{-\f32-\f N2}.
	\end{split}
	\eeno
	Substituting the estimates of $N_{241}$, $N_{242}$ and $N_{243}$ into \eqref{L4}, we arrive at
	\beq\label{L24}
	N_{24}\leq C (1+t)^{-\f32-\f N2}+C\delta\|\na^N v\|_{L^2}.
	\eeq
	One can deal with the term $N_{25}$ in the same manner of $N_{24}$.
    Then, it is easy to check that
    \beq\label{L5}
	N_{25}\leq C(1+t)^{-\f32-\f N2}+C\delta \|\na^N n\|_{L^2}.
	\eeq
	Finally, let us deal with $N_{26}$. By virtue of the estimate \eqref{h} and \eqref{h2} of $\h$, then we have\beq\label{l6}\begin{split}
		N_{26}\leq& C\int|n||\na^N \bar\rho|dx+ C\sum_{1\leq l\leq N-1}\sum_{\gamma_1+\cdots+\gamma_{j}=l}\int|n||\na^{\gamma_{1}}\bar\rho|\cdots|\na^{\gamma_{j}}\bar\rho||\na^{N-l}\bar\rho|dx\\
		&\quad+C\sum_{1\leq l\leq N-1}\sum_{\substack{\gamma_1+\cdots+\gamma_{i+j}=l\\\gamma_1+\cdots+\gamma_{i}=m\\1\leq m\leq l}}\int|\na^{\gamma_1}n|\cdots|{\na^{\gamma_i}n}||\na^{\gamma_{i+1}}\bar\rho|\cdots|\na^{\gamma_{i+j}}\bar\rho||\na^{N-l} \bar\rho|dx\\
		\overset{def}{=}& N_{261}+N_{262}+N_{263}.
	\end{split}
	\eeq
	It follows from Hardy inequality that\beq\label{N261-2}
	\begin{split}
		N_{261}+N_{262}\leq& C\|\f{n}{(1+|x|)^N}\|_{L^2}\|(1+|x|)^N\na^N \bar\rho\|_{L^2}+C\sum_{1\leq l\leq N-1}\sum_{\gamma_1+\cdots+\gamma_{j}=l}\|\f{n}{(1+|x|)^{N}}\|_{L^2}\\
		&\quad\times\|(1+|x|)^{\gamma_{1}}\na^{\gamma_{1}}\bar\rho\|_{L^\infty}\cdots\|(1+|x|)^{\gamma_{j}}\na^{\gamma_{j}}\bar\rho\|_{L^\infty}\|(1+|x|)^{N-l}\na^{N-l} \bar\rho\|_{L^2}\\
		\leq& C\delta\|\na^Nn\|_{L^2}.
	\end{split}
	\eeq
	For the term $N_{263}$, it is easy to check that\beq\label{N263}
	\begin{split}
	N_{263}\leq&C\sum_{1\leq l\leq N-1}\sum_{\substack{\gamma_1+\cdots+\gamma_{i+j}=l\\\gamma_1=m\\1\leq m\leq l}}\|\f{\na^{\gamma_1}n}{(1+|x|)^{N-\gamma_1}}\|_{L^2}\|(1+|x|)^{\gamma_{2}}\na^{\gamma_{2}}\bar\rho\|_{L^\infty}\cdots\|(1+|x|)^{\gamma_{1+j}}\na^{\gamma_{1+j}}\bar\rho\|_{L^\infty}\\
	&\quad\|(1+|x|)^{N-l}\na^{N-l} \bar\rho\|_{L^2}+C\sum_{1\leq l\leq N-1}\sum_{\substack{\gamma_1+\cdots+\gamma_{i+j}=l\\\gamma_1+\cdots+\gamma_{i}=m\\1\leq m\leq l,i\geq2}}\|\na^{\gamma_1}n\|_{L^\infty}\cdots\|\na^{\gamma_{i-1}}n\|_{L^\infty}\|\f{\na^{\gamma_i}n}{(1+|x|)^{N-m}}\|_{L^2}\\
	&\quad\times\|(1+|x|)^{\gamma_{i+1}}\na^{\gamma_{i+1}}\bar\rho\|_{L^\infty}\cdots\|(1+|x|)^{\gamma_{i+j}}\na^{\gamma_{i+j}}\bar\rho\|_{L^\infty}\|(1+|x|)^{N-l}\na^{N-l} \bar\rho\|_{L^2}\\
	\overset{def}{=}&N_{2631}+N_{2632}.
	\end{split}
	\eeq
	We employ Hardy inequality once again, to discover
	\beno
	\begin{split}
		N_{2631}
		\leq& C\d\|\na^Nn\|_{L^2}.
	\end{split}
	\eeno
	Without loss of generality, we assume that $1\leq \gamma_1\leq\cdots\leq \gamma_{i}\leq N-1$. The fact $i\geq2$ implies $\gamma_{i}\leq m-1\leq N-2$. For the term $N_{2632}$, by virtue of Hardy inequality, Sobolev interpolation inequality \ref{Sobolev} in Lemma \ref{inter} and the decay estimate \eqref{n1h1}, we deduce
	\beno
	\begin{split}
		N_{2632}
		\leq&C\delta\sum_{1\leq l\leq N-1}\sum_{\substack{\gamma_1+\cdots+\gamma_{i+j}=l\\\gamma_1+\cdots+\gamma_{i}=m\\1\leq m\leq l,i\geq2}}\|\na^{\gamma_1+1}n\|_{H^1}\cdots\|\na^{\gamma_{i-1}+1}n\|_{H^1}\|\na^{N-m+\gamma_i}n\|_{L^2}\\
		\leq&C\delta\sum_{1\leq l\leq N-1}\sum_{\substack{\gamma_1+\cdots+\gamma_{i+j}=l\\\gamma_1+\cdots+\gamma_{i}=m\\1\leq m\leq l,i\geq2}} (1+t)^{-\f N2+\f12-\f{5i}{4}}
		\leq C (1+t)^{-2-\f N2}.
	\end{split}
	\eeno
	Substituting estimates of terms $N_{2631}$ and $N_{2632}$ into \eqref{N263}, we find\beno
	N_{263}\leq C\d\|\na^Nn\|_{L^2}+C(1+t)^{-2-\f N2}.
	\eeno
	Inserting \eqref{N261-2} and \eqref{N263} into \eqref{l6}, we get immediately\beq\label{L6}
	N_{26}\leq C\delta \|\na^N n\|_{L^2}+C(1+t)^{-2-\f N2}.
	\eeq
	We then conclude from \eqref{L}-\eqref{LL3}, \eqref{L24}, \eqref{L5} and \eqref{L6} that\beno
	\begin{split}
		\||\xi|^{N-1}\widehat F\|_{L^\infty(|\xi|\leq \eta)}\leq&C\delta\|\na^N (n,v)\|_{L^2}+ (1+t)^{-1-\f N2},
	\end{split}
	\eeno
	which, together with the definition of term $N_2$ and
   the estimate in Lemma \ref{tt2}, yields directly
   \beq\label{kF1}
	\begin{split}
		N_2
		\leq&C \int_{\f t2}^t(1+t-\tau)^{-\f54}\big(\delta\|\na^N (n,v)\|_{L^2}+ (1+\tau)^{-1-\f N2}\big)d\tau\\
		\leq&C\delta\sup_{0\leq\tau\leq t}\|\na^N (n,v)\|_{L^2}\int_{\f t2}^t(1+t-\tau)^{-\f54}d\tau+C(1+t)^{-1-\f N2}\\
		\leq&C\delta\sup_{0\leq\tau\leq t}\|\na^N (n,v)\|_{L^2}+C(1+t)^{-1-\f N2}.
	\end{split}\eeq
	Substituting \eqref{F1} and \eqref{kF1} into \eqref{nvl}, it holds true
    \beq\label{Unon}
	\int_0^t\|\na^{N}[S(t-\tau)F(U(\tau))]^l\|_{L^2}d\tau
    \leq C\delta\sup_{0\leq\tau\leq t}\|\na^N (n,v)\|_{L^2}+C(1+t)^{-\f34-\f N2}.
	\eeq
	Inserting \eqref{U0es} and \eqref{Unon} into \eqref{nvexpress}, one obtains immediately that
	\beno
	\|\na^N(n^l,v^l)\|_{L^2}\leq C \delta \sup_{0\leq s\leq t}\|\na^N(n,v)\|_{L^2}+C(1+t)^{-\f34-\f N2}.
	\eeno
	This completes the proof of this lemma.
\end{proof}

Finally, we focus on establishing optimal decay rate for the $N-th$ order spatial derivative of solution.
\begin{lemm}\label{optimaln}
	Under the assumption of Theorem \ref{them3}, we have
	\beq\label{n1h2}
	\|\na^N(n,v)(t)\|_{L^2}\leq C (1+t)^{-\f34-\f N2},
	\eeq
	where $C$ is a positive constant independent of $t$.
\end{lemm}
\begin{proof}
	We may rewrite the estimate \eqref{en6} in Lemma \ref{highfrequency} as\beq\label{ddt}
	\f{d}{dt}\widetilde{\mathcal E}^N(t)+\|\na^{N}v^h\|_{L^2}^2+\eta_3\|\na^{N}n^h\|_{L^2}^2
	\leq  C_4\|\na^{N}(n,v)^l\|_{L^2}^2+C(1+t)^{-3-N}.
	\eeq
	where the energy $\widetilde{\mathcal E}^N(t)$ is defined by
	\[\widetilde{\mathcal E}^N(t)\overset{def}{=}\|\na^{N}(n,v)\|_{L^2}^2-\eta_3\int_{|\xi|\geq\eta}\widehat{\na^{N-1}v}\cdot \overline{\widehat{\na^{N}n}}d\xi.\]
	With the help of Young inequality, by choosing $\eta_3$ small enough, one obtains the equivalent relation\beq\label{endj}
	c_5\|\na^{N}(n,v)\|_{L^2}^2\leq\widetilde{\mathcal E}^N(t)\leq c_6 \|\na^{N}(n,v)\|_{L^2}^2,
	\eeq
	where the constants $c_5$ and $c_6$ are independent of time.
	Then adding on both sides of \eqref{ddt} by $\|\na^{N}(n^l,v^l)\|_{L^2}^2$
     and applying the estimate \eqref{lowfre} in Lemma \ref{lowfrequency}, we find
     \beno
	\begin{split}
		\f{d}{dt}\widetilde{\mathcal E}^N(t)+\|\na^{N}(n,v)\|_{L^2}^2
		\leq ( C_4+1)\|\na^{N}(n^l,v^l)\|_{L^2}^2+C(1+t)^{-3-N}\leq C\delta \sup_{0\leq \tau\leq t}\|\na^N(n,v)\|_{L^2}^2+C(1+t)^{-\f32-N}.
	\end{split}
	\eeno
   By virtue of the equivalent relation \eqref{endj}, we have\beq\label{en7}
	\begin{split}
			\f{d}{dt}\widetilde{\mathcal E}^N(t)+\widetilde{\mathcal E}^N(t)
		\leq C\delta \sup_{0\leq \tau\leq t}\|\na^N(n,v)\|_{L^2}^2+C(1+t)^{-\f32-N},
	\end{split}
	\eeq
	which, using Gronwall inequality, gives immediately\beq\label{estimateE}
	\begin{split}
		\widetilde{\mathcal E}^N(t)\leq e^{-t} \widetilde{\mathcal E}^N(0)+C\delta\sup_{0\leq \tau\leq t}\|\na^N(n,v)\|_{L^2}^2\int_0^te^{\tau-t}d\tau+C\int_0^te^{\tau-t}(1+\tau)^{-\f32-N}d\tau.
	\end{split}
	\eeq
	It is easy to deduce that
    $$
	\int_0^te^{\tau-t}d\tau\leq C
    $$
    and
    $$\int_0^te^{\tau-t}(1+\tau)^{-\f32-N}d\tau\leq C(1+t)^{-\f32-N}.$$
	From the equivalent relation \eqref{endj} and \eqref{estimateE}, it holds
	\beno
	\begin{split}
		\sup_{0\leq \tau\leq t}\|\na^N(n,v)(\tau)\|_{L^2}^2\leq Ce^{-t}\|\na^N(n_0,v_0)\|_{L^2}^2+C\delta\sup_{0\leq \tau\leq t}\|\na^N(n,v)\|_{L^2}^2+ C(1+t)^{-\f32-N}.
	\end{split}
	\eeno
	Applying the smallness of $\delta$, one arrives at
    \beno
		\sup_{0\leq \tau\leq t}\|\na^N(n,v)(\tau)\|_{L^2}^2\leq C(1+t)^{-\f32-N}.
	\eeno
	Therefore, we cpmlete the proof of this lemma.
\end{proof}

\underline{\noindent\textbf{The Proof of Theorem  \ref{them3}.}}
Combining the estimate \eqref{n1h1} in Lemma \ref{N-1decay} with
estimate \eqref{n1h2} in Lemma \ref{optimaln}, then we can obtain the
decay rate \eqref{kdecay} in Theorem \ref{them3}.
  Consequently, we finish the proof of Theorem \ref{them3}.

\subsection{Lower bound of decay rate}\label{lower}
In this subsection, the content of our analysis is to establish the lower bound of decay rate for the global solution and its spatial derivatives of the initial value problem \eqref{ns5}.
In order to achieve this target, we need to analyze the linearized system \eqref{linear}.
We obtain the following proposition immediately, whose proof is similar to \cite{chen2021} and standard, so we omit here.
\begin{prop}\label{lamma-lower}
Let $U_0=(n_0,v_0)\in L^1(\R^3)\cap H^l(\R^3)$ with $l\geq3$,
assume that $M_n\overset{def}{=} \int_{\R^3}n_0(x) d x$
and $M_v\overset{def}{=}\int_{\R^3}v_0(x) d x$
are at least one nonzero.
Then there exists a positive constant $\widetilde c$ independent of time
such that for any large enough $t$, the global solution $(\widetilde n,\widetilde v)$
of the linearized system \eqref{linear} satisfies
\beq\label{linearnudecay}
\min\{\|\pa_{x}^{k} \widetilde n(t)\|_{L^2(\R^3)},
      \|\pa_{x}^{k} \widetilde v(t)\|_{L^2(\R^3)}\}
\geq \widetilde c(1+t)^{-\f34-\f{k}{2}},\quad\text{for}~~ 0\leq k\leq l.
\eeq
Here $\widetilde c$ is a positive constant depending only on $M_n$ and $M_v$.
\end{prop}

Define the difference $(n_{\d},v_{\d})\overset{def}{=}(n-\widetilde n,v-\widetilde v)$,
then the quantity $(n_{\d},v_{\d})$ satisfies the following system:
\beq\label{ns7}
\left\{\begin{array}{lr}
	n_{\d t} +\gamma\dive v_{\d}=\widetilde S_1,\quad (t,x)\in \mathbb{R}^{+}\times \mathbb{R}^3,\\
	v_{\d t}+\gamma\na n_{\d}-\mu_1\tri v_{\d}-\mu_2\na\dive v_{\d} =\widetilde  S_2,\quad (t,x)\in \mathbb{R}^{+}\times \mathbb{R}^3,\\
	(n_\d,v_\d)|_{t=0}=(0,0).
\end{array}\right.
\eeq
By virtue of the formula \eqref{Uexpress}, the solution of system \eqref{ns7}
can be represented as\beno
(n_\d,v_\d)^{t}=\int_0^tS(t-\tau)F(U(\tau))d\tau.
\eeno
Next, we aim to establish the upper bounds of decay rates for the difference $(n_\d,v_\d)$ and its first order spatial derivative.
\begin{lemm}
	Under the assumptions of Theorem \ref{them4}, assume $(n_\d,v_\d)$ be the smooth solution of the initial value problem \eqref{ns7}. Then, it holds on for $t\geq0$,
	\beq\label{nuddecay}
	\|\na^k(n_\d,v_\d)(t)\|_{L^2}\leq \widetilde C\d(1+t)^{-\f34-\f k2},\quad \text{for}~~k=0,1,
	\eeq
	where $\widetilde C$ is a constant independent of time.
\end{lemm}
\begin{proof}
	By Duhamel's principle, it holds for $k\geq0$,\beq\label{nvd}
	\|\na^k(n_\d,v_\d)(t)\|_{L^2}\leq \int_0^t(1+t-\tau)^{-\f34-\f k2}\big(\|(\widetilde {S}_1,\widetilde {S}_2)(\tau)\|_{L^1}+\|\na^k(\widetilde {S}_1,\widetilde {S}_2)(\tau)\|_{L^2}\big)d\tau.
	\eeq
	It then follows from decay estimates \eqref{SS1} and \eqref{SS2} that\beq\label{s1s2}
	\|(\widetilde {S}_1,\widetilde {S}_2)\|_{L^1}\leq C\d(1+t)^{-\f54}.
	\eeq
In view of Sobolev and Hardy inequality, it is easy to deduce that
\begin{equation}\label{ds1}
\|\widetilde {S}_1\|_{L^2}
\leq C\|\bar\rho\|_{L^\infty}\|\na n\|_{L^2}
     +C\|(1+|x|)\na\bar\rho\|_{L^\infty}\|\f{n}{1+|x|}\|_{L^2}
\leq C\delta\|\na n\|_{L^2}\leq C\d(1+t)^{-\f54},
\end{equation}
and
\begin{equation}\label{ds2}
\begin{aligned}
\|\widetilde {S}_2\|_{L^2}
\leq& C\|v\|_{L^\infty}\|\na v\|_{L^2}
      +\|(n+\bar\rho)\|_{L^\infty}\|(\na^2 v,\na n)\|_{L^2}
      +\|\f{n}{1+|x|}\|_{L^\infty}\|(1+|x|)\na \bar\rho\|_{L^2}\\
\leq& C\delta\|(\na n,\na v,\na^2v)\|_{L^2}\leq C\delta(1+t)^{-\f54}.
\end{aligned}
\end{equation}
Then, substituting the decay estimates \eqref{s1s2}, \eqref{ds1}
and \eqref{ds2} into \eqref{nvd} with $k=0$, it holds
\beno
\begin{split}
\|(n_\d,v_\d)(t)\|_{L^2}\leq& \int_0^t(1+t-\tau)^{-\f34}\big(\|(\widetilde {S}_1,\widetilde {S}_2)(\tau)\|_{L^1}+\|(\widetilde {S}_1,\widetilde {S}_2)(\tau)\|_{L^2}\big)d\tau\\
		\leq& C\d \int_0^t(1+t-\tau)^{-\f34}(1+\tau)^{-\f54}d\tau\\
		\leq& C\d(1+t)^{-\f34},
\end{split}
\eeno
where we have used the Lemma \ref{tt2} in the last inequality.
Similar to \eqref{ds1} and \eqref{ds2}, one arrives at
\beno
	\begin{split}
		\|\na\widetilde {S}_1\|_{L^2}\leq& C\|\bar\rho\|_{L^\infty}\|\na^2 n\|_{L^2}+C\|(1+|x|)\na\bar\rho\|_{L^\infty}\|\f{\na n}{1+|x|}\|_{L^2}+C\|(1+|x|)^2\na^2\bar\rho\|_{L^\infty}\|\f{n}{(1+|x|)^2}\|_{L^2}\\
		\leq& C\delta\|\na^2 n\|_{L^2}\leq C\d(1+t)^{-\f74},\\
		\|\na\widetilde {S}_2\|_{L^2}\leq& C\|v\|_{L^\infty}\|\na^2 v\|_{L^2}+\|\na v\|_{L^3}\|\na v\|_{L^6}+\|n+\bar\rho\|_{L^\infty}\|(\na^3 v,\na^2 n)\|_{L^2}+\|\na n\|_{L^3}\|(\na^2v,\na n)\|_{L^6}\\
		&+\|(1+|x|)\na\bar\rho\|_{L^\infty}\|(\f{\na^2v}{1+|x|},\f{\na n}{1+|x|})\|_{L^2}+\|(1+|x|)\na \bar\rho\|_{L^\infty}^2\|\f{ n}{(1+|x|)^2}\|_{L^2}\\
		\leq& C\delta\|(\na^2 n,\na^2 v,\na^3v)\|_{L^2}\leq C\delta(1+t)^{-\f74},
	\end{split}
	\eeno
which, together with \eqref{nvd}, \eqref{s1s2} and Lemma \ref{tt2},
yields directly
\beno
	\begin{split}
		\|\na(n_\d,v_\d)(t)\|_{L^2}\leq& \int_0^t(1+t-\tau)^{-\f54}\big(\|(\widetilde {S}_1,\widetilde {S}_2)(\tau)\|_{L^1}+\|\na(\widetilde {S}_1,\widetilde {S}_2)(\tau)\|_{L^2}\big)d\tau\\
		\leq&C\d\int_0^t(1+t-\tau)^{-\f54}(1+\tau)^{-\f54}d\tau\leq C\d(1+t)^{-\f54}.
	\end{split}
	\eeno
	Therefore, the proof of this lemma is completed.
\end{proof}
Finally, we establish the lower bound of decay rate for the global solution and its spatial derivatives of compressible Navier-Stokes equations with potential force.
Then, the estimate \eqref{lower-estimate} in Lemma \ref{lemmalower} below will
yield the optimal decay estimate in Theorem \ref{them4}.
\begin{lemm}\label{lemmalower}
Under the assumptions of Theorem \ref{them4}, then for any large enough $t$, we have
\beq\label{lower-estimate}
\min\{\|\na^kn(t)\|_{L^2},\|\na^kv(t)\|_{L^2} \}
\geq c_1(1+t)^{-\f34-\f k2},\quad \text{for}~~ 0\leq k\leq N,
\eeq
where $c_1$ is a positive constant independent of time.
\end{lemm}
\begin{proof}
By virtue of the definition of $n_\d$, it holds true
\beno
\|\na^k\widetilde n\|_{L^2}\leq \|\na^kn\|_{L^2}+\|\na^kn_\d\|_{L^2},
\eeno
which, together with the  lower bound decay \eqref{linearnudecay} and upper bound decay \eqref{nuddecay}, yields directly
\beq\label{lowlow-01}
\|\na^kn\|_{L^2}\geq \|\na^k\widetilde n\|_{L^2}-\|\na^kn_\d\|_{L^2}
\geq  \widetilde c(1+t)^{-\f34-\f k2}-\widetilde C\d (1+t)^{-\f34-\f k2},
\eeq
where $k=0,1$. It is worth noting that the small constant $\delta$ is used
to control the upper bound of initial data in $L^2-$norm instead of $L^1$ one
(see \eqref{phik} and \eqref{initial-H2}).
From the estimate \eqref{linearnudecay} in Lemma \ref{lamma-lower},
the constant $\widetilde c$ in \eqref{lowlow-01} only depends on the quantities $M_n$ and $M_v$.
Then, we can choose $\delta$ small enough such that
$\widetilde C \d\leq\f{1}{2} \widetilde c$, and hence,
it follows from \eqref{lowlow-01} that
\beq\label{lowlow}
\|\na^k n\|_{L^2} \geq  \frac12\widetilde c(1+t)^{-\f34-\f k2}, \quad k=0,1.
\eeq
By virtue of the Sobolev interpolation inequality in Lemma \ref{inter},
it holds true for $k\geq2$
\beno
\|\na n\|_{L^2}\leq C\|n\|_{L^2}^{1-\f1k}\|\na^kn\|_{L^2}^{\f1k},
\eeno
which, together with the lower bound decay \eqref{lowlow}
and upper bound decay \eqref{kdecay}, implies directly
\beq\label{highlow}
\|\na^kn\|_{L^2}
\geq C\|\na n\|_{L^2}^k\|n\|_{L^2}^{-(k-1)}
\geq C(1+t)^{-\f{5k}{4}}(1+t)^{\f{3(k-1)}{4}}
\geq c_1(1+t)^{-\f34-\f k2},
\eeq
for all $k \geq 2$.
In the same manner, it is easy to deduce that
\beq\label{vhighlow}
\|\na^kv\|_{L^2}\geq c_1(1+t)^{-\f34-\f k2},\quad \text{for}~~k \geq 0.
\eeq
Then, the combination of estimates \eqref{lowlow}, \eqref{highlow}
and \eqref{vhighlow} yields the estimate \eqref{lower-estimate}.
Therefore, we complete the proof of this lemma.
\end{proof}






\section*{Acknowledgements}

This research was partially supported by NNSF of China(11801586, 11971496, 12026244), Guangzhou Science and technology project of China(202102020769), Natural Science Foundation of Guangdong Province of China (2020A1515110942),
National Key Research and Development Program of China(2020YFA0712500).

\phantomsection

\end{document}